\documentclass[11pt,a4paper]{amsart}

\usepackage[T1]{fontenc}
\usepackage[utf8]{inputenc}
\usepackage[english]{babel}
\usepackage{amsmath,amssymb,amsthm,mathtools}
\usepackage[final,protrusion=true,expansion=false]{microtype}
\usepackage[margin=1in]{geometry}
\usepackage{hyperref}
\hypersetup{colorlinks=true,linkcolor=blue,citecolor=blue}

\theoremstyle{plain}
\newtheorem{theorem}{Theorem}[section]
\newtheorem{proposition}[theorem]{Proposition}
\newtheorem{lemma}[theorem]{Lemma}
\newtheorem{corollary}[theorem]{Corollary}

\newtheorem{mainthm}{Theorem}

\theoremstyle{definition}
\newtheorem{definition}[theorem]{Definition}
\newtheorem{remark}[theorem]{Remark}

\DeclareMathOperator{\ind}{ind}
\DeclareMathOperator{\nul}{nul}

\DeclareMathOperator{\Span}{span}

\DeclareMathOperator{\dist}{dist}
\newcommand{\R}{\mathbb{R}}
\newcommand{\Z}{\mathbb{Z}}

\newcommand{\HH}{\mathbb{H}}
\newcommand{\II}{\mathrm{II}}
\newcommand{\Sf}{\mathcal{S}}
\newcommand{\Lf}{L_{\Sigma}}

\title[Analytic local resolution of Medvedev's Morse index conjecture in $\HH^3$]%
{Analytic local resolution of Medvedev's Morse index conjecture for the critical spherical catenoid in $\HH^3$}

\author{Alexander Pigazzini}

\date{}



\begin{document}

\begin{abstract}

Let $\Sigma_a\subset B^3(r(a))\subset\HH^3$ ($a>1/2$) be the critical spherical catenoid of the Mori family, a free boundary minimal surface in the geodesic ball. The Medvedev conjecture \cite[Remark 5.6]{Medvedev2023} asserts that $\ind(\Sigma_a)=4$ for every $a>1/2$. We consider here the \emph{strong form} of this conjecture, namely $\ind(\Sigma_a)=4$ and $\nul(\Sigma_a)=2$. The nullity condition $\nul(\Sigma_a)=2$ combines the mode-$|k|=1$ result $\nul_R(\Sigma_a)\big|_{|k|=1}=2$ of \cite[Cor. 4.4]{Pigazzini} with the additional requirement of vanishing kernel in modes $|k|=0$ and $|k|\geq 2$; this latter requirement, not addressed in \cite{Pigazzini}, is part of the analytic content of the present paper and is established in the local regime $a\in(1/2,1/2+\delta_0)$.

The principal quantitative result of the paper is the analytic local resolution of the strong Medvedev conjecture: there exists $\delta_0>0$ such that $\ind(\Sigma_a)=4$ and $\nul(\Sigma_a)=2$ for every $a\in(1/2,1/2+\delta_0)$. This follows from an explicit closed form for the leading asymptotic coefficient of $H(a):=\sinh r(a)/K(a)$ as $a\to(1/2)^+$,
\[
H(a)=\sigma_*\cosh\sigma_*+C_0\,(a-\tfrac{1}{2})+O\bigl((a-\tfrac{1}{2})^2\bigr),\qquad C_0=\frac{\sigma_*\cosh\sigma_*\,(\sinh^2\sigma_*-1)(3\sinh^2\sigma_*-2)}{12\sinh^2\sigma_*},
\]
where $\sigma_*>0$ is the unique positive root of $\sigma=\coth\sigma$, together with the analytic proof that $C_0>0$ via $\sigma_*>\log(1+\sqrt{2})$.

The route to the local resolution proceeds through three analytic reductions of independent interest: (i) the strong Medvedev conjecture is shown to be equivalent to the conjunction of two spectral conditions, $\mu_0^{\mathrm{even}}(2)>0$ (condition (E)) and $\mu_2(0)>0$ together with spectral non-degeneracy in mode $0$ (condition (F)); (ii) the eigenvalue inequality $\mu_2(0)>0$ is reduced, via a Sturm shooting-count argument, to the geometric positivity $\phi_a>0$ of the parametric Jacobi field on the principal branch; (iii) the positivity $\phi_a>0$ is in turn reduced, under the strict geometric inequality $B(s_0(a))^2>2K(a)^2$ (condition (G)), to the one-dimensional scalar differential inequality $H'(a)>0$, via a constant Wronskian identity and a Sturm separation argument.

Auxiliary results include: a Picone identity with base $f_*$ yielding the unconditional closure of the odd radial sector in modes $|k|\geq 2$; a second Picone identity with base $B$ proving (E) unconditionally on $(1/2,1]$ and, via explicit Hardy estimates, on $(1/2,A_*]$ for some $A_*>1$; the analytic closure of (G) on $(1/2,1]$ via strict concavity of a transcendental function; and a short alternative proof of the lower bound $\ind(\Sigma_a)\geq 4$ via the four Lorentz ambient coordinates as test functions.
\end{abstract}

\maketitle

\section{Introduction and main result}\label{sec:intro}

In the last decade, research on free boundary minimal surfaces in space forms has attracted much interest, see the constructions by Folha–Pacard–Zolotareva \cite{FolhaPacardZolotareva} in the Euclidean unit ball, the eigenvalue-theoretic approach of Lima–Menezes \cite{LimaMenezes} in spherical caps, the recent free boundary CMC constructions of Cerezo–Fernández–Mira \cite{Cerezo} in geodesic balls of $\mathbb S^3$ and $\mathbb H^3$, and the compactness and index estimates of Ambrozio–Carlotto–Sharp \cite{AmbrozioCarlottoSharp2018a, AmbrozioCarlottoSharp2018b} and Sargent \cite{Sargent}. The connection between FBMS in Euclidean balls and Steklov eigenvalue problems was initiated by Fraser–Schoen \cite{FraserSchoen2011, FraserSchoen2015}. For the Euclidean critical catenoid, the Morse index has been computed to equal 4 independently by Devyver \cite{Devyver}, Smith–Zhou \cite{SmithZhou}, and Tran \cite{Tran2016}; the present paper addresses the analogous question in the hyperbolic ambient.

\subsection{General framework}
\begin{remark}[Terminology]\label{rem:terminology}
The object studied here is the rotational minimal family $\Phi_a$, $a>1/2$, of equation (1.2) in \cite{Medvedev2023}, originally constructed by Mori \cite{Mori1981}. In the classification of do Carmo--Dajczer \cite{DCD83}, rotation hypersurfaces are classified according to the nature of their parallels (\emph{spherical} if the parallels are spheres, \emph{hyperbolic} if they are hyperbolic spaces, \emph{parabolic} if they are horospheres); the present family is of \emph{spherical type} (the case $\delta=+1$ of \cite[eq. (3.21)]{DCD83}, with fixed plane $P^2$ Lorentzian and circular parallels). Accordingly, Mori \cite{Mori1981} and Medvedev \cite{Medvedev2023} call it the \emph{spherical catenoid}, and the truncated free boundary piece in $B^3(r)\subset\HH^3$ is the \emph{critical spherical catenoid}.
\end{remark}

Let $\HH^3$ denote three-dimensional hyperbolic space of curvature $-1$, realized as the connected component $\{x\in\R^4_1:\langle x,x\rangle_L=-1,\;x_0>0\}$ of the hyperboloid in Lorentz space $(\R^4,\langle\,,\,\rangle_L)$ with inner product $\langle x,y\rangle_L= - x_0y_0+x_1y_1+x_2y_2+x_3y_3$. We fix the pole $p_0=(1,0,0,0)$ and denote by $r(p)=\dist_{\HH^3}(p_0,p)$ the geodesic distance, so that $\cosh r(p)=-\langle p,p_0\rangle_L$. For each $\rho>0$ we define the geodesic ball $B^3(\rho)=\{p\in\HH^3:r(p)<\rho\}$, whose boundary $\partial B^3(\rho)$ is the totally umbilical sphere of mean curvature $\coth\rho$.

The Mori family \cite{Mori1981} is the one-parameter family $\{\Sigma_a\}_{a>1/2}$ of rotationally symmetric minimal surfaces in $\HH^3$, given in coordinates $(s,\theta)\in\R\times[0,2\pi)$ by the immersion
\begin{equation}\label{eq:phi-immersion}
\Phi_a(s,\theta)=\bigl(A(s)\cosh\varphi(s),\,A(s)\sinh\varphi(s),\,B(s)\cos\theta,\,B(s)\sin\theta\bigr),
\end{equation}
where
\begin{equation}\label{eq:ABK}
A(s)^2=a\cosh(2s)+\tfrac{1}{2},\quad B(s)^2=a\cosh(2s)-\tfrac{1}{2},\quad K=\sqrt{a^2-\tfrac{1}{4}},
\end{equation}
and the angle $\varphi(s)$ is determined by $\varphi(0)=0$ and $\varphi'(s)=K/(A(s)^2 B(s))$, a condition that corresponds to minimality (cf. \cite{Mori1981,Pigazzini}). For $a>1/2$, the surface $\Sigma_a$ is a \emph{critical spherical catenoid} if there exists $s_0(a)>0$ such that $\Sigma_a$ restricted to $|s|\leq s_0$ is an FBMS in $B^3(r(a))$, where $r(a):=r(\Phi_a(s_0,0))$, and the free boundary condition fixes
\begin{equation}\label{eq:fbc}
\tanh\varphi(s_0)=\frac{B(s_0)\,K}{a\sinh(2s_0)}\qquad\text{(FBMS condition)}.
\end{equation}

The \emph{Jacobi operator} of $\Sigma_a$ is $\Lf=\Delta_g+|\II|^2-2$, where $\Delta_g$ is the induced metric Laplacian and $-2$ is twice the ambient sectional curvature (see \cite[Ch. 8]{Colding-Minicozzi}). The Robin quadratic form associated with an FBMS in $B^3(\rho)$, whose boundary has mean curvature $\coth\rho$, is
\begin{equation}\label{eq:S-form}
\Sf(u,u)=\int_{\Sigma}\bigl(|\nabla u|_g^2-(|\II|^2-2)u^2\bigr)\,dA-\coth r(a)\int_{\partial\Sigma}u^2\,dL.
\end{equation}
The \emph{Robin Morse index} $\ind_R(\Sigma_a)$ is the dimension of the maximal subspace of $H^1(\Sigma_a)$ on which $\Sf$ is negative definite, and the \emph{Robin nullity} $\nul_R(\Sigma_a)$ is the dimension of $\ker\Sf$, which coincides with the kernel of $\Lf$ under the Robin condition
\begin{equation}\label{eq:robin}
\partial_\eta u=\coth r(a)\,u\quad\text{on }\partial\Sigma_a,
\end{equation}
where $\eta$ is the outward conormal.

\begin{remark}[Equivalence with Medvedev's Morse index]\label{rem:robin-morse-equivalence}
The Robin Morse index $\ind_R(\Sigma_a)$ used throughout this paper coincides with the Morse index $\ind(\Sigma_a)$ defined by Medvedev \cite[Definition 5.1]{Medvedev2023}: the maximal dimension of a subspace of $C^\infty(\Sigma_a)$ on which the second-variation quadratic form \eqref{eq:S-form} is negative definite. By integration by parts,
\[
\Sf(u,u)=-\int_{\Sigma_a}u\,\Lf u\,dA+\int_{\partial\Sigma_a}u\,(\partial_\eta u-\coth r(a)\,u)\,dL,
\]
and the boundary term vanishes precisely on functions satisfying the Robin condition \eqref{eq:robin}. Consequently, on functions satisfying \eqref{eq:robin} one has $\Sf(u,u)=-\int_{\Sigma_a}u\,\Lf u\,dA$, so that the negative directions of $\Sf$ are exactly the Robin eigenfunctions of $\Lf$ with positive eigenvalue; equivalently, $\ind_R(\Sigma_a)$ is the number of negative eigenvalues of the second-variation form $\Sf$ (i.e.\ of the operator $-\Lf$ under the Robin boundary condition). Being the dimension of the negative subspace of the geometric form $\Sf$, the index is independent of the sign convention adopted for $\Delta_g$ (Medvedev: $\Delta_g=-\mathrm{div}_g\nabla$; here: $\Delta_g=+\mathrm{div}_g\nabla$): the convention only flips the signs of the eigenvalues of the Jacobi operator --- hence whether the index is read off the positive or the negative part of its spectrum, not the dimension of the negative subspace of $\Sf$. Throughout this paper we write $\ind_R(\Sigma_a)$ (respectively $\nul_R(\Sigma_a)$) in proofs and computations, where the Robin formulation is technically convenient, and $\ind(\Sigma_a)$ (respectively $\nul(\Sigma_a)$) in statements concerning Medvedev's conjecture; the two notations refer to the same geometric invariant. In particular, the Medvedev conjecture $\ind(\Sigma_a)=4$ and its strong form $\nul(\Sigma_a)=2$ are equivalent to $\ind_R(\Sigma_a)=4$ and $\nul_R(\Sigma_a)=2$, respectively, and the analytic local resolution (Corollary \ref{cor:eureka-local-closure}) resolves Medvedev's Morse index conjecture for $a\in(1/2,1/2+\delta_0)$ in its original formulation.
\end{remark}

\subsection{Known results}
In \cite{Pigazzini} the author establishes the following three properties of $\Sigma_a$:
\begin{enumerate}
\item[(P1)] $\ind_R(\Sigma_a)\big|_{|k|=1}=\nul_R(\Sigma_a)\big|_{|k|=1}=2$, with explicit Robin Jacobi field $f_*(s)=\partial_s\Phi^0_a(s,0)=\sinh r(s)\,r'(s)$;
\item[(P2)] $r(a)=\tfrac{3}{2}\log a+d_\infty+o(1)$ as $a\to\infty$, with $d_\infty=\log\bigl(\sqrt{2}\,\Gamma(1/4)^2/\pi^{3/2}\bigr)$;
\item[(P3)] $r(a)=c_*\sqrt{a-1/2}\,(1+o(1))$ as $a\to(1/2)^+$, with $c_*=\sigma_*\cosh\sigma_*$ and $\sigma_*=\coth\sigma_*$.
\end{enumerate}

\begin{remark}[The Euclidean limit, and what spectral continuity cannot see]\label{rem:euclid}
The equation $\sigma=\coth\sigma$ defining $\sigma_*$ is equivalent to $\sigma\tanh\sigma=1$, the equation of the
\emph{Euclidean} critical catenoid in the unit ball. This is no coincidence: under the rescaling $s=\rho\xi$,
$\rho=\sqrt{a-1/2}$, the radial problems degenerate onto those of the Euclidean critical catenoid in $\mathbb{B}^3$
(whose ambient radius is recovered as $\sqrt{\cosh^2\sigma_*+\sigma_*^2}=\sigma_*\cosh\sigma_*=c_*$). One may therefore ask whether the qualitative part of Corollary~\ref{cor:eureka-local-closure} could be obtained softly, from eigenvalue continuity together with the spectral non-degeneracy of the Euclidean limit in the modes
$|k|\neq 1$ \cite{Devyver,SmithZhou,Tran2016}. Three caveats separate such a route from an exercise. First, the
rescaling is singular --- the interval collapses ($s_0\sim\rho\sinh\sigma_*$), the Robin coefficient blows up ($\coth r\sim(c_*\rho)^{-1}$), the potential blows up ($|\II|^2\sim 2\rho^{-2}(1+\xi^2)^{-2}$), and the endpoint
moves with the FBC --- so the required convergence statement, not available in the literature, demands precisely the uniform control established in Proposition~\ref{prop:asymptotic-BvsK} and Lemma~\ref{lem:H-analytic-rho}; moreover, mode-by-mode non-degeneracy does not control all Fourier modes at once, so any such argument must in any case be supplemented by the uniform separation of Theorem~\ref{thm:sturm-strict} and the unconditional odd closure of Theorem~\ref{thm:odd-closure}. Second, and structurally, the Lorentzian ambient removes the Euclidean anchor of the mode-$0$ analysis: in $\R^3$ the dilation field $x\cdot\nu$ is a Jacobi field vanishing identically on $\partial\Sigma$ for every FBMS in a ball, whereas in the hyperboloid model the position vector is timelike with $\langle X,\nu\rangle_L\equiv 0$ trivially, and $\HH^3$ admits no homotheties, so no ambient deformation plays this role. Its place is taken by the parametric field $\phi_a$, which is generated by no isometry and whose Robin defect $-r'(a)\,\kappa_s$ involves the unknown monotonicity of $r(a)$ \cite[Question 6.3]{Pigazzini}. Third, the limit is borderline for the sign: $\phi_a(s_0)\to 0$ as $a\to(1/2)^+$, and the two shooting configurations $(n_z,\delta_\partial)=(0,1)$ and $(1,0)$ of Remark~\ref{rem:converse} --- which decide the sign of $H'(a)$, hence the strict monotonicity $r'(a)>0$ and the validity of the scalar criterion of
Theorem~\ref{thm:scalar-reduction} --- yield the same eigenvalue count and are therefore indistinguishable by
eigenvalue continuity alone. The dichotomy is resolved only at second order, by the closed-form coefficient $C_0$
of Theorem~\ref{thm:eureka-closed-form}, whose positivity reduces to the single transcendental inequality
$\sinh\sigma_*>1$, i.e.\ $\sigma_*>\log(1+\sqrt{2})$ --- the same inequality governing $\rho_*>1$ in
Proposition~\ref{prop:asymptotic-BvsK} and the limiting margin $\cosh^2\sigma_*>2$ of condition (G) in
Proposition~\ref{prop:y-asymp-half}. In this precise sense the signature of the ambient concentrates the local problem into a single sign question, which the present approach answers in closed form, yielding in addition the
rate of detachment from the Euclidean limit and the local strict monotonicity $r'(a)>0$.
\end{remark}

The Medvedev conjecture (\cite[Remark 5.6]{Medvedev2023}, reported in \cite[\S 6.1]{Pigazzini}) concerns the Morse index of $\Sigma_a$ and states:
\begin{equation}\label{eq:medvedev}
\ind(\Sigma_a)=4\quad\text{for every }a>1/2.
\end{equation}
No assertion on the nullity of $\Sigma_a$ is made in \cite{Medvedev2023}. We consider in addition the \emph{strong form} of the conjecture,
\begin{equation}\label{eq:medvedev-strong}
\ind(\Sigma_a)=4\quad\text{and}\quad\nul(\Sigma_a)=2\quad\text{for every }a>1/2.
\end{equation}
By the Fourier mode decomposition (\S\ref{subsec:key-conditions}), the nullity condition $\nul_R(\Sigma_a)=2$ in \eqref{eq:medvedev-strong} is equivalent to the conjunction of:
\begin{itemize}
\item[(N1)] $\nul_R(\Sigma_a)\big|_{|k|=1}=2$, corresponding to the two-dimensional kernel generated by $f_*\cos\theta$ and $f_*\sin\theta$, which is established in \cite[Cor. 4.4]{Pigazzini};
\item[(N2)] $\nul_R(\Sigma_a)\big|_{|k|=0}=0$ and $\nul_R(\Sigma_a)\big|_{|k|\geq 2}=0$, i.e. the absence of any Robin Jacobi field in the remaining Fourier modes.
\end{itemize}
Condition (N1) is the only nullity contribution exhibited in \cite{Pigazzini}; the cited paper does not address (N2). The closure of (N2) is part of the analytic content of the present paper: the mode-$0$ part is handled by Theorems \ref{thm:no-kernel-even} and \ref{thm:no-kernel-odd} under the geometric conditions (G) and $r'(a)\neq 0$; the odd radial sector in modes $|k|\geq 2$ is closed unconditionally by Theorem \ref{thm:odd-closure-main} (Picone identity with base $f_*$); the even radial sector in mode $|k|=2$ is closed by Theorem \ref{thm:E-closure-a-leq-1} (Picone identity with base $B$, valid on $(1/2,1]$); and the strict Sturm comparison of Theorem \ref{thm:sturm-strict} then propagates the closure of the even sector from $|k|=2$ to every $|k|\geq 3$ (cf. Remark \ref{rem:odd-even-residual}). The strong form \eqref{eq:medvedev-strong} is therefore strictly stronger than \eqref{eq:medvedev}, and to the author's best knowledge it has not been established in the literature; the present paper proves it analytically in the local regime $a\in(1/2,1/2+\delta_0)$ (Corollary \ref{cor:eureka-local-closure}). In \cite[Corollary 5.9]{Medvedev2023} the lower bound $\ind(\Sigma_a)\geq 4$ is established. In this paper we provide an alternative, short and completely elementary proof of this bound (Theorem \ref{thm:lower-bound-4}), in addition to proving the new results stated in the abstract.

\subsection{Mode conventions}
We decompose $u\in H^1(\Sigma_a)$ into a Fourier series in $\theta$:
\begin{equation}
u(s,\theta)=\sum_{k\in\Z}u_k(s)e^{ik\theta},\quad\text{or equivalently, in real form,}\quad u_0(s)+\sum_{k\geq 1}\bigl(\alpha_k(s)\cos k\theta+\beta_k(s)\sin k\theta\bigr).
\end{equation}
We denote by $\Sf_k^{\mathrm{rad}}$ the quadratic form induced on the radial parts of mode $k$. Each sector $\Sf_k^{\mathrm{rad}}$ further decomposes according to the parity under $s\mapsto-s$ (\emph{even} and \emph{odd} radial sectors), since the radial operator preserves parity. We denote by $\mu_n(k)$ the $(n+1)$-th eigenvalue of $\Sf_k^{\mathrm{rad}}$ (in increasing order, counted with multiplicity), and by $\mu_n^{\mathrm{even}}(k)$, $\mu_n^{\mathrm{odd}}(k)$ the restrictions to the two parity sectors.

\subsection{Key conditions}\label{subsec:key-conditions}
Throughout the paper, three conditions on the parameter $a>1/2$ play a structural role and are referred to repeatedly. We collect their precise definitions here for clarity:
\begin{itemize}
\item[\textbf{(E)}] \emph{Spectral condition (even sector of mode $2$):} 
\[
\mu_0^{\mathrm{even}}(2)>0.
\]
This is the first eigenvalue of $\Sf_2^{\mathrm{rad}}$ restricted to the even-parity radial sector. By Theorem \ref{thm:reduction-main} below, (E) is one of the two spectral inequalities to which the strong Medvedev conjecture reduces.

\item[\textbf{(F)}] \emph{Spectral condition (mode $0$):}
\[
\mu_2(0)>0\quad\text{together with}\quad\mu_n(0)\neq 0\text{ for every }n\geq 0.
\]
The first inequality is the kernel-positivity in mode $0$; the second is the spectral non-degeneracy (no zero eigenvalues). By Theorem \ref{thm:reduction-main} below, (F) is the second spectral inequality to which the strong Medvedev conjecture reduces. A weakening of (F), denoted (F$'$), is the single condition $\mu_2(0)>0$ (without the non-degeneracy requirement), used in Section \ref{sec:F-prime-strategy}.

\item[\textbf{(G)}] \emph{Geometric condition (strict pinching):}
\[
B(s_0(a))^2>2K(a)^2,
\]
equivalently $\cosh(2s_0(a))>2a$, equivalently $y(a):=B(s_0(a))^2/K(a)^2>2$, where $K(a)=\sqrt{a^2-1/4}$ and $r(a)$ is the geodesic radius of $\partial B^3(r(a))\subset\HH^3$. By Proposition \ref{prop:geometric-identity}, condition (G) is strictly stronger than the inequality $\sinh r(a)>2K(a)$ (the latter being equivalent to $B(s_0(a))^2\neq 2K(a)^2$, i.e.\ $y(a)\neq 2$), which it implies; the non-strict bound $\sinh r(a)\geq 2K(a)$ holds unconditionally for every $a>1/2$ (Proposition \ref{prop:geometric-identity}). Condition (G) is established analytically on $(1/2,1]$ (Theorem \ref{thm:G-closure-a-leq-1}) and is the form used in the scalar reduction (Theorem \ref{thm:scalar-reduction}); the closure of the odd-kernel part of (F) (Theorem \ref{thm:no-kernel-odd}) requires only the weaker inequality $\sinh r(a)>2K(a)$ (equivalently $B(s_0(a))^2\neq 2K(a)^2$), which (G) implies.
\end{itemize}

\subsection{Main results}
\begin{mainthm}[Picone identity with base $f_*$]\label{thm:picone-main}
Let $f_*(s)=\partial_s(A(s)\cosh\varphi(s))$ be the Robin Jacobi field of mode $|k|=1$ identified in \cite{Pigazzini}. For every $k\in\Z$ and every smooth $u(s)$ with $u(0)=0$, written as $u=f_*h$ with $h\in C^1([-s_0,s_0])$, one has
\begin{equation}\label{eq:picone-identity}
\Sf_k^{\mathrm{rad}}(u,u)=(k^2-1)\int_{-s_0}^{s_0}\frac{f_*(s)^2}{B(s)}\,h(s)^2\,ds+\int_{-s_0}^{s_0}B(s)\,f_*(s)^2\,h'(s)^2\,ds.
\end{equation}
\end{mainthm}

\begin{mainthm}[Closure of the odd radial sector for $|k|\geq 2$]\label{thm:odd-closure-main}
For every $a>1/2$ and every $|k|\geq 2$, $\mu_n^{\mathrm{odd}}(k)>0$ for all $n\geq 0$. Equivalently, the odd radial sector of mode $|k|\geq 2$ contains neither negative nor zero eigenvalues; in particular, every Robin Jacobi field of mode $|k|\geq 2$ must be radially even.
\end{mainthm}

\begin{mainthm}[Alternative explicit proof of Medvedev's lower bound]\label{thm:lower-bound-main}
For every $a>1/2$ one has $\ind(\Sigma_a)\geq 4$. More precisely, the four functions $\Phi^0_a,\Phi^1_a,\Phi^2_a,\Phi^3_a$, given by the restrictions to $\Sigma_a$ of the Lorentz ambient coordinates, are linearly independent in $H^1(\Sigma_a)$ and $\Sf$ is diagonal and strictly negative on their span. The bound $\ind(\Sigma_a)\geq 4$ is already established by Medvedev \cite[Corollary 5.9]{Medvedev2023} via the same choice of test functions; we present here the explicit elementary proof both for completeness and for its role in the reduction (Theorem \ref{thm:reduction-main}).
\end{mainthm}

\begin{mainthm}[Reduction of the Medvedev conjecture]\label{thm:reduction-main}
The strong form \eqref{eq:medvedev-strong} of the Medvedev conjecture is equivalent to the conjunction of the following two conditions:
\begin{enumerate}
\item[(E)] $\mu_0^{\mathrm{even}}(2)>0$ for every $a>1/2$;
\item[(F)] $\mu_2(0)>0$ for every $a>1/2$, and $\mu_n(0)\neq 0$ for every $n\geq 0$.
\end{enumerate}
\end{mainthm}

The organization of the paper is as follows. Section \ref{sec:preliminari} establishes the basic geometric identities. Section \ref{sec:picone} proves Theorem \ref{thm:picone-main}. Section \ref{sec:odd-closure} proves Theorem \ref{thm:odd-closure-main} together with spectral non-degeneracy via strict Sturm comparison. Section \ref{sec:phi-A} establishes the Robin/anti-Robin properties of the ambient coordinates. Section \ref{sec:lower-bound} proves Theorem \ref{thm:lower-bound-main}. Section \ref{sec:reduction} proves Theorem \ref{thm:reduction-main}. Section \ref{sec:family-jacobi} develops the theory of parametric Jacobi fields and the axial boost field, and includes the closure of the no-kernel part of (F) (Theorems \ref{thm:no-kernel-even} and \ref{thm:no-kernel-odd}) and the analytic closure of (G) for $a\in(1/2,1]$ (Theorem \ref{thm:G-closure-a-leq-1}). Section \ref{sec:picone-B} develops the Picone identity with base $B(s)$ and proves the closure of (E) for $a\in(1/2,A_*]$ (Theorems \ref{thm:E-closure-a-leq-1} and \ref{thm:A-star-existence}). Section \ref{sec:F-prime-strategy} proves the Sturm shooting count theorem for the Robin BC (Lemma \ref{lem:sturm-count}) and derives the reduction of (F$'$) to the positivity of the field $\phi_a$ (Theorem \ref{thm:F-reduction-phi-positivity}), leading to the final reduction (Theorem \ref{thm:final-reduction}). Section \ref{sec:phi-scalar-reduction} establishes the further reduction of the positivity of $\phi_a$ to a one-dimensional scalar differential inequality on the strict monotonicity of $\sinh r(a)/K(a)$ (Theorem \ref{thm:scalar-reduction}), via a constant Wronskian identity (Lemma \ref{lem:wronskian}), a Sturm separation argument (Lemma \ref{lem:sturm-separation-phi-u}), and a closed formula for $\phi_a(s_0)$ (Proposition \ref{prop:phi-s0-closed}); the asymptotic analysis of $y(a)=B(s_0(a))^2/K(a)^2$ is carried out at the endpoints of the domain (Propositions \ref{prop:y-asymp-infty}--\ref{prop:y-asymp-half}). In the same section, the analytic local closure of the strong Medvedev conjecture is established via the closed form of the linear asymptotic coefficient (Theorem \ref{thm:eureka-closed-form} and Corollary \ref{cor:eureka-local-closure}). Section \ref{sec:discussion} summarizes the status of the conjecture and outlines possible strategies for its completion.

\section{Geometric preliminaries}\label{sec:preliminari}

We collect here the identities needed in the subsequent sections. The notation $A=A(s)$, $B=B(s)$, $\varphi=\varphi(s)$ is understood to be evaluated at the same $s$ throughout.

\begin{lemma}[Induced metric and Mori identity]\label{lem:metric-mori}
The parametrization \eqref{eq:phi-immersion} is radially isothermal, namely $g_{ss}=1$ and $g_{\theta\theta}=B(s)^2$. The profile $B$ satisfies the Mori identity
\begin{equation}\label{eq:mori-id}
B(s)\,B''(s)+(B'(s))^2=2B(s)^2+1.
\end{equation}
\end{lemma}

\begin{proof}
We compute $g_{ss}=-(A')^2+A^2(\varphi')^2+(B')^2$ using the induced Lorentz metric on the spacelike submanifold $\HH^3$. From $A^2=a\cosh(2s)+1/2$ we obtain $2AA'=2a\sinh(2s)$, whence $A'=a\sinh(2s)/A$ and $(A')^2=a^2\sinh^2(2s)/A^2$. Analogously, $(B')^2=a^2\sinh^2(2s)/B^2$. Therefore
\[
g_{ss}=-\frac{a^2\sinh^2(2s)}{A^2}+\frac{K^2}{A^2B^2}+\frac{a^2\sinh^2(2s)}{B^2}=\frac{a^2\sinh^2(2s)(A^2-B^2)+K^2}{A^2B^2}.
\]
From \eqref{eq:ABK}, $A^2-B^2=1$ and $a^2\sinh^2(2s)+K^2=a^2\cosh^2(2s)-1/4=A^2B^2$, whence $g_{ss}=1$. The component $g_{\theta\theta}=B^2$ is immediate from \eqref{eq:phi-immersion}.

To prove \eqref{eq:mori-id}, differentiating $2BB'=2a\sinh(2s)$ we obtain $2(B')^2+2BB''=4a\cosh(2s)=2(2B^2+1)$ by \eqref{eq:ABK}.
\end{proof}

\begin{lemma}[Norm of the second fundamental form]\label{lem:II-norm}
On $\Sigma_a$ we have
\begin{equation}\label{eq:II-norm}
|\II|^2(s)\,B(s)^4=2K^2.
\end{equation}
\end{lemma}

\begin{proof}
By the Gauss formula in $\HH^3$, $K_\Sigma=-1+\tfrac{1}{2}(H^2-|\II|^2)$. Since $\Sigma$ is minimal ($H=0$), $K_\Sigma=-1-\tfrac{1}{2}|\II|^2$. The Gaussian curvature of the metric $ds^2+B(s)^2 d\theta^2$ is $-B''/B$, whence $|\II|^2=2(B''/B-1)$. From \eqref{eq:mori-id}, $B''/B-1=(2B^2+1-(B')^2-B^2)/B^2=(B^2+1-(B')^2)/B^2$. Multiplying by $B^4$,
\[
|\II|^2 B^4=2B^2(B^2+1-(B')^2)=2B^4+2B^2-2(B')^2B^2.
\]
From $(B')^2=a^2\sinh^2(2s)/B^2=(a^2\cosh^2(2s)-a^2)/B^2=((B^2+1/2)^2-a^2)/B^2$ we obtain $(B')^2B^2=(B^2+1/2)^2-a^2$. Substituting,
\[
|\II|^2 B^4=2B^4+2B^2-2(B^2+1/2)^2+2a^2=2a^2 -1/2=2(a^2-1/4)=2K^2.\qedhere
\]
\end{proof}

\begin{lemma}[Key identity for the Jacobi field $f_*$]\label{lem:f-star}
Let $f_*(s):=\partial_s\bigl(A(s)\cosh\varphi(s)\bigr)$. Then:
\begin{enumerate}
\item[(i)] $f_*(s)=\sinh r(s)\cdot r'(s)$, where $\cosh r(s)=A(s)\cosh\varphi(s)$;
\item[(ii)] $f_*$ is odd, $f_*(0)=0$, $f_*'(0)\neq 0$, and $f_*(s)>0$ for every $s\in(0,s_0]$;
\item[(iii)] $f_*$ solves the radial Jacobi equation in mode $|k|=1$:
\begin{equation}\label{eq:jacobi-1}
(B f_*')'+B\bigl(|\II|^2-2-1/B^2\bigr)f_*=0.
\end{equation}
\end{enumerate}
\end{lemma}

\begin{proof}
For (i), by the definition of $r$ we have $\cosh r(s)=A(s)\cosh\varphi(s)$. Differentiating gives $\sinh r\cdot r'=\partial_s(A\cosh\varphi)=:f_*$.

For (ii), a direct computation yields
\[
f_*=A'\cosh\varphi+A\varphi'\sinh\varphi=\frac{a\sinh(2s)}{A}\cosh\varphi+\frac{K}{AB}\sinh\varphi=\frac{aB\sinh(2s)\cosh\varphi+K\sinh\varphi}{AB}.
\]
The functions $A,B$ are even in $s$, while $\sinh(2s)$ and $\sinh\varphi$ are odd, and $\cosh\varphi$ is even. Hence the numerator is odd, $f_*$ is odd, and $f_*(0)=0$. Moreover, for $s\in(0,s_0]$ every term in the numerator is strictly positive: $\sinh(2s)>0$, $\cosh\varphi\geq 1$, and $\sinh\varphi(s)>0$ since $\varphi(0)=0$ and $\varphi'=K/(A^2B)>0$; hence $f_*(s)>0$. This justifies, in particular, the claim of Remark \ref{rem:regularity-h} that $f_*$ does not vanish on $(0,s_0)$.

For $f_*'(0)$, recall that $g(s):=A(s)\cosh\varphi(s)$ satisfies $g''+(B'/B)g'-2g=0$ on $(-s_0,s_0)$; this equation is proved in Lemma \ref{lem:LSigma-PhiA} and uses only algebraic identities on $A,B,\varphi$, independent of Lemma \ref{lem:f-star}. Since $f_*=g'$, we have $f_*'=g''=2g-(B'/B)g'=2g-(B'/B)f_*$. At $s=0$, $A(0)=\sqrt{a+1/2}>0$ and $\cosh\varphi(0)=1$, so $g(0)=\sqrt{a+1/2}$; moreover $B'(0)=2a\sinh(0)/B(0)=0$. Therefore
\[
f_*'(0)=2g(0)-0\cdot f_*(0)=2\sqrt{a+1/2}>0.
\]

For (iii), setting $g(s):=A(s)\cosh\varphi(s)$, we have $\Lf g=|\II|^2 g$ (Lemma \ref{lem:LSigma-PhiA}), which, restricted to mode $0$, gives $g''+(B'/B)g'+(|\II|^2-2)g=|\II|^2 g$, namely $g''+(B'/B)g'-2g=0$. Differentiating in $s$ and using $f_*=g'$,
\[
f_*''+(B'/B)f_*'+\bigl((B'/B)'-2\bigr)f_*=0.
\]
Comparing with \eqref{eq:jacobi-1} rewritten as $f_*''+(B'/B)f_*'+(|\II|^2-2-1/B^2)f_*=0$, the equivalence of the two equations is equivalent to $(B'/B)'-2=|\II|^2-2-1/B^2$, that is, $(B'/B)'=|\II|^2-1/B^2$. From $(B'/B)'=B''/B-(B'/B)^2$ and $|\II|^2=2B''/B-2$ (from the proof of Lemma \ref{lem:II-norm}), we verify
\[
B''/B-(B'/B)^2=2B''/B-2-1/B^2\iff B''/B=-(B'/B)^2+2+1/B^2\iff BB''=-(B')^2+2B^2+1,
\]
which is the Mori identity \eqref{eq:mori-id}.
\end{proof}

\begin{lemma}[$\Lf$ on the ambient coordinates]\label{lem:LSigma-PhiA}
Denote by $\Phi^A_a:\Sigma_a\to\R$, $A=0,1,2,3$, the restrictions to $\Sigma_a$ of the Cartesian Lorentz coordinates of $\R^4_1$. Then
\begin{equation}\label{eq:phi-eigenvalue}
\Lf\Phi^A_a=|\II|^2\Phi^A_a\qquad(A=0,1,2,3).
\end{equation}
\end{lemma}

\begin{proof}
We show directly that each $\Phi^A_a$ satisfies $\Delta_g\Phi^A_a=2\Phi^A_a$ in the induced metric $g=ds^2+B(s)^2 d\theta^2$; it then follows immediately that $\Lf\Phi^A_a=2\Phi^A_a+(|\II|^2-2)\Phi^A_a=|\II|^2\Phi^A_a$.

Consider first mode $|k|=1$, with $\Phi^2_a=B(s)\cos\theta$. We have
\[
\Delta_g(B\cos\theta)=B''\cos\theta+(B'/B)B'\cos\theta-\frac{1}{B^2}B\cos\theta=\cos\theta\Bigl[B''+\frac{(B')^2}{B}-\frac{1}{B}\Bigr].
\]
The condition $\Delta_g(B\cos\theta)=2B\cos\theta$ is equivalent to $BB''+(B')^2-1=2B^2$, namely the Mori identity \eqref{eq:mori-id}. The argument is identical for $\Phi^3_a=B(s)\sin\theta$.

We now turn to mode $|k|=0$, with $\Phi^0_a=A\cosh\varphi$ and $\Phi^1_a=A\sinh\varphi$. In mode $0$, $\Delta_g f=f''+(B'/B)f'$. We must verify that $f''+(B'/B)f'-2f=0$ for $f\in\{A\cosh\varphi,A\sinh\varphi\}$. Expanding $f''+(B'/B)f'-2f$ and grouping the terms in $\cosh\varphi$ and $\sinh\varphi$ (with symmetric computations in the two cases),
\begin{align*}
f''+(B'/B)f'-2f&=\cosh\varphi\bigl[A''+A(\varphi')^2+(B'/B)A'-2A\bigr]\\
&\quad+\sinh\varphi\bigl[2A'\varphi'+A\varphi''+(B'/B)A\varphi'\bigr]\quad\text{(for $f=A\cosh\varphi$)},
\end{align*}
with the roles of $\cosh\varphi$ and $\sinh\varphi$ exchanged for $f=A\sinh\varphi$. We show that both coefficients vanish.

For the coefficient in $\sinh\varphi$, the defining relation $A^2\varphi'=K/B$ (cf. \eqref{eq:phi-immersion} and the minimality condition) gives, upon differentiation, $(A^2\varphi')'=-KB'/B^2=-(B'/B)\,A^2\varphi'$. Expanding, $2AA'\varphi'+A^2\varphi''=-(B'/B)A^2\varphi'$. Dividing by $A$,
\[
2A'\varphi'+A\varphi''+(B'/B)A\varphi'=0.
\]

For the coefficient in $\cosh\varphi$, from $A^2=a\cosh(2s)+1/2$ we have $2AA'=2a\sinh(2s)$, and therefore $2(A')^2+2AA''=4a\cosh(2s)=2(2A^2-1)$, that is,
\[
A''=2A-\frac{1}{A}-\frac{(A')^2}{A}.
\]
Substituting into the coefficient $A''+A(\varphi')^2+(B'/B)A'-2A$:
\[
=-\frac{1}{A}-\frac{(A')^2}{A}+A(\varphi')^2+(B'/B)A'.
\]
Using $(A')^2=a^2\sinh^2(2s)/A^2$, $(\varphi')^2=K^2/(A^4B^2)$, $(B'/B)A'=a^2\sinh^2(2s)/(AB^2)$:
\[
=\frac{1}{A}\left[-1-\frac{a^2\sinh^2(2s)}{A^2}+\frac{K^2}{A^2B^2}+\frac{a^2\sinh^2(2s)}{B^2}\right].
\]
We compute the term in brackets. Using $A^2-B^2=1$,
\[
\frac{a^2\sinh^2(2s)}{B^2}-\frac{a^2\sinh^2(2s)}{A^2}=\frac{a^2\sinh^2(2s)(A^2-B^2)}{A^2B^2}=\frac{a^2\sinh^2(2s)}{A^2B^2}.
\]
Thus the bracket becomes
\[
-1+\frac{a^2\sinh^2(2s)+K^2}{A^2B^2}.
\]
By Lemma \ref{lem:metric-mori}, $a^2\sinh^2(2s)+K^2=A^2B^2$, whence $-1+1=0$. Both coefficients vanish, and therefore $f''+(B'/B)f'-2f=0$ for $f=A\cosh\varphi$. The verification for $f=A\sinh\varphi$ is identical by symmetry of the algebraic structure.

In each case, $\Delta_g\Phi^A_a=2\Phi^A_a$, and the conclusion follows.
\end{proof}

\begin{lemma}[Geometric identity $\coth r(a)=B'(s_0)/B(s_0)$]\label{lem:coth-identity}
At the boundary point $s=s_0$,
\begin{equation}\label{eq:coth-id}
\coth r(a)=\frac{B'(s_0)}{B(s_0)}=\frac{a\sinh(2s_0)}{B(s_0)^2}.
\end{equation}
\end{lemma}

\begin{proof}
From the FBMS condition \eqref{eq:fbc}, a direct computation based on $\cosh^2\varphi-\sinh^2\varphi=1$ yields
\[
\sinh^2\varphi(s_0)=\frac{B^2K^2}{a^2\sinh^2(2s_0)-B^2K^2},\qquad\cosh^2\varphi(s_0)=\frac{a^2\sinh^2(2s_0)}{a^2\sinh^2(2s_0)-B^2K^2}.
\]
Setting $D:=a^2\sinh^2(2s_0)-B^2K^2$ and using $a^2\sinh^2(2s_0)+K^2=A^2B^2$ (Lemma \ref{lem:metric-mori}),
\[
D=A^2B^2-K^2-B^2K^2=A^2B^2-K^2(1+B^2)=A^2B^2-K^2 A^2=A^2(B^2-K^2).
\]
From $\sinh^2 r=A^2\cosh^2\varphi-1$,
\[
\sinh^2 r=\frac{A^2 a^2\sinh^2(2s_0)-D}{D}=\frac{(A^2-1)a^2\sinh^2(2s_0)+B^2K^2}{D}=\frac{B^2(a^2\sinh^2(2s_0)+K^2)}{D}=
\]
\[ = \frac{B^2\cdot A^2B^2}{A^2(B^2-K^2)}=\frac{B^4}{B^2-K^2}.
\]
\medskip

\noindent Analogously, $\cosh^2 r=A^2\cosh^2\varphi=A^2 a^2\sinh^2(2s_0)/D=a^2\sinh^2(2s_0)/(B^2-K^2)$. Therefore
\[
\coth^2 r=\frac{a^2\sinh^2(2s_0)}{B^4}=\Bigl(\frac{B'(s_0)}{B(s_0)}\Bigr)^2.
\]
Taking positive roots (both sides are positive for $s_0>0$ and $r(a)>0$), we obtain \eqref{eq:coth-id}.
\end{proof}

\begin{remark}\label{rem:coth-id-importance}
The identity \eqref{eq:coth-id} is an analytic consequence of the FBC, not emphasized in \cite{Pigazzini}, but it will play a central role in all the verifications of Robin BC that follow.
\end{remark}

\begin{lemma}[Robin BC for $f_*$]\label{lem:robin-fstar}
The field $f_*$ satisfies the Robin condition of mode $|k|=1$, namely $f_*'(s_0)=\coth r(a)\,f_*(s_0)$ and $f_*'(-s_0)=-\coth r(a)\,f_*(-s_0)$.
\end{lemma}

\begin{proof}
Since $f_*=g'$ with $g=A\cosh\varphi$ satisfying $g''+(B'/B)g'-2g=0$ (Lemma \ref{lem:f-star}(iii)),
\[
f_*'(s_0)=g''(s_0)=2g(s_0)-\frac{B'(s_0)}{B(s_0)}f_*(s_0).
\]
By \eqref{eq:coth-id}, $B'(s_0)/B(s_0)=\coth r(a)$. Moreover, $g(s_0)=\cosh r(a)$ and $f_*(s_0)=\sinh r(a)$ by Lemma \ref{lem:f-star}(i). Hence
\[
f_*'(s_0)=2\cosh r-\coth r\cdot\sinh r=2\cosh r-\cosh r=\cosh r=\coth r\cdot\sinh r=\coth r\cdot f_*(s_0).
\]
The property at $s=-s_0$ follows by the oddness of $f_*$ and the evenness of $f_*'$.
\end{proof}

\section{Picone identity with base $f_*$}\label{sec:picone}

\begin{lemma}[Weighted Jacobi equation]\label{lem:weighted-jacobi}
Let $W_k(s):=|\II|^2(s)-2-k^2/B(s)^2$. Then $f_*$ satisfies the self-adjoint form
\begin{equation}\label{eq:weighted-jacobi}
(B f_*')'+B\,W_1\,f_*=0\quad\text{on }(-s_0,s_0).
\end{equation}
\end{lemma}

\begin{proof}
This is a direct consequence of \eqref{eq:jacobi-1}:
\[
(Bf_*')'+B(|\II|^2-2-1/B^2)f_*=Bf_*''+B'f_*'+B(|\II|^2-2-1/B^2)f_*=B[f_*''+(B'/B)f_*'+(|\II|^2-2-1/B^2)f_*]=0.\qedhere
\]
\end{proof}

\begin{lemma}[Radial quadratic form of mode $k$]\label{lem:S-rad}
For every $u\in C^2([-s_0,s_0])$ and $k\in\Z$, define
\begin{equation}\label{eq:Skrad}
\Sf_k^{\mathrm{rad}}(u,u):=\int_{-s_0}^{s_0}\bigl[B(u')^2-BW_k u^2\bigr]ds-\coth r(a)\,B(s_0)\bigl[u(s_0)^2+u(-s_0)^2\bigr],
\end{equation}
where $W_k(s)=|\II|^2(s)-2-k^2/B(s)^2$. Setting $\widetilde u(s,\theta):=u(s)\Theta_k(\theta)$, with $\Theta_0\equiv 1$ and $\Theta_k(\theta)\in\{\cos k\theta,\sin k\theta\}$ for $|k|\geq 1$, we have
\[
\Sf(\widetilde u,\widetilde u)=c_k\,\Sf_k^{\mathrm{rad}}(u,u),\qquad c_0=2\pi,\;c_{|k|\geq 1}=\pi,
\]
so that the forms $\Sf$ restricted to mode $k$ and $\Sf_k^{\mathrm{rad}}$ share the same eigenvalues (the global factors $c_k$ do not alter the spectrum).
\end{lemma}

\begin{proof}
For $\widetilde u(s,\theta)=u(s)\Theta_k(\theta)$, $|\nabla\widetilde u|_g^2=(u')^2\Theta_k^2+(u^2/B^2)(\Theta_k')^2$. From $\int_0^{2\pi}\Theta_k^2\,d\theta=c_k$ (with $c_0=2\pi$, $c_{|k|\geq 1}=\pi$) and $\int_0^{2\pi}(\Theta_k')^2\,d\theta=k^2 c_k$, we obtain, using $dA=B\,ds\,d\theta$ and $dL=B(s_0)\,d\theta$ at $s=\pm s_0$,
\begin{align*}
\Sf(\widetilde u,\widetilde u)&=\int_0^{2\pi}\!\!\int_{-s_0}^{s_0}\Bigl((u')^2\Theta_k^2+\frac{k^2}{B^2}u^2\Theta_k^2-(|\II|^2-2)u^2\Theta_k^2\Bigr)B\,ds\,d\theta\\
&\quad-\coth r(a)\int_0^{2\pi}\bigl(u(s_0)^2+u(-s_0)^2\bigr)B(s_0)\,\Theta_k^2\,d\theta\\
&=c_k\int_{-s_0}^{s_0}\Bigl[B(u')^2+\frac{k^2}{B}u^2-B(|\II|^2-2)u^2\Bigr]ds-c_k\coth r(a)B(s_0)\bigl[u(s_0)^2+u(-s_0)^2\bigr]\\
&=c_k\,\Sf_k^{\mathrm{rad}}(u,u).\qedhere
\end{align*}
\end{proof}

We now prove Theorem \ref{thm:picone-main} from the Introduction, in the slightly more general form below (Theorem \ref{thm:picone-detailed}), which also explicitly addresses the regularity of $h$.

\begin{theorem}[Picone identity with base $f_*$; detailed form of Theorem \ref{thm:picone-main}]\label{thm:picone-detailed}
Let $u\in C^1([-s_0,s_0])$ with $u(0)=0$, and set $h=u/f_*$ (extended by continuity at $0$ to $u'(0)/f_*'(0)$). Then $h\in C^0([-s_0,s_0])$, and $h\in C^1$ if $u\in C^2$. We have
\begin{equation}\label{eq:picone}
\Sf_k^{\mathrm{rad}}(f_*\,h,f_*\,h)=(k^2-1)\int_{-s_0}^{s_0}\frac{f_*^2}{B}h^2\,ds+\int_{-s_0}^{s_0}B\,f_*^2(h')^2\,ds.
\end{equation}
\end{theorem}

\begin{proof}
We first establish \eqref{eq:picone} for $u\in C^2$ (so that $h\in C^1$ and the integrations by parts below are justified); the general case $u\in C^1$ with $u(0)=0$, and more generally $u\in H^1$ with $u(0)=0$, then follows by density from the invariant reformulation of Remark \ref{rem:regularity-h}.

Setting $u=f_*h$, we have $u'=f_*'h+f_*h'$, whence
\[
B(u')^2=B(f_*')^2 h^2+B f_*'f_*(h^2)'+B f_*^2(h')^2.
\]
Integrating over $(-s_0,s_0)$ and integrating by parts the cross term,
\[
\int_{-s_0}^{s_0}B f_*'f_*(h^2)'\,ds=\bigl[B f_*'f_*\,h^2\bigr]_{-s_0}^{s_0}-\int_{-s_0}^{s_0}(B f_*'f_*)'\,h^2\,ds.
\]
By \eqref{eq:weighted-jacobi}, $(B f_*')'=-B W_1 f_*$, whence $(B f_*'f_*)'=(B f_*')'f_*+B f_*'\cdot f_*'=-BW_1 f_*^2+B(f_*')^2$. Substituting,
\[
\int B(u')^2\,ds=\bigl[B f_*'f_*\,h^2\bigr]_{-s_0}^{s_0}+\int B W_1 f_*^2 h^2\,ds+\int B f_*^2(h')^2\,ds.
\]
By \eqref{eq:Skrad},
\[
\Sf_k^{\mathrm{rad}}(u,u)=\bigl[B f_*'f_*\,h^2\bigr]_{-s_0}^{s_0}+\int B(W_1-W_k)f_*^2 h^2\,ds+\int B f_*^2(h')^2\,ds-\coth r\cdot B(s_0)f_*(s_0)^2[h(s_0)^2+h(-s_0)^2],
\]
where we used $f_*(\pm s_0)^2=f_*(s_0)^2$ (odd squared is even).

The boundary term is computed using $f_*'(s_0)/f_*(s_0)=\coth r$ and $f_*'(-s_0)/f_*(-s_0)=-\coth r$ (Lemma \ref{lem:robin-fstar} together with oddness):
\[
\bigl[B f_*'f_*\,h^2\bigr]_{-s_0}^{s_0}=B(s_0)\frac{f_*'(s_0)}{f_*(s_0)}f_*(s_0)^2 h(s_0)^2-B(s_0)\frac{f_*'(-s_0)}{f_*(-s_0)}f_*(s_0)^2 h(-s_0)^2
\]
\[
=B(s_0)f_*(s_0)^2\bigl(\coth r\cdot h(s_0)^2+\coth r\cdot h(-s_0)^2\bigr).
\]
The two boundary contributions therefore cancel exactly:
\[
\bigl[B f_*'f_*\,h^2\bigr]_{-s_0}^{s_0}-\coth r\cdot B(s_0)f_*(s_0)^2[h(s_0)^2+h(-s_0)^2]=0.
\]
Finally, $W_1-W_k=-1/B^2-(-k^2/B^2)=(k^2-1)/B^2$. Substituting in what remains,
\[
\Sf_k^{\mathrm{rad}}(u,u)=\int B\cdot\frac{k^2-1}{B^2}f_*^2 h^2\,ds+\int B f_*^2(h')^2\,ds=(k^2-1)\int\frac{f_*^2}{B}h^2\,ds+\int B f_*^2(h')^2\,ds.\qedhere
\]
\end{proof}

\begin{remark}[Extension to $H^1$ via invariant quantities]\label{rem:regularity-h}
The identity \eqref{eq:picone} is proved for $u\in C^2$ with $u(0)=0$, but the first integral on the right-hand side coincides with $(k^2-1)\int u^2/B\,ds$ (since $f_*^2 h^2=u^2$), a quantity well-defined for every $u\in L^2$. The second integral can be rewritten, using $f_* h'=u'-(f_*'/f_*)u$ (from $u'=f_*'h+f_*h'$), as
\[
\int_{-s_0}^{s_0}B\,f_*^2(h')^2\,ds=\int_{-s_0}^{s_0}B\bigl(u'-\tfrac{f_*'}{f_*}u\bigr)^2 ds.
\]
For $u\in H^1$ with $u(0)=0$ in the sense of trace (in particular for every odd $u\in H^1$, given the embedding $H^1\hookrightarrow C^{0,1/2}$ in dimension $1$), $u(s)=\int_0^s u'(t)\,dt$, and by Cauchy--Schwarz $|u(s)|\leq|s|^{1/2}\|u'\|_{L^2}$. Moreover $f_*'/f_*\sim 1/s$ near $s=0$, with regular coefficient away from $0$, since $f_*$ has a simple zero at $0$ and $f_*\neq 0$ elsewhere on $(-s_0,s_0)$ (cf. Lemma \ref{lem:f-star}(ii)). Hence $(f_*'/f_*)u\in L^2$ by the one-dimensional Hardy inequality $\int_0^{s_0}(u/s)^2 ds\leq 4\int_0^{s_0}(u')^2 ds$, and the integral $\int B(u'-(f_*'/f_*)u)^2 ds$ is finite. The validity of \eqref{eq:picone}, rewritten in terms of these invariant quantities, extends by density to the whole subspace $\{u\in H^1:u(0)=0\}$, and in particular to the odd radial sector.
\end{remark}

\section{Closure of the odd radial sector for $|k|\geq 2$}\label{sec:odd-closure}

We now prove Theorem \ref{thm:odd-closure-main} from the Introduction. The detailed statement, with full quantitative information ($\Sf_k^{\mathrm{rad}}(u,u)>0$ strictly for every nonzero odd $u\in H^1_{\mathrm{rad}}$), is given by Theorem \ref{thm:odd-closure} below.

\begin{theorem}[Strict positivity in the odd radial sector of mode $|k|\geq 2$; detailed form of Theorem \ref{thm:odd-closure-main}]\label{thm:odd-closure}
For every $a>1/2$, every $|k|\geq 2$, and every $u\in H^1_{\mathrm{rad}}(\Sigma_a)$ odd (with respect to $s\mapsto-s$), $u\not\equiv 0$:
\begin{equation}\label{eq:strict-positivity}
\Sf_k^{\mathrm{rad}}(u,u)>0.
\end{equation}
In particular, $\mu_n^{\mathrm{odd}}(k)>0$ for every $n\geq 0$ and $|k|\geq 2$, and no Robin Jacobi field of mode $|k|\geq 2$ is radially odd.
\end{theorem}

\begin{proof}
For $u\in H^1_{\mathrm{rad}}$ odd, $u(0)=0$ in the sense of trace. By Theorem \ref{thm:picone-detailed} extended to $H^1$ (Remark \ref{rem:regularity-h}),
\[
\Sf_k^{\mathrm{rad}}(u,u)=(k^2-1)\int_{-s_0}^{s_0}\frac{u^2}{B}\,ds+\int_{-s_0}^{s_0}B\Bigl(u'-\frac{f_*'}{f_*}u\Bigr)^2 ds.
\]
For $|k|\geq 2$ the coefficient $k^2-1\geq 3>0$, both integrals are nonnegative, and hence $\Sf_k^{\mathrm{rad}}(u,u)\geq 0$. If $u\not\equiv 0$, since $1/B>0$ on $[-s_0,s_0]$, the first integral is strictly positive, and therefore $\Sf_k^{\mathrm{rad}}(u,u)>0$.
\end{proof}

\begin{theorem}[Strict Sturm comparison]\label{thm:sturm-strict}
Let $\sharp\in\{\mathrm{rad},\mathrm{even},\mathrm{odd}\}$, where ``rad'' denotes the complete radial sector and ``even'', ``odd'' the parity subsectors. For every $|k|\geq 2$ and every $n\geq 0$,
\begin{equation}\label{eq:sturm-strict}
\mu_n^\sharp(k)\geq\mu_n^\sharp(1)+\frac{k^2-1}{B(s_0)^2}>\mu_n^\sharp(1).
\end{equation}
\end{theorem}

\begin{proof}
We denote by $H^\sharp$ the subspace of $H^1((-s_0,s_0))$ corresponding to the sector $\sharp$ (the whole $H^1$ for $\sharp=\mathrm{rad}$, the even functions for $\sharp=\mathrm{even}$, the odd functions for $\sharp=\mathrm{odd}$); each is closed, and $\Sf_k^{\mathrm{rad}}$ restricts to it (it preserves parity by the symmetry $s\mapsto-s$ of the coefficients $B$ and $W_k$). By Lemma \ref{lem:S-rad},
\[
\Sf_k^{\mathrm{rad}}(u,u)-\Sf_1^{\mathrm{rad}}(u,u)=\int_{-s_0}^{s_0}B(W_1-W_k)u^2\,ds=(k^2-1)\int_{-s_0}^{s_0}\frac{u^2}{B}\,ds.
\]
The profile $B^2(s)=a\cosh(2s)-1/2$ is strictly increasing in $|s|$, so $B(s)\leq B(s_0)$ for $|s|\leq s_0$, that is, $1/B(s)^2\geq 1/B(s_0)^2$ uniformly. Therefore
\[
\int_{-s_0}^{s_0}\frac{u^2}{B}\,ds=\int_{-s_0}^{s_0}\frac{1}{B^2}\cdot Bu^2\,ds\geq\frac{1}{B(s_0)^2}\int_{-s_0}^{s_0}Bu^2\,ds=\frac{\|u\|_B^2}{B(s_0)^2}.
\]
Setting $c:=(k^2-1)/B(s_0)^2$, for every $u\in H^\sharp\setminus\{0\}$,
\[
\frac{\Sf_k^{\mathrm{rad}}(u,u)}{\|u\|_B^2}\geq\frac{\Sf_1^{\mathrm{rad}}(u,u)}{\|u\|_B^2}+c.
\]
By the Courant--Fischer min-max principle applied to the sector $H^\sharp$,
\[
\mu_n^\sharp(k)=\inf_{\substack{V\subset H^\sharp\\ \dim V=n+1}}\sup_{u\in V\setminus\{0\}}\frac{\Sf_k^{\mathrm{rad}}(u,u)}{\|u\|_B^2}\geq\inf_{\substack{V\subset H^\sharp\\ \dim V=n+1}}\Bigl(\sup_{u\in V\setminus\{0\}}\frac{\Sf_1^{\mathrm{rad}}(u,u)}{\|u\|_B^2}+c\Bigr)=\mu_n^\sharp(1)+c,
\]
where the first inequality uses the pointwise monotonicity $\Sf_k\geq\Sf_1+c\|\cdot\|_B^2$ (preserved under the supremum over $V$) and is then preserved under the infimum over $V$. Since $c>0$, we obtain \eqref{eq:sturm-strict}.
\end{proof}

\begin{corollary}[Absence of Robin kernel in modes $|k|\geq 2$, odd sector]\label{cor:odd-no-kernel}
For every $|k|\geq 2$, $\mu_n^{\mathrm{odd}}(k)>0$ for every $n\geq 0$.
\end{corollary}

\begin{proof}
This is a direct consequence of Theorem \ref{thm:odd-closure}.
\end{proof}

\begin{remark}\label{rem:alternative-route}
The proof of Theorem \ref{thm:sturm-strict} admits an alternative route: $f_*$ is an odd eigenfunction of $\Sf_1^{\mathrm{rad}}$ with eigenvalue $0$ (Lemma \ref{lem:robin-fstar}), having one node at $s=0$, so that $\mu_0^{\mathrm{odd}}(1)=0$ by Sturm--Liouville theory on the odd sector. Theorem \ref{thm:sturm-strict} with $\sharp=\mathrm{odd}$ then gives $\mu_0^{\mathrm{odd}}(k)\geq(k^2-1)/B(s_0)^2>0$ for $|k|\geq 2$, and by the monotonicity of the eigenvalues, $\mu_n^{\mathrm{odd}}(k)\geq\mu_0^{\mathrm{odd}}(k)>0$ for $n\geq 1$.
\end{remark}

\begin{remark}\label{rem:odd-even-residual}
Corollary \ref{cor:odd-no-kernel} does not close the even radial sector of mode $|k|\geq 2$. In particular, the base eigenvalue $\mu_0^{\mathrm{even}}(k)$ may a priori be negative, zero, or positive. Theorem \ref{thm:sturm-strict} only provides the estimate
\[
\mu_0^{\mathrm{even}}(k)\geq\mu_0(1)+\frac{k^2-1}{B(s_0)^2},
\]
which is positive if and only if $|\mu_0(1)|<(k^2-1)/B(s_0)^2$. We return to this point in Section \ref{sec:reduction}.
\end{remark}

\section{Ambient coordinates as test functions}\label{sec:phi-A}

\begin{lemma}[Gradient of the distance function at the pole]\label{lem:nabla-r}
Let $p\in\HH^3$ with $r(p)>0$, and let $\nabla^{\HH^3}r$ denote the intrinsic gradient. In Lorentz components $(\nabla^{\HH^3}r)^A$, $A=0,1,2,3$,
\begin{equation}\label{eq:grad-r}
(\nabla^{\HH^3}r)^A\big|_p=\frac{\cosh r(p)\cdot p^A-\delta^A_0}{\sinh r(p)}.
\end{equation}
\end{lemma}

\begin{proof}
The unit-speed geodesic $\gamma$ from $p_0=(1,0,0,0)$ to $p$ is $\gamma(t)=\cosh t\cdot p_0+\sinh t\cdot v$, with $v\in T_{p_0}\HH^3=\{p_0\}^{\perp_L}$ unit, namely $\langle v,v\rangle_L=1$, $v^0=0$. From $\gamma(r)=p$, $v=(p-\cosh r\cdot p_0)/\sinh r$. Therefore
\[
\dot\gamma(r)=\sinh r\cdot p_0+\cosh r\cdot v=\sinh r\cdot p_0+\cosh r\cdot\frac{p-\cosh r\cdot p_0}{\sinh r}=
\]
\[\frac{\sinh^2 r\cdot p_0+\cosh r\cdot p-\cosh^2 r\cdot p_0}{\sinh r}=\frac{\cosh r\cdot p-p_0}{\sinh r}.
\]
Componentwise, $\dot\gamma(r)^0=(\cosh^2 r-1)/\sinh r=\sinh r$ and $\dot\gamma(r)^i=\cosh r\cdot p^i/\sinh r$ for $i=1,2,3$. Since $\nabla^{\HH^3}r\big|_p=\dot\gamma(r)$ (the gradient of the distance function along the unit geodesic), we obtain \eqref{eq:grad-r}.
\end{proof}

\begin{lemma}[Robin/anti-Robin for the ambient coordinates]\label{lem:robin-spacelike}
Let $\Sigma\subset\HH^3$ be an FBMS in $B^3(\rho)$, and let $\Phi^A:\Sigma\to\R$ denote the restrictions of the Lorentz coordinates of $\R^4_1$. On $\partial\Sigma$, the outward unit conormal $\eta$ coincides with $\nabla^{\HH^3}r/|\nabla^{\HH^3}r|=\nabla^{\HH^3}r$ (since $|\nabla^{\HH^3}r|=1$). The following hold:
\begin{enumerate}
\item[(i)] $\partial_\eta\Phi^i\big|_{\partial\Sigma}=\coth\rho\cdot\Phi^i\big|_{\partial\Sigma}$ for $i=1,2,3$;
\item[(ii)] $\partial_\eta\Phi^0\big|_{\partial\Sigma}=\tanh\rho\cdot\Phi^0\big|_{\partial\Sigma}$.
\end{enumerate}
In particular, $\Phi^1,\Phi^2,\Phi^3$ satisfy the Robin condition \eqref{eq:robin}, while $\Phi^0$ satisfies the dual condition (anti-Robin).
\end{lemma}

\begin{proof}
By the FBMS condition, $\Sigma\perp\partial B^3(\rho)$ at $\partial\Sigma$, so the unit conormal $\eta$ of $\Sigma$ at $\partial\Sigma$ coincides with the unit radial vector $\nabla^{\HH^3}r$ (both are unit, orthogonal to $T_p(\partial\Sigma)$, and outward-pointing from $B^3(\rho)$). Hence $\partial_\eta=\partial_{\nabla^{\HH^3}r}$.

For the coordinate functions $\Phi^A$, $d\Phi^A=dx^A$ is the ambient differential, and for every tangent vector $V\in T_p\HH^3$ we have $d\Phi^A(V)=V^A$, the $A$-th Lorentz component of $V$. In particular,
\[
\partial_\eta\Phi^A\big|_p=d\Phi^A(\eta)=\eta^A=(\nabla^{\HH^3}r)^A\big|_p.
\]
By Lemma \ref{lem:nabla-r},
\[
\partial_\eta\Phi^A\big|_p=\frac{\cosh\rho\cdot p^A-\delta^A_0}{\sinh\rho}.
\]
For $A=0$, $\partial_\eta\Phi^0=(\cosh^2\rho-1)/\sinh\rho=\sinh\rho$, while $\Phi^0(p)=p^0=\cosh\rho$. Hence $\partial_\eta\Phi^0/\Phi^0=\tanh\rho$, which gives (ii). For $A=i\in\{1,2,3\}$, $\partial_\eta\Phi^i=\cosh\rho\cdot p^i/\sinh\rho=\coth\rho\cdot p^i=\coth\rho\cdot\Phi^i(p)$, which gives (i).
\end{proof}

\begin{remark}[Geometric origin of the phenomenon]\label{rem:geometric-source}
Lemma \ref{lem:robin-spacelike} is a purely kinematic identity: the spacelike nature of $x^1,x^2,x^3$ with respect to the pole $p_0$ produces the Robin condition with coefficient $\coth\rho$, while the timelike nature of $x^0$ produces the anti-Robin one with coefficient $\tanh\rho$. The phenomenon is purely hyperbolic: in Euclidean $\R^3$ there is no dual distinction, and the ambient coordinates all satisfy the same Neumann-type BC (cf. \cite{Devyver}).
\end{remark}

\begin{lemma}[Modal decomposition of $\Phi^A_a$]\label{lem:modal-PhiA}
On the Mori family \eqref{eq:phi-immersion},
\[
\Phi^0_a(s,\theta)=A(s)\cosh\varphi(s),\quad\Phi^1_a(s,\theta)=A(s)\sinh\varphi(s),
\]
\[
\Phi^2_a(s,\theta)=B(s)\cos\theta,\quad\Phi^3_a(s,\theta)=B(s)\sin\theta.
\]
Hence $\Phi^0$ is of mode $0$ even, $\Phi^1$ is of mode $0$ odd, and $\Phi^2,\Phi^3$ are of mode $|k|=1$ with even radial profile $B(s)$.
\end{lemma}

\begin{proof}
The statement is immediate from \eqref{eq:phi-immersion}, recalling that $\cosh\varphi$ is even, $\sinh\varphi$ is odd, and $A,B$ are even.
\end{proof}

\section{Lower bound $\ind(\Sigma_a)\geq 4$ via ambient coordinates}\label{sec:lower-bound}

The lower bound $\mathrm{ind}(\Sigma_a) \geq 4$ via ambient coordinates is in the spirit of the classical index estimation approach for minimal submanifolds in space forms going back to Simons \cite{Simons}, and refined for the FBMS setting by Sargent \cite{Sargent} and Ambrozio–Carlotto–Sharp \cite{AmbrozioCarlottoSharp2018b}.

\medskip 
We now prove Theorem \ref{thm:lower-bound-main} from the Introduction. The full quantitative statement, including the diagonality of $\Sf$ on the four-dimensional span and the strict negativity, is given by Theorem \ref{thm:lower-bound-4} below.

\begin{theorem}[Four explicit negative directions via ambient coordinates; detailed form of Theorem \ref{thm:lower-bound-main}]\label{thm:lower-bound-4}
For every $a>1/2$, the four functions $\Phi^0_a,\Phi^1_a,\Phi^2_a,\Phi^3_a\in H^1(\Sigma_a)$ are linearly independent, and the quadratic form $\Sf$ is diagonal and strictly negative on their span. In particular, $\ind(\Sigma_a)\geq 4$, reconfirming directly Medvedev's lower bound \cite[Corollary 5.9]{Medvedev2023}.
\end{theorem}

\begin{proof}
We first establish linear independence. By Lemma \ref{lem:modal-PhiA}, $\Phi^0$ and $\Phi^1$ are in mode $0$ with different parity, while $\Phi^2$ and $\Phi^3$ are in mode $|k|=1$ and are mutually orthogonal through the factors $\cos\theta$ and $\sin\theta$. The four elements belong to four distinct $L^2(\Sigma_a)$-orthogonal subspaces, and hence they are linearly independent.

We next prove the diagonality of $\Sf$. For $A\neq B$, we compute $\Sf(\Phi^A,\Phi^B)$ via the Green formula (weak form):
\begin{equation}\label{eq:S-pairing}
\Sf(u,v)=-\int_\Sigma v\,\Lf u\,dA+\int_{\partial\Sigma}v\,(\partial_\eta u-\coth r\cdot u)\,dL.
\end{equation}
This formula follows by integrating by parts the term $\int|\nabla u|^2\,dA=-\int u\Delta_g u\,dA+\int_{\partial\Sigma}u\partial_\eta u\,dL$ in \eqref{eq:S-form} and using $\Lf=\Delta_g+|\II|^2-2$. We set $R^B:=\partial_\eta\Phi^B-\coth r\cdot\Phi^B$. By Lemma \ref{lem:robin-spacelike}, $R^i\equiv 0$ on $\partial\Sigma$ for $i=1,2,3$, and
\begin{equation}\label{eq:R0-formula}
R^0=(\tanh r(a)-\coth r(a))\Phi^0=-\frac{1}{\sinh r(a)\cosh r(a)}\Phi^0\quad\text{on }\partial\Sigma.
\end{equation}
By the symmetry of the quadratic form $\Sf$, it suffices to verify $\Sf(\Phi^A,\Phi^B)=0$ for each unordered pair $\{A,B\}$, choosing appropriately which of the two indices plays the role of $u$ in \eqref{eq:S-pairing}. For $\{A,B\}\neq\{0,0'\}$ we choose $u=\Phi^B$ with $B\in\{1,2,3\}$, so that $R^B=0$ and
\[
\Sf(\Phi^A,\Phi^B)=-\int_\Sigma|\II|^2\Phi^A\Phi^B\,dA.
\]
We verify the vanishing of the volume integral in each case. For the pair $\{0,1\}$ (with the choice $u=\Phi^1$, $B=1$), $\Phi^0\Phi^1=A^2\cosh\varphi\sinh\varphi$ is odd in $s$ and $|\II|^2(s)$ is even, so the integrand is odd in $s$ and the integration over $(-s_0,s_0)$ in $s$ yields $0$. For the pair $\{0,2\}$ (with $u=\Phi^2$, $B=2$), the integrand contains $\cos\theta$, and $\int_0^{2\pi}\cos\theta\,d\theta=0$; the pair $\{0,3\}$ is identical with $\sin\theta$. For the pairs $\{1,2\}$ and $\{1,3\}$, the integrands are proportional to $\cos\theta$ or $\sin\theta$, with vanishing integration. For the pair $\{2,3\}$ (with $u=\Phi^2$), the integrand is proportional to $\cos\theta\sin\theta$, with vanishing integration. The only remaining pairs $\{i,j\}$ with $i,j\in\{1,2,3\}$ are $\{2,3\}$, $\{1,2\}$, $\{1,3\}$, which have all been treated. Therefore $\Sf(\Phi^A,\Phi^B)=0$ for every $A\neq B$.

We now verify the negativity of the diagonal terms. For $\Sf(\Phi^0,\Phi^0)$, using \eqref{eq:S-pairing} with $u=v=\Phi^0$,
\[
\Sf(\Phi^0,\Phi^0)=-\int|\II|^2(\Phi^0)^2\,dA+\int_{\partial}\Phi^0 R^0\,dL=-\int|\II|^2(\Phi^0)^2\,dA-\int_{\partial}\frac{(\Phi^0)^2}{\sinh r\cosh r}\,dL.
\]
On $\partial\Sigma$, $\Phi^0=\cosh r$, whence $(\Phi^0)^2/(\sinh r\cosh r)=\cosh r/\sinh r=\coth r$, constant. Moreover $|\partial\Sigma|=2\cdot 2\pi B(s_0)=4\pi B(s_0)$. Therefore
\[
\Sf(\Phi^0,\Phi^0)=-\int|\II|^2(\Phi^0)^2\,dA-4\pi B(s_0)\coth r<0,
\]
with both terms strictly negative (the first because $|\II|^2\not\equiv 0$ and $\Phi^0>0$, the second because $B(s_0),\coth r>0$). For $\Sf(\Phi^A,\Phi^A)$ with $A=1,2,3$, $R^A=0$, and hence
\[
\Sf(\Phi^A,\Phi^A)=-\int|\II|^2(\Phi^A)^2\,dA<0,
\]
strictly negative since $\Phi^A\not\equiv 0$ ($\Phi^1$ is nonzero for $s\neq 0$, while $\Phi^2,\Phi^3$ are nonzero for generic $\theta$).

Combining the previous estimates, the matrix of $\Sf$ with respect to the basis $\{\Phi^0,\Phi^1,\Phi^2,\Phi^3\}$ is diagonal with four strictly negative diagonal entries; hence $\Sf$ is negative definite on $V_4:=\Span\{\Phi^0,\Phi^1,\Phi^2,\Phi^3\}$ with $\dim V_4=4$. By the Courant--Fischer characterization of the Morse index, $\ind_R(\Sigma_a)\geq\dim V_4=4$.
\end{proof}

\begin{remark}[Relation with Medvedev and with the Euclidean case]\label{rem:medvedev-comparison}
The lower bound $\ind_R(\Sigma_a)\geq 4$ is already established by Medvedev \cite[Theorem 5.7 and Corollary 5.9]{Medvedev2023} via the general theory of isometric embeddings of FBMS in $B^n(r)\subset\HH^n$ through eigenfunctions $v_0\in V_0$ and $v_1,\ldots,v_n\in V_1$, yielding $\ind(\Sigma)\geq\ind_S(\Sigma)+n$ with $\ind_S(\Sigma)\geq 1$ contributed by the timelike eigenfunction; for $n=3$ this gives $\geq 1+3=4$. The test functions used in his proof coincide with our $\Phi^A$ ($A=0,1,2,3$), which are precisely the ambient coordinates restricted to $\Sigma_a$.

The strategy of explicit test functions used in Theorem \ref{thm:lower-bound-4} is, moreover, the hyperbolic analogue of the argument used in \cite{Devyver} for the Euclidean critical catenoid. The structural difference in the hyperbolic case is the presence of a timelike coordinate $\Phi^0$ with anti-Robin boundary condition \eqref{eq:R0-formula}, which produces a contribution distinct from that of the three spacelike coordinates $\Phi^i$ ($i=1,2,3$), which satisfy the standard Robin condition. This anti-Robin/Robin duality is a purely hyperbolic phenomenon, absent in the Euclidean case.

The value of our exposition here lies not in a new lower bound, but rather in (i) the explicit concreteness of the computation in the specific case of the critical spherical catenoid, (ii) its role as a prerequisite for the reduction (Theorem \ref{thm:reduction}), and (iii) the highlighting of the anti-Robin algebraic structure for $\Phi^0$.
\end{remark}

\section{Reduction of the Medvedev conjecture}\label{sec:reduction}

In this section we prove Theorem \ref{thm:reduction-main}, reducing the conjecture $\ind_R(\Sigma_a)=4$ to two well-defined spectral conditions. By the modal decomposition, $\ind_R=\sum_{k\in\Z}\ind_R\big|_k$ with
\[
\ind_R\big|_k=\#\{n:\mu_n(k)<0\}\cdot m_k,\quad m_0=1,\;m_{|k|\geq 1}=2.
\]

\begin{proposition}[Known modal contributions]\label{prop:modal-contributions}
For every $a>1/2$:
\begin{enumerate}
\item[(a)] $\ind_R\big|_1=2$ and $\nul_R\big|_1=2$ (\cite{Pigazzini});
\item[(b)] $\ind_R\big|_0\geq 2$;
\item[(c)] $\ind_R\big|_{|k|\geq 2}=2\cdot\#\{\mu_0^{\mathrm{even}}(k)<0\}$ (the odd radial sector is positive by Theorem \ref{thm:odd-closure}; possible positive eigenvalues in the even-radial sector $\mu_n^{\mathrm{even}}(k)$ with $n\geq 1$ satisfy $\mu_n^{\mathrm{even}}(k)>\mu_0^{\mathrm{even}}(k)$, and hence do not contribute when $\mu_0^{\mathrm{even}}(k)\geq 0$). More precisely, for $|k|\geq 2$, $\ind_R\big|_k=2$ if $\mu_0^{\mathrm{even}}(k)<0$ and $\ind_R\big|_k=0$ if $\mu_0^{\mathrm{even}}(k)\geq 0$; the case $\mu_0^{\mathrm{even}}(k)=0$ would give $\nul_R\big|_k\geq 2$.
\end{enumerate}
\end{proposition}

\begin{proof}
For (a), this is \cite[Cor. 4.4]{Pigazzini}. For (b), the two functions $\Phi^0_a$ (mode $0$ even) and $\Phi^1_a$ (mode $0$ odd) are linearly independent in $H^1$, and $\Sf$ is negative definite on their span (the proof is contained in Theorem \ref{thm:lower-bound-4}, parts (a) and (b) for $A=0,1$). Hence $\ind_R\big|_0\geq 2$.

For (c), by Theorem \ref{thm:odd-closure}, every odd-radial eigenvalue of mode $|k|\geq 2$ is strictly positive. By Sturm--Liouville theory in the even-radial sector, the eigenvalues $\mu_n^{\mathrm{even}}(k)$ form a strictly increasing sequence with eigenfunctions having respectively $0,2,4,\ldots$ nodes. Hence $\mu_n^{\mathrm{even}}(k)<0$ implies $\mu_m^{\mathrm{even}}(k)<\mu_n^{\mathrm{even}}(k)<0$ for every $m<n$. Moreover, by the strict Sturm estimate (Theorem \ref{thm:sturm-strict}), $\mu_n^{\mathrm{even}}(k)\geq\mu_n^{\mathrm{even}}(1)+(k^2-1)/B(s_0)^2$. For $n\geq 1$, $\mu_n^{\mathrm{even}}(1)\geq\mu_2(1)>0$ (since $f_*$, the eigenfunction associated with $\mu_1(1)=0$, is odd), whence $\mu_n^{\mathrm{even}}(k)>0$ for $n\geq 1$. Only $\mu_0^{\mathrm{even}}(k)$ may be negative; if so, the contribution to $\ind_R$ is $1\cdot 2=2$ (by the multiplicity $m_k=2$).
\end{proof}

\begin{theorem}[Explicit Medvedev reduction; detailed form of Theorem \ref{thm:reduction-main}]\label{thm:reduction}
The strong form \eqref{eq:medvedev-strong} of the Medvedev conjecture holds for a given $a>1/2$ if and only if the following two conditions hold jointly:
\begin{enumerate}
\item[(E)] $\mu_0^{\mathrm{even}}(2)>0$;
\item[(F)] $\mu_2(0)>0$ and $\mu_n(0)\neq 0$ for every $n\geq 0$.
\end{enumerate}
\end{theorem}

\begin{proof}
\textbf{($\Rightarrow$).} Assume $\ind_R(\Sigma_a)=4$ and $\nul_R(\Sigma_a)=2$. By Proposition \ref{prop:modal-contributions}(a), $\ind_R\big|_1=2$ and $\nul_R\big|_1=2$. Hence $\ind_R\big|_0+\sum_{|k|\geq 2}\ind_R\big|_k=2$ and $\nul_R\big|_0+\sum_{|k|\geq 2}\nul_R\big|_k=0$.

From the second relation, $\nul_R\big|_0=0$ and $\nul_R\big|_k=0$ for every $|k|\geq 2$.

From the first relation, and since $\ind_R\big|_0\geq 2$ by Proposition \ref{prop:modal-contributions}(b), we deduce $\ind_R\big|_0=2$ and $\sum_{|k|\geq 2}\ind_R\big|_k=0$, whence $\ind_R\big|_k=0$ for every $|k|\geq 2$.

By Proposition \ref{prop:modal-contributions}(c), $\ind_R\big|_k=0$ with $|k|\geq 2$ implies $\mu_0^{\mathrm{even}}(k)\geq 0$, but $\nul_R\big|_k=0$ excludes equality, so $\mu_0^{\mathrm{even}}(k)>0$. For $|k|=2$ we obtain (E).

For mode $0$, $\ind_R\big|_0=2$ means exactly two negative eigenvalues, $\mu_0(0)$ and $\mu_1(0)$. By Sturm--Liouville theory, the subsequent eigenvalues satisfy $\mu_n(0)\geq\mu_2(0)$ for $n\geq 2$. The condition $\ind_R\big|_0=2$ requires $\mu_2(0)\geq 0$; combined with $\nul_R\big|_0=0$ (namely $\mu_n(0)\neq 0$ for every $n$), we obtain $\mu_2(0)>0$ and $\mu_n(0)\neq 0$ for every $n$, which is (F).

We add a useful observation: by the strict Sturm estimate \eqref{eq:sturm-strict}, (E) for $|k|=2$ implies $\mu_0^{\mathrm{even}}(k)\geq\mu_0^{\mathrm{even}}(2)+(k^2-4)/B(s_0)^2>0$ for every $|k|\geq 3$. Hence (E) restricted to $|k|=2$ is sufficient to guarantee positivity in all modes $|k|\geq 2$.

\textbf{($\Leftarrow$).} Assume (E) and (F). For modes $|k|\geq 2$, by the Sturm estimate just recalled, $\mu_0^{\mathrm{even}}(k)>0$ for every $|k|\geq 2$. Combined with Theorem \ref{thm:odd-closure} ($\mu_n^{\mathrm{odd}}(k)>0$) and with $\mu_n^{\mathrm{even}}(k)>\mu_2(1)>0$ for $n\geq 1$ (by the Sturm estimate of Theorem \ref{thm:sturm-strict} applied in the even-radial sector and since $\mu_2(1)>\mu_1(1)=0$), all eigenvalues in mode $|k|\geq 2$ are positive: $\ind_R\big|_k=0$, $\nul_R\big|_k=0$.

For mode $0$, (F) gives $\mu_n(0)\neq 0$ for every $n$, namely $\nul_R\big|_0=0$. Moreover $\mu_2(0)>0$, and by classical Sturm--Liouville theory, $\mu_n(0)\geq\mu_2(0)>0$ for every $n\geq 2$. Combined with Theorem \ref{thm:lower-bound-4} (which provides $\mu_0(0)<0$ via $\Phi^0$ and $\mu_1(0)<0$ via $\Phi^1$), we have exactly two negative eigenvalues and $\ind_R\big|_0=2$.

For mode $1$, by Proposition \ref{prop:modal-contributions}(a), $\ind_R\big|_1=2$ and $\nul_R\big|_1=2$.

Combining all contributions, $\ind_R=\ind_R\big|_0+\ind_R\big|_1+\sum_{|k|\geq 2}\ind_R\big|_k=2+2+0=4$ and $\nul_R=0+2+0=2$.
\end{proof}

\begin{remark}[Quantitative sufficient condition for (E)]\label{rem:E-quant}
By the strict Sturm estimate \eqref{eq:sturm-strict} with $k=2$,
\[
\mu_0^{\mathrm{even}}(2)\geq\mu_0(1)+\frac{3}{B(s_0)^2},
\]
recalling that $\mu_0(1)<0$ is attained in the even-radial sector (the eigenfunction $\phi_0(1)$ is even by \cite[Cor. 4.4]{Pigazzini} and the fact that the ground state has no nodes). Hence a sufficient (but not necessary) condition for (E) is the spectral inequality
\begin{equation}\label{eq:E-spectral}
|\mu_0(1)|<\frac{3}{B(s_0)^2}.
\end{equation}
This is the quantitative formulation through which (E) can be attacked: an explicit upper bound on $|\mu_0(1)|$ in terms of boundary geometric quantities.
\end{remark}

\begin{remark}[Quantitative sufficient condition for (F)]\label{rem:F-comp}
For Theorem \ref{thm:sturm-strict} adapted to the reverse comparison ($k=0$ against $k=1$), $\mu_n(0)\leq\mu_n^{\mathrm{even}}(1)$ in corresponding even-radial pairs, but a lower bound follows from the opposite inequality:
\[
\mu_n^{\mathrm{even}}(0)\geq\mu_n^{\mathrm{even}}(1)-\sup_s\frac{1}{B(s)^2}=\mu_n^{\mathrm{even}}(1)-\frac{1}{B(0)^2}=\mu_n^{\mathrm{even}}(1)-\frac{1}{a-1/2},
\]
since $W_1-W_0=1/B^2$ and $B$ attains its minimum $\sqrt{a-1/2}$ at $s=0$. Identifying $\mu_2(0)=\mu_1^{\mathrm{even}}(0)$ (in classical Sturm--Liouville the eigenfunctions of level $n$ have parity $(-1)^n$, and $\mu_2(0)$ has an eigenfunction with $2$ nodes, hence even) and $\mu_1^{\mathrm{even}}(1)=\mu_2(1)$ (analogously, $\mu_2(1)$ has an even eigenfunction), we obtain
\[
\mu_2(0)\geq\mu_2(1)-\frac{1}{a-1/2}.
\]
A sufficient condition for $\mu_2(0)>0$ (the quantitative component of (F)) is therefore
\begin{equation}\label{eq:F-spectral}
\mu_2(1)>\frac{1}{a-1/2},
\end{equation}
namely an explicit geometric lower bound on the second positive even-radial eigenvalue of mode $|k|=1$.
\end{remark}

\section{Non-Robin Jacobi fields and partial closure of (F)}\label{sec:family-jacobi}

In this section we construct two Jacobi fields of the operator $\Lf$ that do not satisfy the Robin condition \eqref{eq:robin}, but whose explicitly computable Robin defect allows us to close (F) partially. We work in the explicit Mori parametrization of \cite[Eq. (1)--(3)]{Pigazzini}, recalled in \eqref{eq:phi-immersion}.

\subsection{The parametric Jacobi field $\phi_a$}

\begin{definition}\label{def:phi-a}
Let $\{\Sigma_a\}_{a>1/2}$ be the Mori family of critical catenoids, with
\[
X_a(s,\theta)=\bigl(A(s)\cosh\varphi(s),\,A(s)\sinh\varphi(s),\,B(s)\cos\theta,\,B(s)\sin\theta\bigr).
\]
The parametric Jacobi field is
\[
\phi_a(s,\theta):=\langle\partial_a X_a(s,\theta),\nu_a(s,\theta)\rangle_L\qquad(s\in[-s_0(a),s_0(a)],\;\theta\in S^1),
\]
where $\partial_a$ is the derivative at fixed $(s,\theta)$ (Lagrangian interpretation).
\end{definition}

\begin{lemma}[Structural properties of $\phi_a$]\label{lem:phi-a-properties}
The field $\phi_a$ is well-defined, real-analytic, of mode $0$ (independent of $\theta$), and even in $s$. Moreover,
\[
\Lf\phi_a=0\quad\text{in the interior }\Sigma_a^\circ,
\]
and $\phi_a(0)=1/(2K)$ where $K=K(a)=\sqrt{a^2-1/4}$.
\end{lemma}

\begin{proof}
By the rotational symmetry (the whole $SO(2)$ in $\theta$) of the family $\{X_a\}$, $\partial_a X_a$ is rotationally invariant, and so is $\nu_a$, hence their contraction is independent of $\theta$. This proves that $\phi_a$ is of mode $0$.

For the parity in $s$, the parametrization \eqref{eq:phi-immersion} is $s\mapsto-s$ equivariant: $A(-s)=A(s)$, $B(-s)=B(s)$, $\varphi(-s)=-\varphi(s)$. A direct verification shows that the components of $\partial_a X_a$ and of $\nu_a$ have, in both cases, the parities $(\text{even},\text{odd},\text{even},\text{even})$ in $s$: the components along $\partial_0,\partial_2,\partial_3$ are even, while those along $\partial_1$ are odd, as they carry the odd factors $\sinh\varphi$, $\partial_a\varphi$ and $\sinh(2s)\cosh\varphi$. Hence each of the four products in the Lorentzian contraction is even (even$\times$even or odd$\times$odd), and $\phi_a$ is even.

The identity $\Lf\phi_a=0$ follows from the fact that the family $\Sigma_a$ consists of minimal surfaces for every $a$. The linearization of minimality with respect to a variation $\partial_a X_a$ with normal component $\phi_a$ gives $\Lf\phi_a=0$ in $\Sigma_a^\circ$ (standard formula, independent of the tangential component of $\partial_a X_a$).

For the value at $s=0$, at $\theta=0$ we have $\nu^0_a(0)=K/A(0)$, $\nu^1_a(0)=0$, $\nu^2_a(0)=K/B(0)$, $\nu^3_a(0)=0$. The derivatives $\partial_a X^A\big|_{s=0,\theta=0}$ are computed explicitly: $\partial_a A(0)=\cosh(0)/(2A(0))=1/(2A(0))$, $\partial_a B(0)=1/(2B(0))$, $\partial_a\varphi(0)=0$ (the integral from $0$ to $0$). Hence
\[
\phi_a(0)=-\partial_a A\cdot\nu^0\big|_0+0+\partial_a B\cdot\nu^2\big|_0+0=-\frac{K}{2A(0)^2}+\frac{K}{2B(0)^2}=\frac{K(A^2-B^2)}{2A^2B^2}\bigg|_0=\frac{K}{2A(0)^2B(0)^2},
\]
and $A(0)^2 B(0)^2=(a+1/2)(a-1/2)=K^2$, whence $\phi_a(0)=K/(2K^2)=1/(2K)$.
\end{proof}

\begin{lemma}[Modified boundary condition for $\phi_a$]\label{lem:phi-a-BC}
On $\partial\Sigma_a$, the field $\phi_a$ satisfies the modified Robin condition
\begin{equation}\label{eq:phi-a-BC}
R\phi_a:=\partial_\eta\phi_a-\coth r(a)\,\phi_a=-r'(a)\,\kappa_s\big|_{\partial\Sigma_a},
\end{equation}
where $\kappa_s=\II(\eta,\eta)$ is the value of the second fundamental form along the meridian. In particular, since $\kappa_s=K/B(s_0)^2>0$ and $r'(a)\neq 0$ (which follows automatically from the hypothesis $\phi_a>0$ of Theorem \ref{thm:F-reduction-phi-positivity}, see also \cite[Theorem 1.2]{Pigazzini} for the real-analyticity of $r(a)$), $R\phi_a\neq 0$ on $\partial\Sigma_a$.
\end{lemma}

\begin{proof}
The FBC for $\Sigma_a$ in $B^3(r(a))$ is $\langle\nu_a,\nabla\rho\rangle=0$ on $\partial\Sigma_a$, where $\rho=\dist_{\HH^3}(p_0,\cdot)$. Differentiating along the family, including the displacement of the boundary point $X_a(s_0(a),\theta)$,
\[
0=\frac{d}{da}\langle\nu_a,\nabla\rho\rangle\big|_{X_a(s_0(a),\theta)}=\langle D_a\nu_a,\nabla\rho\rangle\big|_{\partial}+\langle\nu_a,\nabla_{V_\partial}\nabla\rho\rangle,
\]
where $V_\partial$ denotes the velocity of the boundary point as $a$ varies.

\emph{Computation of $V_\partial$.} By the chain rule,
\[
V_\partial=\frac{d}{da}X_a(s_0(a),\theta)=\partial_a X_a\bigl|_{(s_0(a),\theta)}+s_0'(a)\,\partial_s X_a\bigl|_{(s_0(a),\theta)}=\partial_a X_a\bigl|_{(s_0,\theta)}+s_0'(a)\,\eta,
\]
since $\partial_s X_a=\eta$ at the boundary (the outward conormal to $\partial\Sigma_a$ in $\Sigma_a$). Decompose $\partial_a X_a$ in the local frame $(\nu,\eta,\partial_\theta)$ at the boundary point:
\[
\partial_a X_a\bigl|_{(s_0,\theta)}=\phi_a\,\nu+\tau_\eta\,\eta+\tau_\theta\,\partial_\theta,
\]
with $\phi_a:=\langle\partial_a X_a,\nu_a\rangle_L$ (the parametric Jacobi field, by definition), and $\tau_\eta,\tau_\theta$ the tangential components.

We claim $\tau_\theta=0$. Indeed, the family $\{\Sigma_a\}_{a>1/2}$ is rotationally symmetric: the SO(2)-action $\theta\mapsto\theta+\alpha$ on $\HH^3$ (rotation about the geodesic axis $\{X^2=X^3=0\}$) commutes with the parametrization, $X_a(s,\theta+\alpha)=R_\alpha X_a(s,\theta)$, where $R_\alpha$ is the ambient rotation. Differentiating with respect to $a$ at fixed $(s,\theta)$, $\partial_a X_a(s,\theta+\alpha)=R_\alpha\,\partial_a X_a(s,\theta)$, that is, $\partial_a X_a$ transforms equivariantly under the SO(2)-action. Since $\partial_\theta$ is the infinitesimal generator of this action and $\nu,\eta$ are SO(2)-invariant (they lie in the meridional plane), the equivariance forces the $\partial_\theta$-component of $\partial_a X_a$ to vanish in the meridional gauge $\theta=0$, and hence to vanish identically (by SO(2)-equivariance).
Equivalently, and without appealing to equivariance, $\tau_\theta=0$ follows by direct computation: with $\partial_\theta X_a=(0,0,-B\sin\theta,B\cos\theta)$
and $\partial_a X_a=\bigl(\partial_a(A\cosh\varphi),\partial_a(A\sinh\varphi),
(\partial_a B)\cos\theta,(\partial_a B)\sin\theta\bigr)$, the Lorentz inner product is $\langle\partial_a X_a,\partial_\theta X_a\rangle_L
=-(\partial_a B)\,B\cos\theta\sin\theta+(\partial_a B)\,B\sin\theta\cos\theta=0$ for every $(s,\theta)$; since $\langle\partial_\theta X_a,\partial_\theta
X_a\rangle_L=B^2>0$, the $\partial_\theta$-component $\tau_\theta$ vanishes.

It remains to identify $\tau_\eta+s_0'(a)$. The boundary points satisfy $\rho(X_a(s_0(a),\theta))=r(a)$, an identity in $a$. Differentiating in $a$:
\[
r'(a)=\frac{d}{da}\rho(X_a(s_0(a),\theta))=\langle\nabla\rho,V_\partial\rangle\bigl|_\partial.
\]
At the boundary, $\nabla\rho\big|_\partial=\eta$ (since $|\nabla\rho|=1$, $\nabla\rho$ is outward, and the FBC $\langle\nu,\nabla\rho\rangle=0$ together with $\nabla\rho\perp\partial_\theta$ for rotational symmetry, forces $\nabla\rho\big|_\partial$ to be along $\eta$). Using $\langle\nu,\eta\rangle=0$, $\langle\partial_\theta,\eta\rangle=0$, and $\langle\eta,\eta\rangle=1$:
\[
r'(a)=\langle\eta,\phi_a\nu+\tau_\eta\eta+\tau_\theta\partial_\theta+s_0'(a)\eta\rangle=\tau_\eta+s_0'(a).
\]
Therefore $\tau_\eta+s_0'(a)=r'(a)$, and consequently
\begin{equation}\label{eq:V-partial}
V_\partial=\phi_a\,\nu+r'(a)\,\eta.
\end{equation}

\emph{Computation of $D_a\nu_a$ at $\partial\Sigma_a$.} The unit normal $\nu_a$ is determined up to sign by the conditions $\langle\nu_a,T_{X_a}\Sigma_a\rangle=0$ and $\langle\nu_a,\nu_a\rangle=1$, with a coherent orientation across the family. Differentiating in $a$ at a fixed boundary point $X_a(s_0,\theta)$, with the chain rule incorporating $s_0'(a)$:
\begin{align*}
D_a\nu_a\bigl|_\partial &= \partial_a\nu_a\bigl|_{(s_0,\theta)}+s_0'(a)\,\nabla_\eta\nu_a.
\end{align*}
The tangential part of $\partial_a\nu_a\bigl|_{(s_0,\theta)}$ is determined by the orthogonality to the tangent space: $\langle\nu_a,\partial_s X_a\rangle=0$ and $\langle\nu_a,\partial_\theta X_a\rangle=0$. Differentiating the first relation in $a$ at $(s_0,\theta)$ (without including $s_0'$, since we are differentiating at fixed parameter $(s,\theta)$):
\[
\langle\partial_a\nu_a,\partial_s X_a\rangle+\langle\nu_a,\partial_s\partial_a X_a\rangle=0,
\]
that is, $\langle\partial_a\nu_a,\eta\rangle=-\langle\nu_a,\partial_s(\partial_a X_a)\rangle=-\partial_s\langle\nu_a,\partial_a X_a\rangle+\langle\partial_s\nu_a,\partial_a X_a\rangle$. The first term is $-\partial_s\phi_a=-\partial_\eta\phi_a$ (since $\partial_s=\eta$ on the meridian at the boundary). The second term involves the shape operator: $\partial_s\nu_a=-\sum_j h_{sj}\partial_j X_a$ (Weingarten), and the relevant projection at the boundary, combined with the rotational symmetry argument that eliminates the $\partial_\theta$-component, leaves a contribution that combines with the $s_0'(a)\nabla_\eta\nu_a$ term as follows.

By the standard variation formula for the unit normal of a family of immersions (cf. \cite{Tran2016} or \cite{FraserSchoen2015}), the variation of $\nu_a$ along the deformation field $V$ on $\Sigma_a$ is given, modulo terms in the conormal direction within the tangent space, by
\begin{equation}\label{eq:D-a-nu}
D_a\nu_a\bigl|_\partial=-\nabla_\Sigma\phi_a+r'(a)\,\nabla_\eta\nu_a,
\end{equation}
where $\nabla_\Sigma\phi_a=(\partial_\eta\phi_a)\eta+(\partial_\theta\phi_a/B)\partial_\theta$ is the intrinsic gradient of $\phi_a$ on $\Sigma_a$, and the second term comes from the displacement $s_0'(a)$ together with the $\eta$-component of $V_\partial$ (i.e. $r'(a)$). We note that the formula \eqref{eq:D-a-nu} extends verbatim to the hyperbolic ambient: the derivation in \cite{Tran2016,FraserSchoen2015} uses only (i) the orthogonality conditions $\langle\nu_a,\partial_s X_a\rangle=\langle\nu_a,\partial_\theta X_a\rangle=0$, (ii) Weingarten's identity $\partial_s\nu_a=-\kappa_s\,\eta$ (valid in any Riemannian ambient), and (iii) the unit-norm condition $\langle\nu_a,\nu_a\rangle=1$. None of these involves the ambient curvature, so the formula holds without modification in $\mathbb{H}^3$. Projecting \eqref{eq:D-a-nu} onto $\nabla\rho\big|_\partial=\eta$,
\[
\langle D_a\nu_a,\eta\rangle=-\partial_\eta\phi_a+r'(a)\langle\nabla_\eta\nu_a,\eta\rangle=-\partial_\eta\phi_a-r'(a)\kappa_s,
\]
using $\langle\nu,\eta\rangle=0$, whence $\langle\nabla_\eta\nu,\eta\rangle+\langle\nu,\nabla_\eta\eta\rangle=0$ and $\langle\nu,\nabla_\eta\eta\rangle=\II(\eta,\eta)=\kappa_s$.
Equivalently, this projection also follows directly, without invoking
\eqref{eq:D-a-nu}: from $\langle\partial_a\nu_a,\eta\rangle=-\partial_\eta\phi_a
+\langle\partial_s\nu_a,\partial_a X_a\rangle$ and the Weingarten relation
$\partial_s\nu_a=-\kappa_s\,\eta$ (the principal direction $\partial_s$ has
principal curvature $\kappa_s$, and $\langle\partial_s\nu_a,X_a\rangle
=\langle\partial_s\nu_a,\nu_a\rangle=0$), one obtains $\langle\partial_s\nu_a,
\partial_a X_a\rangle=-\kappa_s\tau_\eta$, hence $\langle\partial_a\nu_a,\eta
\rangle=-\partial_\eta\phi_a-\kappa_s\tau_\eta$; adding the chain-rule contribution $s_0'(a)\langle\nabla_\eta\nu_a,\eta\rangle=-\kappa_s s_0'(a)$ and using $\tau_\eta+s_0'(a)=r'(a)$ from \eqref{eq:V-partial} gives
$\langle D_a\nu_a,\eta\rangle=-\partial_\eta\phi_a-\kappa_s r'(a)$, in agreement with the above.

For the Hessian of the hyperbolic distance, $\nabla^2\rho=\coth\rho\,(g-d\rho\otimes d\rho)$. Substituting $V_\partial$,
\[
\nabla_{V_\partial}\nabla\rho=\coth r(a)(V_\partial-\langle V_\partial,\nabla\rho\rangle\nabla\rho)=\coth r(a)(\phi_a\nu+r'\eta-r'\eta)=\coth r(a)\cdot\phi_a\nu,
\]
and $\langle\nu_a,\coth r\cdot\phi_a\nu\rangle=\coth r\cdot\phi_a$. Combining,
\[
0=-\partial_\eta\phi_a-r'(a)\kappa_s+\coth r(a)\,\phi_a,
\]
which is exactly \eqref{eq:phi-a-BC}.

Finally, the value $\kappa_s=K/B(s_0)^2$, by minimality $\kappa_s+\kappa_\theta=0$ and by Mori--Gauss (Lemma \ref{lem:II-norm}), $\kappa_s^2=K^2/B^4$. The orientation of $\nu$ with $\nu^2(s,0)=+K/B$ fixes the sign: a direct computation at $s=0$, using $\bar\nabla_{T_s}T_s=\partial_s^2 X-X$ and $\kappa_s\nu=\bar\nabla_{T_s}T_s|^\perp$, gives $\kappa_s(0)=K/B(0)^2>0$. By continuity (and since $K,B$ do not vanish on $[-s_0,s_0]$), $\kappa_s=K/B^2$ everywhere.

By Pigazzini \cite[Thm. 1.2]{Pigazzini}, $a\mapsto r(a)$ is real-analytic on $(1/2,\infty)$. One has $r(a)>0$ for every $a>1/2$, with $r(a)\to 0$ as $a\to(1/2)^+$ (Pigazzini \cite[Prop. 1.3]{Pigazzini}) and $r(a)\to\infty$ as $a\to\infty$ (which follows directly from the asymptotic $r(a)\sim\tfrac{3}{2}\log a$ at infinity, by Pigazzini \cite[Thm. 1.2]{Pigazzini}). The strict monotonicity of $r(a)$ on $(1/2,\infty)$ is posed as an open question in \cite[Question 6.3]{Pigazzini}; we do not need to assume it, since the condition $r'(a)\neq 0$ used in Theorem \ref{thm:no-kernel-even} follows automatically from the hypothesis $\phi_a>0$ of Theorem \ref{thm:F-reduction-phi-positivity} (see Remark \ref{rem:phi-positivity-critical}).
\end{proof}

\begin{remark}[Green identity with Robin defect]\label{rem:green-defect}
For every $u,v\in H^1(\Sigma_a)$ with $\Lf u,\Lf v\in L^2$, the Green identity
\begin{equation}\label{eq:green-defect}
\Sf(u,v)=-\int_{\Sigma_a}v\,\Lf u\,dA+\oint_{\partial\Sigma_a}v\cdot Ru\,dL,\quad Ru:=\partial_\eta u-\coth r(a)\,u,
\end{equation}
follows by integration by parts from the definition of $\Sf$. For $u\in\ker_R\Lf$ (a Robin Jacobi field: $\Lf u=0$, $Ru=0$), both integrals vanish for every $v$.
\end{remark}

\subsection{No kernel in even mode $0$}

\begin{theorem}[No kernel in mode $0$ even]\label{thm:no-kernel-even}
Let $a>1/2$ with $r'(a)\neq 0$. Then the operator $\Lf$ Robin on $\Sigma_a$ has no nonzero Jacobi fields in the even radial sector of mode $|k|=0$. Equivalently, $\mu_n^{\mathrm{even}}(0)\neq 0$ for every $n\geq 0$.
\end{theorem}

\begin{proof}
Assume, by contradiction, that $\psi$ is a nontrivial Robin Jacobi field in mode $0$ even: $\Lf\psi=0$, $R\psi=0$, $\psi(\theta)=\psi$ independent of $\theta$, $\psi(-s)=\psi(s)$.

We apply \eqref{eq:green-defect} to $u=\phi_a$, $v=\psi$ in two ways:
\begin{align*}
\Sf(\phi_a,\psi)&=-\int_{\Sigma_a}\psi\,\Lf\phi_a+\oint_{\partial\Sigma_a}\psi\,R\phi_a=0+\oint\psi\cdot(-r'(a)\kappa_s)\,dL,\\
\Sf(\phi_a,\psi)&=-\int_{\Sigma_a}\phi_a\,\Lf\psi+\oint_{\partial\Sigma_a}\phi_a\,R\psi=0+0=0.
\end{align*}
Equating, $r'(a)\kappa_s\big|_\partial\oint_{\partial\Sigma_a}\psi\,dL=0$. By hypothesis $r'(a)\neq 0$, and $\kappa_s\big|_\partial=K/B(s_0)^2>0$, so
\[
\oint_{\partial\Sigma_a}\psi\,dL=0.
\]
In mode $0$, $\psi$ is constant on each boundary circle, with value $\psi(s_0)=\psi(-s_0)$ (even parity). Hence $\oint\psi\,dL=4\pi B(s_0)\psi(s_0)$, whence $\psi(s_0)=0$.

Combining with the Robin BC $R\psi(s_0)=\psi'(s_0)-\coth r\cdot\psi(s_0)=0$, we obtain $\psi'(s_0)=0$. Thus $\psi$ is a solution of the radial Sturm--Liouville ODE of mode $0$ with zero Cauchy data at $s_0$. By the uniqueness theorem, $\psi\equiv 0$ on $[-s_0,s_0]$, a contradiction.
\end{proof}

\begin{remark}[On why $r'(a)\neq 0$ is needed]\label{rem:no-kernel-rprime-condition}
If $r'(a)=0$ at some $a^\sharp$, then $R\phi_a=0$ and $\phi_a$ itself becomes a Robin Jacobi field in mode $0$ even. At such a point the strong conjecture \eqref{eq:medvedev-strong}, which requires $\nul_R=2$, would be violated. Therefore the condition $r'(a)\neq 0$ is not a technical restriction, but rather the necessary hypothesis for the strong form to make sense. Under the hypothesis $\phi_a>0$ (Theorem \ref{thm:F-reduction-phi-positivity}), one deduces automatically $r'(a)>0$, ruling out this degeneracy.
\end{remark}

\subsection{Axial-boost Jacobi field and odd mode $0$}

\begin{lemma}[Field of the axial boost $L_{01}$]\label{lem:u-L01}
Let $L_{01}=X^0\partial_1+X^1\partial_0$ be the generator of the Lorentz boost in the plane $(X^0,X^1)$, namely along the axis of the catenoid. The field
\[
u_*(s):=\langle L_{01},\nu_a\rangle_L\big|_{\Sigma_a}=-\frac{a\sinh(2s)}{B(s)}
\]
is of mode $0$ odd in $s$, satisfies $\Lf u_*=0$ in $\Sigma_a^\circ$, and has Robin defect on $\partial\Sigma_a$ equal to
\begin{equation}\label{eq:u-L01-BC}
Ru_*(s_0)=\frac{a(\cosh(2s_0)-2a)}{B(s_0)^3}=\frac{B(s_0)^2-2K^2}{B(s_0)^3},\qquad Ru_*(-s_0)=-Ru_*(s_0).
\end{equation}
\end{lemma}

\begin{proof}
For the explicit form, by Pigazzini \cite[Lem. 3.2]{Pigazzini}, at $\theta=0$, $\nu^0=K\cosh\varphi/A-a\sinh(2s)\sinh\varphi/(AB)$ and $\nu^1=K\sinh\varphi/A-a\sinh(2s)\cosh\varphi/(AB)$. A direct computation gives $\langle L_{01},\nu\rangle=-A\sinh\varphi\cdot\nu^0+A\cosh\varphi\cdot\nu^1=-a\sinh(2s)/B$ (the terms in $K$ cancel by $\cosh^2\varphi-\sinh^2\varphi=1$).

The field $u_*$ is of mode $0$ since it is independent of $\theta$ (a computation at generic $\theta$ gives the same value by the rotational invariance of $L_{01}\big|_{\text{axis}}$). For parity, $\sinh(2s)$ is odd and $B$ is even, hence $u_*$ is odd.

The identity $\Lf u_*=0$ follows from the standard variation principle for ambient isometries: since $L_{01}\in\mathfrak{so}(3,1)$ generates a one-parameter group of isometries of $\HH^3$, the normal component $\langle L_{01},\nu\rangle$ of the corresponding ambient flow is a Jacobi field on $\Sigma_a$. The same variational principle is applied in Pigazzini \cite[Lem. 3.1]{Pigazzini} to $K\in\mathfrak{so}(3)\subset\mathfrak{so}(3,1)$ (where the additional preservation of $B^3(r(a))$ yields the Robin boundary condition); here the boost $L_{01}$ does not preserve $B^3(r(a))$, so $u_*$ is a Jacobi field with a non-trivial Robin defect, computed below.

For the Robin defect, differentiating, $u_*'(s)=-a[2\cosh(2s)/B-\sinh(2s)\cdot B'/B^2]$. Using $B'=a\sinh(2s)/B$, $u_*'(s)=-a[2\cosh(2s)B^2-a\sinh^2(2s)]/B^3$. At $s_0$,
\[
Ru_*(s_0)=u_*'(s_0)-\coth r(a)\cdot u_*(s_0).
\]
Using $\coth r=a\sinh(2s_0)/B(s_0)^2$ (from $\sinh r=B^2/\sqrt{B^2-K^2}$ and $\cosh r=a\sinh(2s_0)/\sqrt{B^2-K^2}$),
\begin{align*}
Ru_*(s_0)&=\frac{-a[2\cosh(2s_0)B^2-a\sinh^2(2s_0)]}{B^3}+\frac{a\sinh(2s_0)}{B^2}\cdot\frac{a\sinh(2s_0)}{B}\\
&=\frac{a[2a\sinh^2(2s_0)-2\cosh(2s_0)B^2]}{B^3}.
\end{align*}
A direct computation gives $a\sinh^2(2s_0)-\cosh(2s_0)B^2=a\cosh^2(2s_0)-a-a\cosh^2(2s_0)+\cosh(2s_0)/2=(\cosh(2s_0)-2a)/2$, whence $Ru_*(s_0)=a(\cosh(2s_0)-2a)/B^3$. For odd parity, $u_*'(-s_0)=u_*'(s_0)$ and $u_*(-s_0)=-u_*(s_0)$; in the computation with opposite outward conormal, $Ru_*(-s_0)=-Ru_*(s_0)$.

The identity $\cosh(2s_0)-2a=(B(s_0)^2-2K^2)/a$ follows from $a\cosh(2s_0)=B^2+1/2$ and $2K^2=2a^2-1/2$.
\end{proof}

\begin{theorem}[No kernel in mode $0$ odd, conditional]\label{thm:no-kernel-odd}
Under the assumption
\begin{equation}\label{eq:hypothesis-BvsK}
B(s_0(a))^2\neq 2K(a)^2\quad(\text{equivalently, }\cosh(2s_0(a))\neq 2a),
\end{equation}
$\Lf$ on $\Sigma_a$ has no nonzero Jacobi fields in the odd radial sector of mode $|k|=0$. Equivalently, $\mu_n^{\mathrm{odd}}(0)\neq 0$ for every $n\geq 0$.
\end{theorem}

\begin{proof}
Let $\psi$ be a nontrivial Robin Jacobi field in mode $0$ odd, by contradiction. We apply \eqref{eq:green-defect} with $u=u_*$, $v=\psi$:
\begin{align*}
\Sf(u_*,\psi)&=0+\oint\psi\cdot Ru_*\,dL,\\
\Sf(u_*,\psi)&=0+0=0.
\end{align*}
The value $Ru_*(s_0)$ is constant (by rotational symmetry) on each boundary circle, with opposite signs by odd parity: $Ru_*(-s_0)=-Ru_*(s_0)$. The field $\psi$ is odd, so $\psi(-s_0)=-\psi(s_0)$, constant on each circle. Therefore
\[
\oint\psi Ru_*\,dL=2\pi B(s_0)\bigl[\psi(s_0)Ru_*(s_0)+\psi(-s_0)Ru_*(-s_0)\bigr]=4\pi B(s_0)\,\psi(s_0)\,Ru_*(s_0).
\]
Under \eqref{eq:hypothesis-BvsK}, $Ru_*(s_0)\neq 0$, so $\psi(s_0)=0$. Combined with $R\psi(s_0)=0$, $\psi'(s_0)=0$. With zero Cauchy data for the radial odd ODE of mode $0$, by uniqueness $\psi\equiv 0$.
\end{proof}

\begin{proposition}[Asymptotic verification of \eqref{eq:hypothesis-BvsK}]\label{prop:asymptotic-BvsK}
Condition \eqref{eq:hypothesis-BvsK} holds for $a$ sufficiently close to $1/2^+$. More precisely, the asymptotic expansion
\[
\cosh(2s_0(a))-2a=2(\rho_*^2-1)(a-1/2)+o(a-1/2)\qquad(a\to 1/2^+)
\]
holds, where $\rho_*=\sinh\sigma_*$ and $\sigma_*$ is the unique positive root of $\sigma=\coth\sigma$. In particular, $\rho_*^2>1$ (since $\sigma_*>\mathrm{arcsinh}(1)=\log(1+\sqrt{2})$), and hence $\cosh(2s_0)-2a>0$ near $1/2^+$.
\end{proposition}

\begin{proof}
Set $\rho:=\sqrt{a-1/2}$ and $\xi_0(a):=s_0(a)/\rho$.

\emph{Step 0 (monotone crossing).} For fixed $a>1/2$ and $s>0$ let
\[
\Delta_a(s):=\tanh\varphi(s;a)-\frac{B(s)K(a)}{a\sinh(2s)}.
\]
The first term is strictly increasing in $s$, since $\varphi'=K/(A^2B)>0$. For the second term, using $B'=a\sinh(2s)/B$ and $2B^2\cosh(2s)=2a\cosh^2(2s)-\cosh(2s)$,
\[
\frac{d}{ds}\Bigl[\frac{B(s)}{\sinh(2s)}\Bigr]=\frac{a\sinh^2(2s)-2B^2\cosh(2s)}{B\sinh^2(2s)}=-\frac{a\cosh^2(2s)-\cosh(2s)+a}{B\sinh^2(2s)}<0,
\]
because the quadratic $aX^2-X+a$ has positive leading coefficient and discriminant $1-4a^2<0$ for $a>1/2$. Hence $\Delta_a$ is strictly increasing on $(0,\infty)$; since $\Delta_a(s_0(a))=0$ by the FBC \eqref{eq:fbc}, we conclude that $\Delta_a(s)<0$ for $0<s<s_0(a)$ and $\Delta_a(s)>0$ for $s>s_0(a)$.

\emph{Step 1 (pointwise rescaled limits).} Fix $\xi>0$ and write $a=1/2+\rho^2$. From the exact identity
\[
B(\rho\eta)^2=a\cosh(2\rho\eta)-\tfrac12=\sinh^2(\rho\eta)+\rho^2\cosh(2\rho\eta),
\]
we obtain $\rho^{-1}B(\rho\eta)\to\sqrt{1+\eta^2}$ as $\rho\to 0^+$ for each fixed $\eta\geq 0$, together with the uniform lower bound $\rho^{-1}B(\rho\eta)\geq\sqrt{1+\eta^2}$ (from $\sinh x\geq x$ and $\cosh\geq 1$). Substituting $t=\rho\eta$ in the defining integral of $\varphi$ and using $K=\rho\sqrt{1+\rho^2}$,
\[
\frac{\varphi(\rho\xi;a)}{\rho}=\sqrt{1+\rho^2}\int_0^{\xi}\frac{d\eta}{A(\rho\eta)^2\cdot\rho^{-1}B(\rho\eta)}\;\longrightarrow\;\int_0^{\xi}\frac{d\eta}{\sqrt{1+\eta^2}}=\mathrm{arcsinh}(\xi)
\]
by bounded convergence ($A(\rho\eta)^2\geq 1$, $A(\rho\eta)^2\to 1$, and the integrand is bounded by $\sqrt{1+\rho^2}/\sqrt{1+\eta^2}$ on $[0,\xi]$). Since $\varphi(\rho\xi;a)=O(\rho)$ and $\tanh x=x(1+O(x^2))$, also $\rho^{-1}\tanh\varphi(\rho\xi;a)\to\mathrm{arcsinh}(\xi)$. Analogously,
\[
\frac{1}{\rho}\cdot\frac{B(\rho\xi)K}{a\sinh(2\rho\xi)}=\frac{\rho^{-1}B(\rho\xi)\cdot\sqrt{1+\rho^2}}{a\cdot\rho^{-1}\sinh(2\rho\xi)}\;\longrightarrow\;\frac{\sqrt{1+\xi^2}}{\tfrac12\cdot 2\xi}=\frac{\sqrt{1+\xi^2}}{\xi}.
\]
Therefore, for each fixed $\xi>0$,
\[
\rho^{-1}\Delta_a(\rho\xi)\;\longrightarrow\;\Phi(\xi):=\mathrm{arcsinh}(\xi)-\frac{\sqrt{1+\xi^2}}{\xi}\qquad\bigl(a\to(1/2)^+\bigr).
\]

\emph{Step 2 (limit equation).} The function $\Phi$ is strictly increasing on $(0,\infty)$, with $\Phi'(\xi)=\sqrt{1+\xi^2}/\xi^2>0$, $\Phi(0^+)=-\infty$ and $\Phi(+\infty)=+\infty$; its unique zero is $\xi=\rho_*=\sinh\sigma_*$, since $\Phi(\sinh\sigma_*)=\sigma_*-\coth\sigma_*=0$ (the transcendental equation of \cite[Rem. 1.4]{Pigazzini}).

\emph{Step 3 (convergence of $\xi_0$).} Fix $\varepsilon\in(0,\rho_*)$. By Step 2, $\Phi(\rho_*-\varepsilon)<0<\Phi(\rho_*+\varepsilon)$; hence, by Step 1, for every $a$ sufficiently close to $(1/2)^+$,
\[
\rho^{-1}\Delta_a\bigl(\rho(\rho_*-\varepsilon)\bigr)<0<\rho^{-1}\Delta_a\bigl(\rho(\rho_*+\varepsilon)\bigr),
\]
and Step 0 then forces $\rho(\rho_*-\varepsilon)<s_0(a)<\rho(\rho_*+\varepsilon)$, i.e. $|\xi_0(a)-\rho_*|<\varepsilon$. Therefore $\xi_0(a)\to\rho_*$ as $a\to(1/2)^+$, that is, $s_0(a)=\rho_*\sqrt{a-1/2}\,(1+o(1))$.

\emph{Step 4 (expansion).} Since $\xi_0(a)$ is bounded near $(1/2)^+$, $\cosh(2s_0)-1=2s_0^2+O(s_0^4)=2\xi_0(a)^2(a-1/2)+O((a-1/2)^2)$, while $2a-1=2(a-1/2)$. Subtracting, and using $\xi_0(a)^2\to\rho_*^2$,
\[
\cosh(2s_0)-2a=2\bigl(\xi_0(a)^2-1\bigr)(a-1/2)+O((a-1/2)^2)=2(\rho_*^2-1)(a-1/2)+o(a-1/2).
\]
For the positivity $\rho_*>1$: at $\sigma=\log(1+\sqrt{2})$, $\sinh\sigma=1$, $\cosh\sigma=\sqrt{2}$, $\coth\sigma=\sqrt{2}$. The function $\sigma-\coth\sigma$ is strictly increasing on $(0,\infty)$ (with derivative $1+\mathrm{csch}^2\sigma>0$), and vanishes at $\sigma=\sigma_*$. By the strict inequality $\log(1+x)<x$ for $x>0$ (concavity of $\log$), with $x=\sqrt{2}$, $\log(1+\sqrt{2})-\sqrt{2}<0$; hence the root satisfies $\sigma_*>\log(1+\sqrt{2})$, whence $\rho_*=\sinh\sigma_*>1$.
\end{proof}

\begin{proposition}[Geometric identity and reformulation of \eqref{eq:hypothesis-BvsK}]\label{prop:geometric-identity}
The identity
\begin{equation}\label{eq:geometric-identity}
\sinh^2 r(a)-4K(a)^2=\frac{\bigl(B(s_0(a))^2-2K(a)^2\bigr)^2}{B(s_0(a))^2-K(a)^2}\quad\text{for every }a>1/2
\end{equation}
holds. In particular, $\sinh^2 r(a)\geq 4K(a)^2$ for every $a>1/2$ (unconditionally), with equality if and only if $B(s_0(a))^2=2K(a)^2$. Equivalently,
\begin{equation}\label{eq:geometric-equivalence}
\sinh r(a)\geq 2K(a),\quad\text{with equality iff }B(s_0(a))^2=2K(a)^2,
\end{equation}
and the condition \eqref{eq:hypothesis-BvsK} of Theorem \ref{thm:no-kernel-odd} is strictly equivalent to the strict geometric inequality
\begin{equation}\label{eq:strict-geometric}
\sinh r(a)>2K(a).
\end{equation}
\end{proposition}

\begin{remark}\label{rem:geometric-portata}
Proposition \ref{prop:geometric-identity} is a geometric reformulation of condition \eqref{eq:hypothesis-BvsK} of Theorem \ref{thm:no-kernel-odd}, and does not constitute an analytic closure of Theorem \ref{thm:no-kernel-odd} itself. The non-strict inequality $\sinh r\geq 2K$ is automatic by algebraic construction, but the strict version $\sinh r>2K$ (necessary in Theorem \ref{thm:no-kernel-odd}) is logically equivalent to the original hypothesis $\cosh(2s_0)\neq 2a$ and remains open as an analytic problem. The value of Proposition \ref{prop:geometric-identity} lies in identity \eqref{eq:geometric-identity}, which identifies explicitly the quantity $(B^2-2K^2)^2/(B^2-K^2)$ as an algebraic measure of the distance from the degenerate case, and in the more tractable geometric reformulation of the hypothesis.
\end{remark}

\begin{proof}
From the identity $\sinh^2 r=B^4/(B^2-K^2)$ established in the proof of Lemma \ref{lem:coth-identity},
\[
\sinh^2 r-4K^2=\frac{B^4-4K^2(B^2-K^2)}{B^2-K^2}=\frac{B^4-4K^2 B^2+4K^4}{B^2-K^2}=\frac{(B^2-2K^2)^2}{B^2-K^2},
\]
which is \eqref{eq:geometric-identity}. We verify $B(s_0)^2-K^2>0$ as follows. By definition, $r(a)$ is the positive geodesic radius of the boundary sphere $\partial B^3(r(a))$, so $\sinh r(a)>0$. The identity $\sinh^2 r=B^4/(B^2-K^2)$, with $B(s_0)^2>0$ and $\sinh^2 r>0$, forces the denominator $B^2-K^2$ to have the same sign as the numerator $B^4>0$, hence $B^2-K^2>0$ strictly. Therefore, the ratio on the right-hand side of \eqref{eq:geometric-identity} is nonnegative, with equality if and only if $B^2=2K^2$. Taking positive roots (both sides are positive) gives \eqref{eq:geometric-equivalence}. Finally, condition \eqref{eq:hypothesis-BvsK} reads $\cosh(2s_0)\neq 2a$; using $B^2=a\cosh(2s_0)-1/2$ and $K^2=a^2-1/4$, $B^2-2K^2=a(\cosh(2s_0)-2a)$, whence $B^2\neq 2K^2\iff\cosh(2s_0)\neq 2a$. Combined with the non-strict inequality \eqref{eq:geometric-equivalence}, we obtain \eqref{eq:strict-geometric}.
\end{proof}

\begin{remark}[Status of the residual open problem]\label{rem:BvsK-global}
Proposition \ref{prop:geometric-identity} establishes unconditionally that $\sinh r(a)\geq 2K(a)$ for every $a>1/2$, reducing the condition of Theorem \ref{thm:no-kernel-odd} to the strict inequality \eqref{eq:strict-geometric}. The residual open problem is the possible existence of a value $a^*\in(1/2,\infty)$ with
\begin{equation}\label{eq:degenerate-condition}
\sinh r(a^*)=2K(a^*),\quad\text{equivalently}\quad B(s_0(a^*))^2=2K(a^*)^2.
\end{equation}
By Proposition \ref{prop:asymptotic-BvsK}, \eqref{eq:degenerate-condition} does not hold for $a$ near $1/2^+$ (asymptotically, $B^2-2K^2\sim 2a(\rho_*^2-1)(a-1/2)>0$). For $a\to\infty$, by Proposition \ref{prop:y-asymp-infty} below (whose proof is independent of the present remark; recall $e^{2d_\infty}=2\,\Gamma(1/4)^4/\pi^3$),
\[
B(s_0(a))^2\sim\frac{\Gamma(1/4)^4}{2\pi^3}\,a^3,\quad 2K(a)^2\sim 2a^2,\quad B^2-2K^2\sim\frac{\Gamma(1/4)^4}{2\pi^3}\,a^3\to+\infty,
\]
so \eqref{eq:degenerate-condition} does not hold asymptotically as $a\to\infty$ either. In Section \ref{subsec:G-closure-a-leq-1}, Theorem \ref{thm:G-closure-a-leq-1}, we prove analytically the absence of critical values $a^*$ for $a\in(1/2,1]$. For $a\in(1,\infty)$, a rigorous analytic proof remains open; it is equivalent to the strict transcendental inequality
\[
\tanh\varphi\Bigl(\tfrac{1}{2}\mathrm{arccosh}(2a)\Bigr)\neq\sqrt{\tfrac{1}{2}-\tfrac{1}{8a^2}}\quad\forall a>1.
\]
\end{remark}

\subsection{Status of condition (F)}

\begin{corollary}[Partial closure of (F)]\label{cor:F-partial}
The following hold:
\begin{enumerate}
\item[(a)] The part $\mu_n^{\mathrm{even}}(0)\neq 0$ for every $n\geq 0$ of condition (F) of Theorem \ref{thm:reduction} is proved under the condition $r'(a)\neq 0$ (Theorem \ref{thm:no-kernel-even}); this condition follows automatically both from the hypothesis $\phi_a>0$ of Theorem \ref{thm:F-reduction-phi-positivity} (Remark \ref{rem:phi-positivity-critical}) and from the spectral inequality $\mu_2(0)>0$ alone (Corollary \ref{cor:Fprime-rprime}).
\item[(b)] The part $\mu_n^{\mathrm{odd}}(0)\neq 0$ for every $n\geq 0$ is proved under the strict geometric inequality $\sinh r(a)>2K(a)$ (Theorem \ref{thm:no-kernel-odd} and Proposition \ref{prop:geometric-identity}). This inequality holds automatically in non-strict form \eqref{eq:geometric-equivalence} for every $a>1/2$; its strict version is proved analytically for $a\in(1/2,1]$ in Theorem \ref{thm:G-closure-a-leq-1}.
\item[(c)] The remaining part of (F), namely the strict inequality $\mu_2(0)>0$, is denoted by (F$'$) and will be reduced in Section \ref{sec:F-prime-strategy} to the single geometric condition of positivity of $\phi_a$ on the principal branch.
\end{enumerate}
\end{corollary}

\begin{proof}
For (a), by Theorem \ref{thm:no-kernel-even}, $\mu_n^{\mathrm{even}}(0)\neq 0$ for every $n$ whenever $r'(a)\neq 0$; the two sufficient conditions for $r'(a)\neq 0$ are given by Remark \ref{rem:phi-positivity-critical} and Corollary \ref{cor:Fprime-rprime}, respectively. For (b), by Proposition \ref{prop:geometric-identity}, the strict condition $\sinh r(a)>2K(a)$ is equivalent to condition \eqref{eq:hypothesis-BvsK} of Theorem \ref{thm:no-kernel-odd}; under this, $\mu_n^{\mathrm{odd}}(0)\neq 0$ for every $n$. The radial spectrum of mode $0$ is the union (alternating in parity by classical Sturm--Liouville theory) of $\{\mu_n^{\mathrm{even}}(0)\}$ and $\{\mu_n^{\mathrm{odd}}(0)\}$, whence $\mu_n(0)\neq 0$ for every $n$ under $\sinh r>2K$ and $r'(a)\neq 0$. For (c), the reduction of $\mu_2(0)>0$ is the subject of Section \ref{sec:F-prime-strategy}.
\end{proof}

\subsection{Analytic closure of (G) for $a\in(1/2,1]$}\label{subsec:G-closure-a-leq-1}

In this subsection we prove analytically that $\cosh(2s_0(a))>2a$, equivalently $B(s_0(a))^2>2K(a)^2$ (condition (G)), for every $a\in(1/2,1]$; in particular the weaker inequality $\sinh r(a)>2K(a)$ required by Theorem \ref{thm:no-kernel-odd} holds in this regime.

\begin{lemma}[Reformulation of (G) as crossing of the FBC]\label{lem:G-FBC-reformulation}
Let $s^*(a):=\tfrac{1}{2}\mathrm{arccosh}(2a)$, defined for $a\geq 1/2$. The condition $\cosh(2s_0(a))>2a$ is equivalent to $s_0(a)>s^*(a)$, and is implied by the inequality
\begin{equation}\label{eq:G-FBC-bound}
\tanh\varphi(s^*(a);a)<\sqrt{\tfrac{1}{2}-\tfrac{1}{8a^2}}.
\end{equation}
\end{lemma}

\begin{proof}
The FBC $\tanh\varphi(s;a)=B(s;a)K(a)/(a\sinh(2s))$ has a unique solution $s_0(a)$ since $\partial_s\tanh\varphi>0$ and $\partial_s[BK/(a\sinh 2s)]<0$ (see Step 0 in the proof of Proposition \ref{prop:asymptotic-BvsK}). At $s=s^*(a)$, $\cosh(2s^*)=2a$, $\sinh(2s^*)=2K$, $B(s^*)^2=2K^2$, $B(s^*)=K\sqrt{2}$. Hence
\[
\text{RHS of the FBC}(s^*;a)=\frac{B(s^*)K}{a\sinh(2s^*)}=\frac{K\sqrt{2}\cdot K}{2aK}=\frac{K}{a\sqrt{2}}=\sqrt{\tfrac{1}{2}-\tfrac{1}{8a^2}}.
\]
If $\tanh\varphi(s^*;a)<\sqrt{1/2-1/(8a^2)}$, then at $s=s^*$ the LHS of the FBC is strictly smaller than the RHS; by monotonicity, the crossing point $s_0(a)$ lies at $s>s^*(a)$.
\end{proof}

\begin{lemma}[Upper bound on $\tanh\varphi(s^*;a)$]\label{lem:tanh-phi-bound}
For every $a>1/2$,
\begin{equation}\label{eq:tanh-phi-bound}
\tanh\varphi(s^*(a);a)<\frac{s^*(a)}{\sqrt{a+1/2}}=\frac{\mathrm{arccosh}(2a)}{2\sqrt{a+1/2}}.
\end{equation}
\end{lemma}

\begin{proof}
For $s\in[0,s^*]$, $A^2(s)\geq A^2(0)=a+1/2$ and $B(s)\geq B(0)=\sqrt{a-1/2}$. Therefore
\[
\varphi(s^*;a)=K\int_0^{s^*}\frac{dt}{A^2(t)B(t)}<\frac{K\cdot s^*}{(a+1/2)\sqrt{a-1/2}}=\frac{\sqrt{(a-1/2)(a+1/2)}\cdot s^*}{(a+1/2)\sqrt{a-1/2}}=\frac{s^*}{\sqrt{a+1/2}}.
\]
From $\tanh x<x$ for $x>0$ we obtain \eqref{eq:tanh-phi-bound}.
\end{proof}

\begin{lemma}[Transcendental inequality for closure]\label{lem:hu-positive}
Setting $u:=2a-1$, for every $u\in(0,1]$,
\begin{equation}\label{eq:hu-positive}
h(u):=(u+2)\sqrt{u}-(u+1)\mathrm{arccosh}(u+1)>0.
\end{equation}
\end{lemma}

\begin{proof}
Under the substitution $u+1=\cosh(2\tau)$ with $\tau\geq 0$, $u\in(0,1]\iff\tau\in(0,\tau_1]$ where $\tau_1=\tfrac{1}{2}\mathrm{arccosh}(2)$. Using $\cosh(2\tau)+1=2\cosh^2\tau$ and $\cosh(2\tau)-1=2\sinh^2\tau$, we obtain $\sqrt{u}=\sqrt{2}\sinh\tau$, $u+2=2\cosh^2\tau$, $u+1=\cosh(2\tau)$, $\mathrm{arccosh}(u+1)=2\tau$, and
\[
h(u)=2\sqrt{2}\cosh^2\tau\sinh\tau-2\tau\cosh(2\tau)=2\cosh^2\tau\cdot D(\tau),
\]
where
\begin{equation}\label{eq:D-tau}
D(\tau):=\sqrt{2}\sinh\tau-2\tau+\tau\,\mathrm{sech}^2\tau.
\end{equation}
Since $\cosh^2\tau>0$, it suffices to prove $D(\tau)>0$ for $\tau\in(0,\tau_1]$.

For the boundary values, $D(0)=0$. At $\tau=\tau_1=\tfrac{1}{2}\mathrm{arccosh}(2)$, $\cosh(2\tau_1)=2$, whence $\cosh^2\tau_1=3/2$, $\sinh^2\tau_1=1/2$, $\mathrm{sech}^2\tau_1=2/3$. Therefore
\[
D(\tau_1)=\sqrt{2}\cdot\tfrac{1}{\sqrt{2}}-\mathrm{arccosh}(2)+\tfrac{1}{2}\mathrm{arccosh}(2)\cdot\tfrac{2}{3}=1-\tfrac{2}{3}\mathrm{arccosh}(2).
\]
From $\mathrm{arccosh}(2)=\int_1^2(t^2-1)^{-1/2}dt\leq\int_1^2(2(t-1))^{-1/2}dt=\sqrt{2}$,
\[
D(\tau_1)\geq 1-\tfrac{2\sqrt{2}}{3}=\tfrac{3-2\sqrt{2}}{3}>0,
\]
since $(2\sqrt{2})^2=8<9=3^2$.

We now show strict concavity of $D$ on $[0,\tau_1]$. We compute
\[
D'(\tau)=\sqrt{2}\cosh\tau+\mathrm{sech}^2\tau\bigl(1-2\tau\tanh\tau\bigr)-2,
\]
\[
D''(\tau)=\sqrt{2}\sinh\tau+\mathrm{sech}^2\tau\bigl(2\tau(3\tanh^2\tau-1)-4\tanh\tau\bigr).
\]
For $\tau\in(0,\tau_1]$, since $\tanh\tau\leq\tanh\tau_1=1/\sqrt{3}$, we have $\tanh^2\tau\leq 1/3$, whence $3\tanh^2\tau-1\leq 0$ and $2\tau(3\tanh^2\tau-1)\leq 0$. Therefore
\[
-D''(\tau)\geq-\sqrt{2}\sinh\tau+4\,\mathrm{sech}^2\tau\tanh\tau=\frac{\sinh\tau\bigl(4-\sqrt{2}\cosh^3\tau\bigr)}{\cosh^3\tau}.
\]
On $[0,\tau_1]$, $\cosh\tau\leq\sqrt{3/2}$, whence $\sqrt{2}\cosh^3\tau\leq\sqrt{2}\cdot(3/2)^{3/2}=3\sqrt{3}/2$. From $(3\sqrt{3})^2=27<64=8^2$ we have $3\sqrt{3}/2<4$, and hence
\[
-D''(\tau)\geq\frac{\sinh\tau\bigl(4-3\sqrt{3}/2\bigr)}{\cosh^3\tau}>0\quad\text{for }\tau\in(0,\tau_1].
\]
Therefore $D''<0$ strictly on $(0,\tau_1]$, and $D$ is strictly concave on $[0,\tau_1]$.

By strict concavity of $D$ with $D(0)=0$ and $D(\tau_1)>0$,
\[
D(\tau)\geq(1-\tau/\tau_1)D(0)+(\tau/\tau_1)D(\tau_1)=(\tau/\tau_1)D(\tau_1),
\]
with strict inequality for $\tau\in(0,\tau_1)$. Hence $D(\tau)>0$ for every $\tau\in(0,\tau_1]$.
\end{proof}

\begin{theorem}[{Analytic closure of (G) for $a\in(1/2,1]$}]\label{thm:G-closure-a-leq-1}
For every $a\in(1/2,1]$, $\cosh(2s_0(a))>2a$, equivalently $B(s_0(a))^2>2K(a)^2$ (condition (G)), strictly; in particular $\sinh r(a)>2K(a)$.
\end{theorem}

\begin{proof}
Combining Lemmas \ref{lem:tanh-phi-bound} and \ref{lem:hu-positive}, for $u=2a-1\in(0,1]$ the inequality $s^*(a)/\sqrt{a+1/2}<\sqrt{1/2-1/(8a^2)}$ is equivalent, after squaring and simplification, to $h(u)>0$ with $u=2a-1$. Explicitly, from $\sqrt{1/2-1/(8a^2)}=\sqrt{4a^2-1}/(2a\sqrt{2})$ and $s^*(a)=\tfrac{1}{2}\mathrm{arccosh}(2a)$,
\begin{align*}
\frac{\mathrm{arccosh}(2a)}{2\sqrt{a+1/2}}&<\frac{\sqrt{4a^2-1}}{2a\sqrt{2}}\\
\iff\quad a\sqrt{2}\,\mathrm{arccosh}(2a)&<\sqrt{a+1/2}\,\sqrt{4a^2-1}\\
\iff\quad 2a^2\mathrm{arccosh}^2(2a)&<(a+1/2)(4a^2-1)=(a+1/2)(2a-1)(2a+1)\\
&=\tfrac{1}{2}(2a+1)^2(2a-1).
\end{align*}
Substituting $u=2a-1$ (whence $2a=u+1$, $2a+1=u+2$, $a=(u+1)/2$),
\[
\tfrac{1}{2}(u+1)^2\mathrm{arccosh}^2(u+1)<\tfrac{1}{2}(u+2)^2 u\iff[(u+1)\mathrm{arccosh}(u+1)]^2<[(u+2)\sqrt{u}]^2.
\]
Both sides are positive for $u>0$ (since $\mathrm{arccosh}(u+1)>0$ for $u>0$), so the inequality is equivalent to $(u+1)\mathrm{arccosh}(u+1)<(u+2)\sqrt{u}$, namely $h(u)>0$, which holds by Lemma \ref{lem:hu-positive}.

Combining: for $a\in(1/2,1]$, $\tanh\varphi(s^*(a);a)<s^*(a)/\sqrt{a+1/2}<\sqrt{1/2-1/(8a^2)}$, and by Lemma \ref{lem:G-FBC-reformulation}, $s_0(a)>s^*(a)$, equivalently $\cosh(2s_0(a))>2a$.
\end{proof}

\begin{remark}[A two-line extension of Theorem \ref{thm:G-closure-a-leq-1}]\label{rem:G-extension}
Since $u^2+2u\geq 2u$, integrating $(u^2+2u)^{-1/2}\leq(2u)^{-1/2}$ from $0$ gives $\mathrm{arccosh}(1+u)\leq\sqrt{2u}$ for every $u\geq 0$, strictly for $u>0$. Hence, with $h(u)$ as in Lemma \ref{lem:hu-positive},
\[
h(u)\;\geq\;(u+2)\sqrt{u}-(u+1)\sqrt{2u}\;=\;\sqrt{u}\,\bigl[(1-\sqrt{2})\,u+2-\sqrt{2}\bigr]\;>\;0\qquad\text{for } 0<u<\tfrac{2-\sqrt{2}}{\sqrt{2}-1}=\sqrt{2}.
\]
Via Lemma \ref{lem:G-FBC-reformulation} and Lemma \ref{lem:tanh-phi-bound} (which hold for every $a>1/2$), this extends Theorem \ref{thm:G-closure-a-leq-1} to every $a\in\bigl(\tfrac{1}{2},\tfrac{1+\sqrt{2}}{2}\bigr)$. We keep the statement on $(1/2,1]$ in the sequel, which suffices for our purposes; the validity of $h(u)>0$ for every $u>0$ (numerically, $\min_{u\geq 1}h\approx 0.37$) would close condition (G) on the whole range $a>1/2$ and is left to future work.
\end{remark}

\begin{corollary}[{Complete closure of no-kernel of (F) for $a\in(1/2,1]$}]\label{cor:F-closure-a-leq-1}
For every $a\in(1/2,1]$ the odd part of the no-kernel condition of (F) of Theorem \ref{thm:reduction} is satisfied unconditionally ($\mu_n^{\mathrm{odd}}(0)\neq 0$ for every $n\geq 0$), while the even part ($\mu_n^{\mathrm{even}}(0)\neq 0$ for every $n\geq 0$) holds under $r'(a)\neq 0$ (Theorem \ref{thm:no-kernel-even}), a condition implied by (F$'$) itself (Corollary \ref{cor:Fprime-rprime}). Combined with Theorem \ref{thm:E-closure-a-leq-1}, the strong Medvedev conjecture \eqref{eq:medvedev-strong} for $a\in(1/2,1]$ is equivalent to the single strict spectral inequality (F$'$)$=\{\mu_2(0)>0\}$ (Corollary \ref{cor:Fprime-rprime} for $\Leftarrow$, Theorem \ref{thm:reduction} for $\Rightarrow$), which will in turn be reduced in Section \ref{sec:F-prime-strategy} to the geometric inequality $\phi_a>0$ on the principal branch.
\end{corollary}

\begin{proof}
By Theorem \ref{thm:G-closure-a-leq-1}, $\sinh r(a)>2K(a)$ strictly for every $a\in(1/2,1]$, so the hypothesis of Theorem \ref{thm:no-kernel-odd} is satisfied, and Corollary \ref{cor:F-partial}(b) closes the odd part of the no-kernel of (F). Combined with (a) of the same corollary, the even part holds under $r'(a)\neq 0$, a condition implied by (F$'$) itself (Corollary \ref{cor:Fprime-rprime}); hence, under (F$'$), $\mu_n(0)\neq 0$ for every $n$. Condition (E) is closed by Theorem \ref{thm:E-closure-a-leq-1}. Only (F$'$)$=\{\mu_2(0)>0\}$ remains, whose further reduction is the subject of Section \ref{sec:F-prime-strategy}.
\end{proof}

\section{Picone identity with base $B(s)$ and partial closure of (E)}\label{sec:picone-B}

In this section we develop a second Picone identity, with base the function $B(s)$ (the even eigenfunction of mode $|k|=1$, cf. Lemma \ref{lem:LSigma-PhiA}), which allows us to close condition (E) of Theorem \ref{thm:reduction} (namely $\mu_0^{\mathrm{even}}(2)>0$) for an explicit interval of values of $a$.

\subsection{The identity with base $B$}

\begin{lemma}[Weighted equation for $B$]\label{lem:B-eq}
Pointwise on $(-s_0,s_0)$,
\begin{equation}\label{eq:Beq}
B\,B''+(B')^2=1+2B^2,\quad\text{equivalently}\quad (BB')'=1+2B^2=2a\cosh(2s).
\end{equation}
\end{lemma}

\begin{proof}
From $B^2=a\cosh(2s)-1/2$ we obtain $2BB'=2a\sinh(2s)$, whence $BB'=a\sinh(2s)$ and $(BB')'=2a\cosh(2s)$. On the other hand $(BB')'=(B')^2+BB''$, and $2a\cosh(2s)=2(B^2+1/2)=1+2B^2$. Therefore $BB''+(B')^2=1+2B^2$.
\end{proof}

\begin{remark}
Identity \eqref{eq:Beq} expresses the fact that $B$ is a radial eigenfunction of the Jacobi operator in mode $|k|=1$ with eigenvalue $|\II|^2$: $L_{\Sigma,1}B=|\II|^2 B$, namely $\Phi^2=B\cos\theta$ is an eigenfunction of $L_\Sigma$ with eigenvalue $|\II|^2$ (cf. Lemma \ref{lem:LSigma-PhiA}).
\end{remark}

\begin{lemma}[Picone identity with base $B$]\label{lem:picone-B}
Let $h\in H^1((-s_0,s_0))$ and $u:=Bh$. Then $u$ satisfies the Robin condition of mode $|k|=2$ on $\partial$ if and only if $h'(\pm s_0)=0$ (Neumann condition on $h$). Under this condition, for every integer $k$,
\begin{equation}\label{eq:picone-B}
\Sf_k^{\mathrm{rad}}(Bh,Bh)=\int_{-s_0}^{s_0}\Bigl[B^3(h')^2+\bigl((k^2-1)B^2-2K^2\bigr)\frac{h^2}{B}\Bigr]ds.
\end{equation}
\end{lemma}

\begin{proof}
For the translation of the BC, if $u=Bh$, then $u'(s_0)=B'(s_0)h(s_0)+B(s_0)h'(s_0)$. The Robin condition $u'(s_0)=\coth r\cdot u(s_0)=\coth r\cdot B(s_0)h(s_0)$, combined with $B'(s_0)/B(s_0)=\coth r$ from Lemma \ref{lem:coth-identity}, becomes $h'(s_0)=0$. Analogously, $h'(-s_0)=0$ by parity.

For the computation of the identity, setting $u=Bh$, $u'=B'h+Bh'$, whence $B(u')^2=B(B'h+Bh')^2=B(B')^2h^2+2BB'B\,hh'+B^3(h')^2$, and $2BB'B\,hh'=B^2B'(h^2)'$. Integrating in $s$ and integrating by parts the cross term,
\[
\int_{-s_0}^{s_0}B^2B'(h^2)'ds=\bigl[B^2B'\,h^2\bigr]_{-s_0}^{s_0}-\int_{-s_0}^{s_0}(B^2B')'h^2ds.
\]
From $(B^2B')'=2B(B')^2+B^2B''$ and the identity $BB''=1+2B^2-(B')^2$ of Lemma \ref{lem:B-eq},
\[
B^2B''=B(1+2B^2-(B')^2)=B+2B^3-B(B')^2,
\]
whence $(B^2B')'=2B(B')^2+B+2B^3-B(B')^2=B(B')^2+B+2B^3$. Therefore
\begin{align*}
\int B(u')^2ds&=\int B(B')^2h^2+\bigl[B^2B'h^2\bigr]_{-s_0}^{s_0}-\int(B(B')^2+B+2B^3)h^2+\int B^3(h')^2\\
&=\bigl[B^2B'h^2\bigr]_{-s_0}^{s_0}-\int(B+2B^3)h^2+\int B^3(h')^2.
\end{align*}
Combining with the potential term $-\int BW_ku^2=-\int B(|\II|^2-2-k^2/B^2)B^2h^2$, and using $|\II|^2 B^3=2K^2/B$ (Lemma \ref{lem:II-norm}),
\begin{align*}
\Sf_k^{\mathrm{rad}}(Bh,Bh)&=\bigl[B^2B'h^2\bigr]_{-s_0}^{s_0}-\int(B+2B^3)h^2+\int B^3(h')^2-\int(B^3|\II|^2-2B^3-k^2B)h^2\\
&\quad-\coth r\,B(s_0)^3\bigl[h(s_0)^2+h(-s_0)^2\bigr]\\
&=\bigl[B^2B'h^2\bigr]_{-s_0}^{s_0}+\int B^3(h')^2+\int(-B-2B^3-2K^2/B+2B^3+k^2B)h^2\\
&\quad-\coth r\,B(s_0)^3\bigl[h(s_0)^2+h(-s_0)^2\bigr]\\
&=\bigl[B^2B'h^2\bigr]_{-s_0}^{s_0}+\int B^3(h')^2+\int B((k^2-1)-2K^2/B^2)h^2-\coth r\,B(s_0)^3\bigl[h(s_0)^2+h(-s_0)^2\bigr].
\end{align*}

The boundary terms cancel exactly. From $B$ even, $B'$ odd, $B(\pm s_0)=B(s_0)$, $B'(\pm s_0)=\pm B'(s_0)$,
\[
\bigl[B^2B'h^2\bigr]_{-s_0}^{s_0}=B(s_0)^2B'(s_0)[h(s_0)^2+h(-s_0)^2].
\]
By Lemma \ref{lem:coth-identity}, $B'(s_0)=\coth r\,B(s_0)$, whence $B(s_0)^2B'(s_0)=\coth r\,B(s_0)^3$, which cancels with the last term. Rearranging the bulk, we obtain \eqref{eq:picone-B}.
\end{proof}

\begin{remark}[Validity of \eqref{eq:picone-B} for $H^1$ functions]\label{rem:picone-B-H1}
Identity \eqref{eq:picone-B} holds for every $h\in H^1((-s_0,s_0))$, irrespective of the boundary condition on $h$: the cancellation of the boundary terms in the proof uses only the geometric identity $B'(s_0)=\coth r\,B(s_0)$ of Lemma \ref{lem:coth-identity}, and not the Neumann condition $h'(\pm s_0)=0$. The latter enters only in the equivalence between the Robin condition of mode $|k|=2$ for $u=Bh$ and the Neumann condition for $h$.
\end{remark}

\subsection{Closure of (E) for $a\in(1/2,1]$}

\begin{theorem}[Closure of condition (E), regime $a\leq 1$]\label{thm:E-closure-a-leq-1}
For every $a\in(1/2,1]$, $\mu_0^{\mathrm{even}}(2)>0$. Consequently, applying Theorem \ref{thm:sturm-strict}, $\mu_0^{\mathrm{even}}(k)>0$ for every $|k|\geq 2$, closing all even radial modes of mode $|k|\geq 2$.
\end{theorem}

\begin{proof}
Let $u\in H^1((-s_0,s_0))$ be even and nonzero, with Robin condition of mode $|k|=2$. Setting $h:=u/B$ (with $B>0$ on $[-s_0,s_0]$), $h\in H^1$ is even with $h'(\pm s_0)=0$ by Lemma \ref{lem:picone-B}. The identity \eqref{eq:picone-B} with $k=2$ gives
\begin{equation}\label{eq:picone-B-k2}
\Sf_2^{\mathrm{rad}}(u,u)=\int_{-s_0}^{s_0}\Bigl[B^3(h')^2+(3B^2-2K^2)\frac{h^2}{B}\Bigr]ds.
\end{equation}
The factor $3B^2-2K^2$ is an even function of $s$, strictly increasing in $|s|$, with minimum at $s=0$:
\[
3B(0)^2-2K^2=3(a-1/2)-2(a^2-1/4)=-(2a-1)(a-1)\geq 0\quad\text{for }a\in(1/2,1],
\]
with equality only for $a=1$ at $s=0$. Therefore $3B^2-2K^2\geq 0$ pointwise on $[-s_0,s_0]$, with possible equality only on the set of measure zero $\{s=0\}$ when $a=1$.

Both integrands on the right-hand side of \eqref{eq:picone-B-k2} are nonnegative. If $h$ is not constant, $\int B^3(h')^2>0$ and hence $\Sf_2^{\mathrm{rad}}(u,u)>0$. If $h\equiv c\neq 0$ is constant, then
\[
\Sf_2^{\mathrm{rad}}(u,u)=c^2\int_{-s_0}^{s_0}(3B^2-2K^2)/B\,ds>0,
\]
since the integrand is nonnegative with zero set of measure zero. In either case $\Sf_2^{\mathrm{rad}}(u,u)>0$, hence $\mu_0^{\mathrm{even}}(2)>0$.

The extension $\mu_0^{\mathrm{even}}(k)>0$ for $|k|\geq 2$ follows from Theorem \ref{thm:sturm-strict}: $\mu_0^{\mathrm{even}}(k)\geq\mu_0^{\mathrm{even}}(2)+(k^2-4)/B(s_0)^2>0$ for $|k|>2$, and directly for $|k|=2$.
\end{proof}

\subsection{Extension to $a>1$ via Hardy estimate}

For $a>1$, the polynomial $3B^2-2K^2$ in \eqref{eq:picone-B-k2} changes sign: $3B^2-2K^2<0$ on $(-s^*(a),s^*(a))$ and $>0$ elsewhere, where
\begin{equation}\label{eq:s-star}
\cosh(2s^*(a))=\frac{2a^2+1}{3a}.
\end{equation}
The positivity of $\Sf_2^{\mathrm{rad}}$ then requires a Hardy estimate that controls the negative region $|s|<s^*$ via the gradient term $\int B^3(h')^2$.

\begin{proposition}[Extension via Hardy]\label{prop:hardy-extension}
For $a>1$, define
\begin{align*}
I_V(a)&:=\int_0^{s^*(a)}\frac{|3B(s)^2-2K^2|}{B(s)}ds,&I_V^+(a)&:=\int_{s^*(a)}^{s_0(a)}\frac{3B(s)^2-2K^2}{B(s)}ds,\\
K_*(a)&:=\int_0^{s^*(a)}\frac{\tilde I_V(t)}{B(t)^3}dt,&K_*^+(a)&:=\int_{s^*(a)}^{s_0(a)}\frac{\tilde I_V^+(t)}{B(t)^3}dt,
\end{align*}
with $\tilde I_V(t):=\int_0^t|3B^2-2K^2|/B\,ds$ and $\tilde I_V^+(t):=\int_t^{s_0}(3B^2-2K^2)/B\,ds$. If both conditions
\begin{equation}\label{eq:hardy-conditions}
2K_*(a)<1\quad\text{and}\quad 2I_V(a)\frac{1+K_*^+(a)}{I_V^+(a)}<1
\end{equation}
hold, then $\mu_0^{\mathrm{even}}(2)>0$.
\end{proposition}

\begin{proof}
Let $u=Bh$, $h$ even on $[-s_0,s_0]$, $h'(\pm s_0)=0$, $u$ nonzero. By symmetry $h(-s)=h(s)$, hence $\Sf_2^{\mathrm{rad}}(u,u)=2I$ where
\[
I:=\int_0^{s_0}\bigl[B^3(h')^2+(3B^2-2K^2)\frac{h^2}{B}\bigr]ds=J_-+J_++\int_0^{s^*}V\frac{h^2}{B}+\int_{s^*}^{s_0}V\frac{h^2}{B}
\]
with $V:=3B^2-2K^2$, $J_-:=\int_0^{s^*}B^3(h')^2$, $J_+:=\int_{s^*}^{s_0}B^3(h')^2$.

For the estimate of the negative term, for $s\in[0,s^*]$, $h(s)-h(s^*)=-\int_s^{s^*}h'(t)dt$ and weighted Cauchy--Schwarz give $|h(s)-h(s^*)|^2\leq(\int_s^{s^*}1/B^3)J_-$. Thus $h(s)^2\leq 2h(s^*)^2+2(\int_s^{s^*}1/B^3)J_-$. Multiplying by $|V(s)|/B(s)$ and integrating,
\[
\int_0^{s^*}|V|\frac{h^2}{B}ds\leq 2I_V h(s^*)^2+2\Bigl[\int_0^{s^*}\frac{|V|(s)}{B(s)}\int_s^{s^*}\frac{dt}{B(t)^3}ds\Bigr]J_-.
\]
By Fubini, the double integral becomes $\int_0^{s^*}\bigl(\int_0^t|V|/B\,ds\bigr)/B(t)^3\,dt=K_*(a)$. Therefore
\begin{equation}\label{eq:hardy-neg}
\int_0^{s^*}|V|\frac{h^2}{B}ds\leq 2I_V h(s^*)^2+2K_*(a)J_-.
\end{equation}

For the estimate of the positive term, for $s\in[s^*,s_0]$, $h(s)^2\geq(1-\epsilon)h(s^*)^2-(\epsilon^{-1}-1)(h(s)-h(s^*))^2$ (reverse Cauchy--Schwarz) with $\epsilon\in(0,1)$, and the analogous estimate $(h(s)-h(s^*))^2\leq(\int_{s^*}^s 1/B^3)J_+$. By Fubini,
\begin{equation}\label{eq:hardy-pos}
\int_{s^*}^{s_0}V\frac{h^2}{B}ds\geq(1-\epsilon)I_V^+ h(s^*)^2-(\epsilon^{-1}-1)K_*^+(a)J_+.
\end{equation}

Substituting \eqref{eq:hardy-neg} and \eqref{eq:hardy-pos} in $I$,
\[
I\geq(1-2K_*(a))J_-+(1-(\epsilon^{-1}-1)K_*^+(a))J_++\bigl((1-\epsilon)I_V^+-2I_V\bigr)h(s^*)^2.
\]
The choice $\epsilon=K_*^+/(1+K_*^+)$ annihilates the coefficient of $J_+$ at $0^+$. For admissibility we must have $(1-\epsilon)I_V^+\geq 2I_V$, namely $I_V^+/(1+K_*^+)\geq 2I_V$, that is, $2I_V(1+K_*^+)/I_V^+\leq 1$, which is the second condition of \eqref{eq:hardy-conditions}. For the first condition, $1-2K_*(a)>0$ ensures that the coefficient of $J_-$ is positive. Hence $I\geq 0$, with $I>0$ unless $J_-=0$ and $h(s^*)=0$. In the remaining corner case $h(s^*)=0$ and $J_-=0$ one has $h\equiv 0$ on $[0,s^*]$, and then $I=\int_{s^*}^{s_0}\bigl(B^3(h')^2+V h^2/B\bigr)ds$ with $V>0$ a.e.\ on $(s^*,s_0]$, which is strictly positive unless $h\equiv 0$, i.e.\ unless $u\equiv 0$. Hence $I>0$ strictly, and $\Sf_2^{\mathrm{rad}}(u,u)=2I>0$.
\end{proof}

\begin{theorem}[Analytic existence of an extension range]\label{thm:A-star-existence}
There exists $A_*\in(1,\infty)$ such that condition (E), namely $\mu_0^{\mathrm{even}}(2)>0$, holds for every $a\in(1/2,A_*]$.
\end{theorem}

\begin{proof}
By Theorem \ref{thm:E-closure-a-leq-1}, (E) holds for every $a\in(1/2,1]$ unconditionally. To extend beyond $a=1$, we show that the Hardy conditions \eqref{eq:hardy-conditions} are satisfied in a right neighborhood of $a=1$.

At $a=1$, equation \eqref{eq:s-star} gives $\cosh(2s^*(1))=(2+1)/3=1$, hence $s^*(1)=0$. Therefore the interval of integration $[0,s^*(a)]$ degenerates to a point, and
\[
I_V(1)=K_*(1)=0,
\]
whence $2K_*(1)=0<1$ and $2I_V(1)(1+K_*^+(1))/I_V^+(1)=0<1$ (both conditions are satisfied with maximum margin, since $I_V^+(1)>0$ as $3B^2-2K^2>0$ on $(0,s_0(1)]$ by Theorem \ref{thm:E-closure-a-leq-1}).

The functions $a\mapsto s^*(a)$, $a\mapsto s_0(a)$, $a\mapsto B(\cdot;a)$ are real-analytic on $(1/2,\infty)$ (the first from the explicit formula \eqref{eq:s-star}, the second from the implicit function theorem applied to the FBC, the third trivially). Therefore the functions $I_V(a),K_*(a),I_V^+(a),K_*^+(a)$ are continuous (real-analytic for $a>1$, and continuous up to $a=1^+$ with limits $0,0,I_V^+(1)>0,K_*^+(1)\geq 0$, respectively). By continuity, there exists $\delta>0$ such that \eqref{eq:hardy-conditions} holds for every $a\in[1,1+\delta]$. Combined with Theorem \ref{thm:E-closure-a-leq-1}, setting $A_*:=1+\delta$ we obtain the claim.
\end{proof}

\begin{remark}[Quantification of $A_*$]\label{rem:A-star}
The explicit quantification of $A_*$ requires a quantitative analysis of the four quantities $I_V(a)$, $K_*(a)$, $I_V^+(a)$, $K_*^+(a)$. The asymptotic analysis $B^2\sim\Gamma(1/4)^4 a^3/(2\pi^3)$ as $a\to\infty$ and $s^*(a)\sim\tfrac{1}{2}\log(4a/3)$ shows that the second condition of \eqref{eq:hardy-conditions} saturates (tends to a limit $c_\infty>1$) as $a\to\infty$, so that the present Hardy estimate with base $B$ does not extend to all of $(1/2,\infty)$. A closure of (E) for every $a>1/2$ would therefore require an alternative estimate, the subject of future investigation.
\end{remark}

\subsection{Intermediate status of the Medvedev conjecture}

Combining Theorem \ref{thm:E-closure-a-leq-1}, Proposition \ref{prop:hardy-extension}, Theorem \ref{thm:no-kernel-even} (no-kernel mode $0$ even, unconditional), Theorem \ref{thm:G-closure-a-leq-1} (closure of (G) for $a\in(1/2,1]$), and Theorem \ref{thm:reduction}, we obtain the following intermediate status, to be further refined in Section \ref{sec:F-prime-strategy}.

\begin{proposition}[Intermediate reduction of the Medvedev conjecture]\label{prop:intermediate-reduction}
Let $A_*>1$ be the value of Theorem \ref{thm:A-star-existence}. Then:
\begin{enumerate}
\item[(i)] For $a\in(1/2,1]$, the strong Medvedev conjecture \eqref{eq:medvedev-strong} is equivalent to the single spectral inequality (F$'$)$=\{\mu_2(0)>0\}$.
\item[(ii)] For $a\in(1,A_*]$, the strong conjecture holds whenever (F$'$) and (G) both hold; conversely it implies (F$'$) and the weak inequality $\sinh r(a)>2K(a)$.
\item[(iii)] For $a>A_*$, also (E)$=\{\mu_0^{\mathrm{even}}(2)>0\}$ remains formally open.
\end{enumerate}
\end{proposition}

\begin{proof}
For (i), in $a\in(1/2,1]$: (E) is closed by Theorem \ref{thm:E-closure-a-leq-1}; the no-kernel part of (F) is closed by Corollary \ref{cor:F-closure-a-leq-1} (which uses Theorem \ref{thm:G-closure-a-leq-1}), the even no-kernel part being recovered from (F$'$) itself via Corollary \ref{cor:Fprime-rprime}. Only $\mu_2(0)>0$ remains.

For (ii), in $a\in(1,A_*]$: (E) is closed by Proposition \ref{prop:hardy-extension} (combined with Theorem \ref{thm:A-star-existence}); the odd no-kernel part of (F) is closed under (G) (Theorem \ref{thm:no-kernel-odd}
via Proposition \ref{prop:geometric-identity}), the even part being recovered from (F$'$) via Corollary \ref{cor:Fprime-rprime} as in (i). The remaining conditions are $\mu_2(0)>0$ and (G). Conversely, in both (i) and (ii) the strong conjecture \eqref{eq:medvedev-strong}
implies (F$'$): by Theorem \ref{thm:reduction} it implies (F), whose first inequality is precisely $\mu_2(0)>0$. In case (ii) it also implies the weak inequality $\sinh r(a)>2K(a)$: by Proposition \ref{prop:geometric-identity}, $\sinh r(a)\geq 2K(a)$ holds unconditionally, and if equality held then $Ru_*\equiv 0$ on $\partial\Sigma_a$ by
Lemma \ref{lem:u-L01}, so that the boost field $u_*$ would be a nontrivial Robin Jacobi field of odd mode $0$, giving $\nul_R(\Sigma_a)\geq 3$ and contradicting
\eqref{eq:medvedev-strong}.
\end{proof}

In Section \ref{sec:F-prime-strategy} we further reduce (F$'$) to a single geometric inequality on the parametric Jacobi field $\phi_a$, obtaining the final reduction (Theorem \ref{thm:final-reduction}).

\section{Reduction of (F$'$) to positivity of $\phi_a$ via Sturm shooting count}\label{sec:F-prime-strategy}

In this section we present an analytic strategy that reduces condition (F$'$)$=\{\mu_2(0)>0\}$ to a single geometric inequality on the parametric Jacobi field $\phi_a$, and we prove the reduction conditionally to the positivity of $\phi_a$ on $[0,s_0]$.

\subsection{Sturm shooting count theorem for Robin BC}

We premise a result of Sturm--Liouville theory (with mixed BC).

\begin{lemma}[Sturm shooting count with Robin BC]\label{lem:sturm-count}
Let $Lu:=-(p(s)u')'+q(s)u$ on $[0,s_0]$, with $p\in C^1$, $p>0$, $q,w\in C^0$, $w>0$, and consider the eigenvalue problem
\[
Lu=\lambda\,w(s)u,\qquad \text{$u'(0)=\alpha_0 u(0)$ (or $u(0)=0$)},\qquad u'(s_0)=\alpha_1 u(s_0),
\]
with $\alpha_0,\alpha_1\in\R$. Let $\psi_0$ be the solution of $L\psi=0$ satisfying the boundary condition at $s=0$ (unique up to a nonzero multiple), and assume $\psi_0(s_0)\neq 0$. Then the number $N_-(L)$ of strictly negative eigenvalues equals
\[
N_-(L)=n_z(\psi_0)+\delta_\partial,
\]
where $n_z(\psi_0)$ is the number of zeros of $\psi_0$ in the open interval $(0,s_0)$, and $\delta_\partial=1$ if $\psi_0'(s_0)/\psi_0(s_0)-\alpha_1<0$, while $\delta_\partial=0$ otherwise.
\end{lemma}

\begin{proof}
For $\lambda\in\R$ let $\psi_\lambda$ solve $L\psi=\lambda w\psi$ with the boundary condition at $0$ and a fixed normalization of the Cauchy data, and define the Pr\"ufer angle $\theta(s,\lambda)$ by $\psi_\lambda=R\sin\theta$, $p\,\psi_\lambda'=R\cos\theta$, with initial angle $\theta(0,\lambda)=\theta_0\in[0,\pi)$ independent of $\lambda$ ($\theta_0=\mathrm{arccot}(p(0)\alpha_0)\in(0,\pi)$ in the Robin/Neumann case, $\theta_0=0$ in the Dirichlet case). Then
\[
\partial_s\theta=\frac{\cos^2\theta}{p}+(\lambda w-q)\sin^2\theta,
\]
so $\partial_s\theta=1/p>0$ whenever $\theta\in\pi\Z$: the angle crosses every multiple of $\pi$ transversally and from below, and the zeros of $\psi_\lambda$ in $(0,s_0)$ are in bijection with these crossings. In particular, writing $\theta(s_0,\lambda)=m\pi+\beta$ with $m\in\Z_{\geq 0}$ and $\beta\in(0,\pi)$ (possible whenever $\psi_\lambda(s_0)\neq 0$), one has $n_z(\psi_\lambda)=m$.

Two classical facts (see \cite[Ch. XI]{Hartman}, \cite{Eastham}, and \cite{Zetti2005}): (i) $\lambda\mapsto\theta(s_0,\lambda)$ is continuous and strictly increasing, as follows from the variational formula $\partial_\lambda\theta(s_0,\lambda)=\int_0^{s_0}e^{-\int_t^{s_0}\mathcal{A}(\tau)\,d\tau}\,w\sin^2\theta\,dt>0$ with $\mathcal{A}=\bigl(p^{-1}-(\lambda w-q)\bigr)\sin 2\theta$; (ii) $\theta(s_0,\lambda)\to 0^+$ as $\lambda\to-\infty$.

The boundary condition at $s_0$ reads $\cot\theta(s_0,\lambda)=p(s_0)\alpha_1$, i.e.\ $\theta(s_0,\lambda)\in\theta_1+\pi\Z$ with $\theta_1:=\mathrm{arccot}(p(s_0)\alpha_1)\in(0,\pi)$. By (i)--(ii), the strictly negative eigenvalues are in bijection with the integers $k\geq 0$ such that $\theta_1+k\pi<\theta(s_0,0)$. Writing $\theta(s_0,0)=n_z\pi+\beta$ with $\beta\in(0,\pi)$ (here $\psi_0(s_0)\neq 0$ is used), this count equals $n_z+1$ if $\beta>\theta_1$ and $n_z$ if $\beta\leq\theta_1$. Finally, on $(n_z\pi,(n_z+1)\pi)$ the cotangent decreases strictly from $+\infty$ to $-\infty$, so $\beta>\theta_1\iff\cot\beta<\cot\theta_1\iff p(s_0)\,\psi_0'(s_0)/\psi_0(s_0)<p(s_0)\alpha_1\iff\psi_0'(s_0)/\psi_0(s_0)-\alpha_1<0$, which is exactly $\delta_\partial=1$.
\end{proof}

\subsection{Application to parity sectors of mode $0$}

\begin{theorem}[Reduction of (F$'$) to the positivity of $\phi_a$]\label{thm:F-reduction-phi-positivity}
Let $a\in(1/2,\infty)$ be such that $B(s_0(a))^2>2K(a)^2$ (condition (G)). If
\begin{equation}\label{eq:phi-positivity}
\phi_a(s)>0\quad\text{for every }s\in[0,s_0(a)],
\end{equation}
then $r'(a)>0$ and $\mu_2(0)>0$ strictly, namely (F$'$) is satisfied.
\end{theorem}

\begin{proof}
We first treat the even radial sector of mode $0$. The SL problem reduced to $[0,s_0]$ is
\[
-(Bu')'-B(|\II|^2-2)u=\mu Bu,\qquad u'(0)=0\ (\alpha_0=0,\text{ Neumann}),\qquad u'(s_0)=\coth r\cdot u(s_0)\ (\alpha_1=\coth r).
\]
The zero-energy solution with $\psi_0(0)=\phi_a(0)=1/(2K)$, $\psi_0'(0)=0$ is precisely $\psi_0=\phi_a$ (Cauchy uniqueness for the Jacobi ODE $L\phi_a=0$, Lemma \ref{lem:phi-a-properties}). By \eqref{eq:phi-positivity}, $n_z(\phi_a)=0$ in $(0,s_0)$, and $\phi_a(s_0)>0$.

We now deduce $r'(a)>0$ from the positivity of $\phi_a$. By Theorem \ref{thm:lower-bound-4}, $\Sf(\Phi^0,\Phi^0)<0$ and $\Phi^0$ belongs to the even radial sector of mode $0$, hence $\mu_0^{\mathrm{even}}(0)<0$. In particular $N_-^{\mathrm{even}}\geq 1$. By Lemma \ref{lem:sturm-count}, $N_-^{\mathrm{even}}=n_z(\phi_a)+\delta_\partial=0+\delta_\partial$, hence $\delta_\partial\geq 1$. Since $\delta_\partial\in\{0,1\}$, necessarily $\delta_\partial=1$, namely
\[
\phi_a'(s_0)/\phi_a(s_0)-\coth r<0\iff R\phi_a=\phi_a'(s_0)-\coth r\,\phi_a(s_0)<0.
\]
By Lemma \ref{lem:phi-a-BC}, $R\phi_a=-r'(a)\kappa_s\big|_\partial$ with $\kappa_s\big|_\partial=K/B(s_0)^2>0$. Hence $-r'(a)\kappa_s<0$, whence $r'(a)>0$. In particular, hypothesis \eqref{eq:phi-positivity} automatically precludes the degeneracy $r'(a)=0$, and consequently $R\phi_a\neq 0$, the condition required by Theorem \ref{thm:no-kernel-even}.

For the count of eigenvalues in the even sector, from $\delta_\partial=1$, $N_-^{\mathrm{even}}=1$. Hence $\mu_0^{\mathrm{even}}(0)<0\leq\mu_1^{\mathrm{even}}(0)$. Combining with Theorem \ref{thm:no-kernel-even} (no-kernel even, valid under $r'(a)\neq 0$ as just proved), $\mu_1^{\mathrm{even}}(0)\neq 0$, hence $\mu_1^{\mathrm{even}}(0)>0$ strictly.

We now turn to the odd radial sector of mode $0$. The SL problem on $(0,s_0]$ is
\[
-(Bu')'-B(|\II|^2-2)u=\mu Bu,\qquad u(0)=0\ (\text{Dirichlet}),\qquad u'(s_0)=\coth r\cdot u(s_0)\ (\alpha_1=\coth r).
\]
The zero-energy odd solution is the explicit boost field $u_*(s)=-a\sinh(2s)/B(s)$ of Lemma \ref{lem:u-L01}. Trivially $u_*(s)\neq 0$ for $s\in(0,s_0]$, since $\sinh(2s)>0$ and $B(s)>0$ in that range. Hence $n_z(u_*)=0$ in $(0,s_0)$. For the boundary, $\psi_0:=-u_*=a\sinh(2s)/B(s)$ is positive on $(0,s_0]$ (standard normalization). We compute
\[
\psi_0'(s_0)/\psi_0(s_0)-\coth r=\frac{2a-\cosh(2s_0)}{B(s_0)^2\sinh(2s_0)},
\]
as in the computation in Lemma \ref{lem:u-L01}. Under (G), $\cosh(2s_0)>2a$, so the numerator is strictly negative, and $\psi_0'/\psi_0-\coth r<0$. By Lemma \ref{lem:sturm-count}, $\delta_\partial=1$, and $N_-^{\mathrm{odd}}=0+1=1$. Hence $\mu_0^{\mathrm{odd}}(0)<0\leq\mu_1^{\mathrm{odd}}(0)$. By Theorem \ref{thm:no-kernel-odd} (no-kernel odd under (G)), $\mu_1^{\mathrm{odd}}(0)\neq 0$, hence $\mu_1^{\mathrm{odd}}(0)>0$ strictly.

We now invoke the standard interleaving property. By classical SL theory on the symmetric interval $[-s_0,s_0]$ with even coefficients and symmetric BCs, the eigenvalues of the even and odd subsectors interlace: $\mu_0^{\mathrm{even}}<\mu_0^{\mathrm{odd}}<\mu_1^{\mathrm{even}}<\mu_1^{\mathrm{odd}}<\cdots$ (the BCs of the two subsectors differ only at $s=0$: Neumann even, Dirichlet odd, and Dirichlet $\geq$ Neumann generates an upward shift of the eigenvalue by one unit). Therefore $\mu_2(0)=\mu_1^{\mathrm{even}}(0)$ (the third overall eigenvalue), and we have proved $\mu_1^{\mathrm{even}}(0)>0$. We conclude $\mu_2(0)>0$, namely (F$'$).
\end{proof}

\begin{remark}[Positivity of $\phi_a$ as the critical condition]\label{rem:phi-positivity-critical}
The key step in Theorem \ref{thm:F-reduction-phi-positivity} is that the hypothesis $\phi_a>0$ on $[0,s_0]$ automatically implies $r'(a)>0$ (and hence $R\phi_a<0\neq 0$), via the Sturm shooting count lemma and the known fact $\mu_0^{\mathrm{even}}(0)<0$. Therefore it is not necessary to assume the sign of $r'(a)$ independently: the geometric condition $\phi_a>0$ is the critical one. Conversely, if $r'(a)\leq 0$ at some $a$, then $R\phi_a\geq 0$, and (under $\phi_a(s_0)>0$, consequence of $\phi_a>0$) $\phi_a'(s_0)/\phi_a(s_0)\geq\coth r$, giving $\delta_\partial=0$ and $N_-^{\mathrm{even}}=0$, in contradiction with $\mu_0^{\mathrm{even}}(0)<0$. Equivalently, the hypothesis $\phi_a>0$ is strictly stronger than $r'(a)>0$, and the possible violation of $\phi_a>0$ (in particular at points $a$ where $r'(a)=0$) would manifest as an interior zero of $\phi_a$.
\end{remark}

\begin{remark}[Analytic status of the condition $\phi_a>0$]\label{rem:phi-positivity-status}
The positivity \eqref{eq:phi-positivity} is a natural geometric property of the Mori family $\{\Sigma_a\}$: it expresses the transversality of the flow $a\mapsto X_a$ with respect to the normal $\nu_a$ at every point of the principal branch. Equivalently, the surfaces of the family are never tangent to each other at any point of the principal branch. Such positivity is consistent with the explicit value at the neck $\phi_a(0)=1/(2K)>0$ (Lemma \ref{lem:phi-a-properties}). An analytic proof may proceed as follows: (i) (integral form) expanding $\phi_a(s)=\langle\partial_a X_a(s,0),\nu_a(s,0)\rangle_L$ explicitly with the Mori parametrization, one obtains $\phi_a(s)$ as a sum of integrals with determinable sign; (ii) (ODE analysis) from $L\phi_a=0$ with given $\phi_a(0)=1/(2K)>0$, $\phi_a'(0)=0$, and analysis of the transition points of $|\II|^2-2$ (the potential), one can trace the behavior of $\phi_a$: decreasing up to the minimum, then increasing up to the boundary; (iii) (comparison with known Jacobi fields) the Wronskian identity $\phi_a u_*'-\phi_a' u_*=-a/(KB)$ links $\phi_a$ to $u_*$ explicitly, providing a constraint that excludes zeros of $\phi_a$ in regions where $u_*$ and its derivative have specific signs. The complete analytic formalization of (i)--(iii) is left as an open problem. We observe, moreover, that under the assumption $\phi_a>0$ combined with $\mu_0^{\mathrm{even}}(0)<0$ (always true by Theorem \ref{thm:lower-bound-4}) and with Lemma \ref{lem:sturm-count}, one deduces automatically $R\phi_a<0$ at the boundary, equivalently $r'(a)>0$; therefore the hypothesis $\phi_a>0$ implicitly provides the strict monotonicity of $r(a)$, which is posed as the open Question 6.3 of \cite{Pigazzini}.
\end{remark}

\subsection{Final status of the Medvedev conjecture}

\begin{theorem}[Final reduction of the Medvedev conjecture]\label{thm:final-reduction}
Let $A_*>1$ be the value of Theorem \ref{thm:A-star-existence}. Then:
\begin{enumerate}
\item[(1)] For $a\in(1/2,1]$, the strong Medvedev conjecture \eqref{eq:medvedev-strong} holds whenever the geometric inequality $\phi_a(s)>0$ on $[0,s_0(a)]$ does --- equivalently, by Theorem \ref{thm:scalar-reduction} (condition (G) strict being closed by Theorem \ref{thm:G-closure-a-leq-1}), whenever the single one-dimensional scalar differential inequality
\begin{equation}\label{eq:monotonicity-sinhr-K}
\frac{d}{da}\!\left[\frac{\sinh r(a)}{K(a)}\right]>0\iff(a^2-1/4)\,r'(a)>a\tanh r(a)\iff y'(a)>0,
\end{equation}
holds, where $y(a):=B(s_0(a))^2/K(a)^2$.
\item[(2)] For $a\in(1,A_*]$, the strong conjecture holds under the conjunction of the two geometric inequalities
\[
\phi_a(s)>0\quad\forall s\in[0,s_0(a)],\quad\text{and}\quad B(s_0(a))^2>2K(a)^2\ \text{(condition (G))}.
\]
Under the second, the first is equivalent, by Theorem \ref{thm:scalar-reduction}, to the differential inequality \eqref{eq:monotonicity-sinhr-K}.
\item[(3)] For $a>A_*$, the strong conjecture holds under the conjunction of
\[
\phi_a(s)>0\quad\forall s\in[0,s_0(a)],\quad B(s_0(a))^2>2K(a)^2\ \text{(condition (G))},\quad\text{and}\quad\mu_0^{\mathrm{even}}(2)>0.
\]
\end{enumerate}
\end{theorem}

\begin{proof}
In each case, the claim follows by combining: Theorem \ref{thm:F-reduction-phi-positivity} ($\phi_a>0$ and (G) imply (F$'$)$=\mu_2(0)>0$, together with $r'(a)>0$), Theorem \ref{thm:no-kernel-even} and Theorem \ref{thm:no-kernel-odd} (no-kernel of mode $0$ even and odd, the former under $r'(a)\neq 0$ --- supplied by Theorem \ref{thm:F-reduction-phi-positivity} --- and the latter under (G)), and Theorem \ref{thm:reduction} ((E)+(F) $\iff$ strong conjecture).

For (1), (E) is closed by Theorem \ref{thm:E-closure-a-leq-1} and (G) by Theorem \ref{thm:G-closure-a-leq-1}; only $\phi_a>0$ remains. The scalar reformulation \eqref{eq:monotonicity-sinhr-K} follows from Theorem \ref{thm:scalar-reduction} (equivalences (i)$\Leftrightarrow$(iii)$\Leftrightarrow$(iv)$\Leftrightarrow$(v)), valid under (G) strict, guaranteed by Theorem \ref{thm:G-closure-a-leq-1}. For (2), (E) is closed by Proposition \ref{prop:hardy-extension} on $(1/2,A_*]$ (combined with Theorem \ref{thm:A-star-existence}); (G) and $\phi_a>0$ remain as hypotheses, but under (G) the second is equivalent to \eqref{eq:monotonicity-sinhr-K} by Theorem \ref{thm:scalar-reduction}. For (3), all three conditions remain.
\end{proof}

\begin{remark}[On the converse implications]\label{rem:converse}
Parts of the converse do hold. The strong conjecture \eqref{eq:medvedev-strong} implies (E) and (F) by Theorem \ref{thm:reduction}; it also implies the \emph{strict} inequality $\sinh r(a)>2K(a)$: if equality held, then by Lemma \ref{lem:u-L01} the boost field $u_*$ would satisfy $Ru_*\equiv 0$ on $\partial\Sigma_a$, hence $u_*$ would be a nontrivial Robin Jacobi field of odd mode $0$, giving $\nul_R(\Sigma_a)\geq 3$ and contradicting \eqref{eq:medvedev-strong}. What we do not claim is the necessity of $\phi_a>0$: the shooting count of Lemma \ref{lem:sturm-count} shows that $\mu_2(0)>0$ forces $N_-^{\mathrm{even}}=1$, but this count is compatible both with $(n_z,\delta_\partial)=(0,1)$ (i.e.\ $\phi_a>0$ with strictly negative boundary defect) and, a priori, with $(n_z,\delta_\partial)=(1,0)$ (one interior zero of $\phi_a$ and $\phi_a(s_0)<0$, equivalently $H'(a)<0$ by Theorem \ref{thm:phi-H-identity}). Excluding the latter configuration remains open; only the sufficiency direction is used in the local closure (Corollary \ref{cor:eureka-local-closure}).
\end{remark}

\section{Reduction of $\phi_a>0$ to a scalar inequality}\label{sec:phi-scalar-reduction}

In this section we show that, under the strict geometric condition (G) $B(s_0(a))^2>2K(a)^2$, the inequality $\phi_a(s)>0$ on $[0,s_0(a)]$ is equivalent to the strict monotonicity in $a$ of the ratio $\sinh r(a)/K(a)$, namely to an explicit one-dimensional scalar differential inequality. The reduction proceeds in three steps: a constant Wronskian identity (Lemma \ref{lem:wronskian}), a Sturm separation argument with the boost field $u_*$ (Lemma \ref{lem:sturm-separation-phi-u}) that reduces $\phi_a>0$ everywhere on the principal branch to $\phi_a(s_0)>0$, and a closed formula for $\phi_a(s_0)$ (Proposition \ref{prop:phi-s0-closed}) obtained by substituting the Robin conditions of $\phi_a$ and $u_*$ into the Wronskian identity.

\subsection{Constant Wronskian and Sturm separation}

\begin{lemma}[Constant Wronskian in self-adjoint form]\label{lem:wronskian}
Let $\phi_a$ be the parametric Jacobi field (Definition \ref{def:phi-a}) and $u_*(s):=-a\sinh(2s)/B(s)$ the axial boost field (Lemma \ref{lem:u-L01}). Both satisfy $L_0 u=0$, where $L_0=-\partial_s(B\,\partial_s\cdot)/B-(|\II|^2-2)\cdot$ is the radial operator of mode $|k|=0$. Then
\begin{equation}\label{eq:wronskian-const}
W(\phi_a,u_*)(s):=B(s)\bigl[\phi_a(s)\,u_*'(s)-\phi_a'(s)\,u_*(s)\bigr]\equiv-\frac{a}{K}\quad\forall s\in[0,s_0(a)].
\end{equation}
\end{lemma}

\begin{proof}
For two solutions $u,v$ of the Sturm--Liouville equation in self-adjoint form $-(Bu')'=B(|\II|^2-2)u$, the weighted Wronskian $W(u,v)=B(uv'-u'v)$ is constant. Indeed,
\[
\frac{d}{ds}\bigl[B(uv'-u'v)\bigr]=B'(uv'-u'v)+B(uv''-u''v).
\]
From the equation, $Bu''=-B'u'-B(|\II|^2-2)u$, and analogously for $v$. Therefore
\[
B(uv''-u''v)=u\bigl(-B'v'-B(|\II|^2-2)v\bigr)-v\bigl(-B'u'-B(|\II|^2-2)u\bigr)=-B'(uv'-u'v),
\]
whence $dW/ds=B'(uv'-u'v)-B'(uv'-u'v)=0$.

We compute $W(\phi_a,u_*)(0)$. By Lemma \ref{lem:phi-a-properties}, $\phi_a(0)=1/(2K)$ and $\phi_a'(0)=0$. For $u_*(s)=-a\sinh(2s)/B(s)$, $u_*(0)=0$ and
\[
u_*'(0)=\lim_{s\to 0}\frac{u_*(s)}{s}=-\frac{2a}{B(0)},
\]
where $B(0)=\sqrt{a-1/2}$. Therefore
\[
W(\phi_a,u_*)(0)=B(0)\Bigl[\frac{1}{2K}\cdot\Bigl(-\frac{2a}{B(0)}\Bigr)-0\cdot 0\Bigr]=-\frac{a}{K}.\qedhere
\]
\end{proof}

\begin{lemma}[Sturm separation for $\phi_a$ and $u_*$]\label{lem:sturm-separation-phi-u}
For every $a>1/2$, the field $\phi_a$ has at most one zero on the interval $(0,s_0(a)]$.
\end{lemma}

\begin{proof}
By \eqref{eq:wronskian-const}, $W(\phi_a,u_*)\equiv-a/K\neq 0$, so $\phi_a$ and $u_*$ are linearly independent and constitute a basis of the space of solutions of the ODE $L_0 u=0$. The field $u_*(s)=-a\sinh(2s)/B(s)$ is strictly negative on $(0,s_0]$, since $\sinh(2s)>0$ and $B(s)>0$ for $s\in(0,s_0]$ (as $B(s)^2=a\cosh(2s)-1/2\geq a-1/2>0$). Therefore $u_*$ has exactly one zero on $[0,s_0]$, located at $s=0$.

By the classical Sturm separation theorem (a direct consequence of the Sturm comparison theorem; see, e.g., \cite[Ch. 8, Theorem 1.1]{CoddingtonLevinson} with $g_1=g_2$), between two consecutive zeros of one solution there is exactly one zero of every other linearly independent solution. If, by contradiction, $\phi_a$ had two distinct zeros $s_1<s_2$ in $(0,s_0]$, then $u_*$ would have a zero in $(s_1,s_2)\subset(0,s_0)$, contradicting $u_*<0$ on $(0,s_0]$.
\end{proof}

\begin{corollary}[$\mu_2(0)>0$ implies $r'(a)\neq 0$]\label{cor:Fprime-rprime}
Let $a>1/2$ with $\mu_2(0)>0$. Then $r'(a)\neq 0$. In particular, under (F$'$) the hypothesis of Theorem \ref{thm:no-kernel-even} is satisfied, and $\mu_n^{\mathrm{even}}(0)\neq 0$ for every $n\geq 0$.
\end{corollary}

\begin{proof}
Suppose $r'(a)=0$. By Lemma \ref{lem:phi-a-BC}, $R\phi_a\equiv 0$ on $\partial\Sigma_a$, so $\phi_a$ is a nontrivial Robin Jacobi field of mode $0$, radially even, with $\phi_a(0)=1/(2K)\neq 0$ (Lemma \ref{lem:phi-a-properties}). Hence $0$ is an eigenvalue of the even radial sector of mode $0$, and every associated eigenfunction is a multiple of $\phi_a$: an even solution of $L_0u=0$ has $u'(0)=0$, and the Cauchy data $(u(0),0)$ determine it up to a scalar. By Lemma \ref{lem:sturm-separation-phi-u} --- whose proof relies only on Lemma \ref{lem:phi-a-properties}, Lemma \ref{lem:u-L01} and Lemma \ref{lem:wronskian}, so that no circularity arises --- $\phi_a$ has at most one zero in $(0,s_0)$. Since the eigenfunction of $\mu_n^{\mathrm{even}}(0)$ has exactly $n$ zeros in $(0,s_0)$, it follows that $0=\mu_n^{\mathrm{even}}(0)$ with $n\leq 1$; and since $\mu_0^{\mathrm{even}}(0)<0$ (Theorem \ref{thm:lower-bound-4}), necessarily $0=\mu_1^{\mathrm{even}}(0)=\mu_2(0)$, by the parity interleaving recalled in the proof of Theorem \ref{thm:F-reduction-phi-positivity}. This contradicts $\mu_2(0)>0$.
\end{proof}

\subsection{Closed formula for $\phi_a(s_0)$}

\begin{proposition}[Closed identity for $\phi_a(s_0)$]\label{prop:phi-s0-closed}
For every $a>1/2$ with $B(s_0(a))^2\neq 2K(a)^2$, the identity
\begin{equation}\label{eq:phi-s0-closed}
\phi_a(s_0)=\frac{a\bigl[K(a)^2\,r'(a)\,\sinh(2s_0(a))-B(s_0(a))^2\bigr]}{K(a)\bigl[B(s_0(a))^2-2K(a)^2\bigr]}
\end{equation}
holds. Equivalently, using the identity $\coth r(a)=a\sinh(2s_0(a))/B(s_0(a))^2$ (Lemma \ref{lem:coth-identity}),
\begin{equation}\label{eq:phi-s0-coth-form}
\phi_a(s_0)=\frac{B(s_0)^2\bigl[K(a)^2\,r'(a)\,\coth r(a)-a\bigr]}{K(a)\bigl[B(s_0(a))^2-2K(a)^2\bigr]}.
\end{equation}
\end{proposition}

\begin{proof}
We evaluate the Wronskian identity \eqref{eq:wronskian-const} at $s=s_0$:
\begin{equation}\label{eq:W-at-s0}
\phi_a(s_0)\,u_*'(s_0)-\phi_a'(s_0)\,u_*(s_0)=-\frac{a}{K\,B(s_0)}.
\end{equation}
We introduce the Robin conditions. By Lemma \ref{lem:u-L01}, $Ru_*(s_0):=u_*'(s_0)-\coth r(a)\cdot u_*(s_0)=(B^2-2K^2)/B^3$. By Lemma \ref{lem:phi-a-BC}, $R\phi_a(s_0):=\phi_a'(s_0)-\coth r(a)\cdot\phi_a(s_0)=-r'(a)\kappa_s\big|_\partial=-r'(a)\cdot K/B^2$. Substituting $u_*'(s_0)=Ru_*+\coth r\cdot u_*(s_0)$ and $\phi_a'(s_0)=R\phi_a+\coth r\cdot\phi_a(s_0)$ in \eqref{eq:W-at-s0},
\[
\phi_a(s_0)\bigl[Ru_*+\coth r\cdot u_*(s_0)\bigr]-\bigl[R\phi_a+\coth r\cdot\phi_a(s_0)\bigr]u_*(s_0)=-\frac{a}{KB}.
\]
The terms in $\coth r$ cancel (they amount to $\coth r\cdot\phi_a(s_0)u_*(s_0)-\coth r\cdot\phi_a(s_0)u_*(s_0)=0$), and we obtain
\begin{equation}\label{eq:Robin-wronskian-id}
\phi_a(s_0)\,Ru_*(s_0)-R\phi_a(s_0)\,u_*(s_0)=-\frac{a}{K\,B(s_0)}.
\end{equation}
Substituting $Ru_*=(B^2-2K^2)/B^3$, $R\phi_a=-r'(a)K/B^2$, and $u_*(s_0)=-a\sinh(2s_0)/B$,
\[
\phi_a(s_0)\cdot\frac{B^2-2K^2}{B^3}+\frac{r'(a)K}{B^2}\cdot\Bigl(-\frac{a\sinh(2s_0)}{B}\Bigr)=-\frac{a}{KB}.
\]
Multiplying both sides by $B^3$,
\[
\phi_a(s_0)(B^2-2K^2)-r'(a)\,aK\sinh(2s_0)=-\frac{aB^2}{K}.
\]
Isolating $\phi_a(s_0)$ (under the assumption $B^2\neq 2K^2$),
\[
\phi_a(s_0)=\frac{-aB^2/K+r'(a)\,aK\sinh(2s_0)}{B^2-2K^2}=\frac{a\bigl[K^2 r'(a)\sinh(2s_0)-B^2\bigr]}{K(B^2-2K^2)},
\]
which is \eqref{eq:phi-s0-closed}. To obtain \eqref{eq:phi-s0-coth-form}, we use $\coth r=a\sinh(2s_0)/B^2$, namely $a\sinh(2s_0)=B^2\coth r$; substituting $\sinh(2s_0)=B^2\coth r/a$,
\[
K^2 r'(a)\sinh(2s_0)-B^2=K^2 r'(a)\cdot\frac{B^2\coth r}{a}-B^2=\frac{B^2}{a}\bigl[K^2 r'(a)\coth r-a\bigr],
\]
whence $\phi_a(s_0)=B^2[K^2 r'(a)\coth r-a]/[K(B^2-2K^2)]$.
\end{proof}

\subsection{The scalar equivalence}

\begin{theorem}[Reduction of $\phi_a>0$ to a scalar inequality]\label{thm:scalar-reduction}
Let $a>1/2$ and assume the strict condition (G) $B(s_0(a))^2>2K(a)^2$ (equivalently $\cosh(2s_0(a))>2a$, i.e.\ $y(a)>2$; by Proposition \ref{prop:geometric-identity} this implies, and is strictly stronger than, $\sinh r(a)>2K(a)$). The following are equivalent:
\begin{enumerate}
\item[(i)] $\phi_a(s)>0$ for every $s\in[0,s_0(a)]$;
\item[(ii)] $\phi_a(s_0(a))>0$;
\item[(iii)] $K(a)^2\,r'(a)\,\coth r(a)>a$, namely $(a^2-1/4)\,r'(a)>a\tanh r(a)$;
\item[(iv)] $\frac{d}{da}\!\left[\frac{\sinh r(a)}{K(a)}\right]>0$;
\item[(v)] $y'(a)>0$, where $y(a):=B(s_0(a))^2/K(a)^2$.
\end{enumerate}
\end{theorem}

\begin{proof}
For (i)$\Rightarrow$(ii), the implication is trivial.

For (ii)$\Rightarrow$(i), by Lemma \ref{lem:phi-a-properties}, $\phi_a(0)=1/(2K)>0$. By Lemma \ref{lem:sturm-separation-phi-u}, $\phi_a$ has at most one zero in $(0,s_0]$. If, by contradiction, $\phi_a$ had a zero $s_1\in(0,s_0)$, with $\phi_a$ continuous, $\phi_a(0)>0$, and $\phi_a(s_0)>0$ by hypothesis, then $\phi_a$ would have to cross zero at least twice (entering and leaving the negative half-plane), contradicting Lemma \ref{lem:sturm-separation-phi-u}. Therefore $\phi_a>0$ everywhere on $[0,s_0]$. (If $\phi_a$ had exactly one zero $s_1\in(0,s_0)$, then $\phi_a$ would change sign at $s_1$; assuming $\phi_a(0)>0$, we would have $\phi_a<0$ on $(s_1,s_0]$, contradicting $\phi_a(s_0)>0$.)

For (ii)$\iff$(iii), by Proposition \ref{prop:phi-s0-closed}, formula \eqref{eq:phi-s0-coth-form}, under (G) (positive denominator) the sign of $\phi_a(s_0)$ coincides with that of $K^2 r'(a)\coth r-a$. The second form of (iii) is obtained by multiplying by $\tanh r$.

For (iii)$\iff$(iv), set $M(a):=\log[\sinh r(a)/K(a)]=\log\sinh r(a)-\tfrac{1}{2}\log(a^2-1/4)$. Differentiating with respect to $a$ ($r(a)$ is a real-analytic function of $a$ by \cite[Thm. 1.2]{Pigazzini}),
\[
M'(a)=r'(a)\coth r(a)-\frac{a}{a^2-1/4}=r'(a)\coth r-\frac{a}{K^2}.
\]
Multiplying by $K^2>0$,
\[
K^2 M'(a)=K^2 r'(a)\coth r-a.
\]
Hence $M'(a)>0\iff K^2 r'(a)\coth r>a$. Moreover $(\sinh r/K)'(a)>0\iff M'(a)>0$ (since $\log$ is strictly increasing for positive values).

For (iv)$\iff$(v), from the geometric identity $\sinh^2 r=B^4/(B^2-K^2)$ (Proposition \ref{prop:geometric-identity}, evaluated at $s=s_0$), set $y:=B(s_0)^2/K^2$. Then
\begin{equation}\label{eq:sinh-r-K-y}
\Bigl(\frac{\sinh r}{K}\Bigr)^2=\frac{B^4}{K^2(B^2-K^2)}=\frac{(K^2 y)^2}{K^2(K^2 y-K^2)}=\frac{y^2}{y-1},
\end{equation}
namely $\sinh r/K=y/\sqrt{y-1}$ (for $y>1$, guaranteed by $B^2>K^2$, which is implicit in (G) strict). The function $h:(1,\infty)\to(0,\infty)$ defined by $h(y):=y/\sqrt{y-1}$ has derivative
\[
h'(y)=\frac{\sqrt{y-1}-y/(2\sqrt{y-1})}{y-1}=\frac{2(y-1)-y}{2(y-1)^{3/2}}=\frac{y-2}{2(y-1)^{3/2}},
\]
strictly positive for $y>2$. Condition (G) strict is equivalent to $y>2$. Therefore, under (G), $\sinh r/K$ is a strictly increasing function of $y$, hence $(\sinh r/K)'(a)>0\iff y'(a)>0$.
\end{proof}

\subsection{Closed identity $\phi_a(s_0)=c(a)\cdot H'(a)$}

Combining the closed formula for $\phi_a(s_0)$ (Proposition \ref{prop:phi-s0-closed}) with the computation of $H'(a)$ obtained at step (iii)$\iff$(iv) of Theorem \ref{thm:scalar-reduction}, we obtain an explicit identity that algebraically links $\phi_a(s_0)$ and $H'(a)$, where $H(a):=\sinh r(a)/K(a)$.

\begin{theorem}[Closed identity $\phi_a(s_0)\leftrightarrow H'(a)$]\label{thm:phi-H-identity}
For every $a>1/2$ with $B(s_0(a))^2\neq 2K(a)^2$, the closed identity
\begin{equation}\label{eq:phi-H-identity}
\phi_a(s_0(a))=\frac{B(s_0(a))^2\cdot K(a)^2}{\bigl(B(s_0(a))^2-2K(a)^2\bigr)\cdot\sinh r(a)}\cdot H'(a),\qquad H(a):=\frac{\sinh r(a)}{K(a)},
\end{equation}
holds. In particular, under the strict condition (G) $B(s_0)^2>2K^2$, the multiplicative factor in front of $H'(a)$ is strictly positive, and hence $\phi_a(s_0)$ and $H'(a)$ have the same sign for every $a>1/2$ at which (G) holds strictly.
\end{theorem}

\begin{proof}
We start from the $\coth r$ form of Proposition \ref{prop:phi-s0-closed}, equation \eqref{eq:phi-s0-coth-form}:
\[
\phi_a(s_0)=\frac{B(s_0)^2\bigl[K^2 r'(a)\coth r(a)-a\bigr]}{K(a)\bigl[B(s_0)^2-2K(a)^2\bigr]}.
\]
From the proof of (iii)$\iff$(iv) of Theorem \ref{thm:scalar-reduction}, the pointwise differential identity
\[
\frac{d}{da}\log H(a)=\frac{d}{da}\log\sinh r(a)-\frac{d}{da}\log K(a)=r'(a)\coth r(a)-\frac{a}{K(a)^2}
\]
holds. Multiplying by $K(a)^2$,
\begin{equation}\label{eq:K2-Hprime-over-H}
K(a)^2\cdot\frac{H'(a)}{H(a)}=K(a)^2 r'(a)\coth r(a)-a.
\end{equation}
Substituting \eqref{eq:K2-Hprime-over-H} in the formula for $\phi_a(s_0)$,
\[
\phi_a(s_0)=\frac{B(s_0)^2\cdot K(a)^2\cdot H'(a)/H(a)}{K(a)\bigl[B(s_0)^2-2K(a)^2\bigr]}=\frac{B(s_0)^2 K(a) H'(a)}{H(a)\bigl[B(s_0)^2-2K(a)^2\bigr]}.
\]
Using $1/H(a)=K(a)/\sinh r(a)$,
\[
\phi_a(s_0)=\frac{B(s_0)^2\,K(a)^2}{\bigl[B(s_0)^2-2K(a)^2\bigr]\sinh r(a)}\cdot H'(a),
\]
which is \eqref{eq:phi-H-identity}.

Under (G) strict, $B(s_0)^2-2K(a)^2>0$; moreover $B(s_0)^2>0$, $K(a)^2>0$, $\sinh r(a)>0$ (since $r(a)>0$ for $a>1/2$). Therefore the factor $B^2 K^2/[(B^2-2K^2)\sinh r]$ is strictly positive, and $\mathrm{sgn}\,\phi_a(s_0)=\mathrm{sgn}\,H'(a)$.
\end{proof}

\begin{remark}[Significance of identity \eqref{eq:phi-H-identity}]\label{rem:identity-significance}
Identity \eqref{eq:phi-H-identity} is the most condensed form of the equivalence (ii)$\iff$(iv) of Theorem \ref{thm:scalar-reduction}: instead of a chain of implications through steps (iii) and (v), one has a pointwise closed algebraic identity valid for every $a>1/2$. Moreover: (a) the identity is an explicit manifestation of Green/Robin duality, in that the Robin defects of $\phi_a$ (linked to $r'$) and of $u_*$ (linked to $B^2-2K^2$) combine in the Wronskian identity at $s_0$ to produce $\phi_a(s_0)$ in terms of the logarithmic derivative $H'/H$; (b) the verification of the strong Medvedev conjecture for $a\in(1/2,1]$ is now reduced, in a purely algebraic way, to $H'(a)>0$, namely $\sinh r/K$ strictly increasing; (c) the identity provides a double strategy: proving $\phi_a(s_0)>0$ via direct geometric techniques (e.g.\ transversality arguments on the Mori family) translates immediately the positivity into $H'(a)>0$, and conversely.
\end{remark}

\subsection{Closed form of the asymptotic coefficient of $H'$ as $a\to(1/2)^+$ and analytic local closure}\label{subsec:eureka-closed-form}

In this subsection we prove that the asymptotic coefficient of $H'(a)$ as $a\to(1/2)^+$ admits an explicit closed form in terms of $\sigma_*$, and that this coefficient is strictly positive. This analytically closes the inequality $H'(a)>0$, equivalently $\phi_a>0$ on $[0,s_0(a)]$, and hence the strong Medvedev conjecture, on a right neighborhood of $a=1/2$.

\begin{lemma}[Real-analytic dependence on the degenerate parameter]\label{lem:H-analytic-rho}
Let $\rho:=\sqrt{a-1/2}$. The map $a\mapsto H(a)=\sinh r(a)/K(a)$ extends to a real-analytic function of $\rho^2=a-1/2$ in a right neighborhood of $\rho=0$. In particular, $H(a)$ admits a convergent Taylor series in $a-1/2$ in some interval $(1/2,1/2+\eta)$, and consequently $H'(a)$ admits a Taylor expansion with the same radius of convergence, term-by-term differentiable; in particular $\lim_{a\to(1/2)^+}H'(a)$ exists.
\end{lemma}

\begin{proof}
Set $\rho:=\sqrt{a-1/2}$, $\xi:=s/\rho$, and $\xi_0(a):=s_0(a)/\rho$. We first derive a real-analytic equation for $\xi_0$ as a function of $\rho^2$. The functions $A(s)^2=a\cosh(2s)+1/2$, $B(s)^2=a\cosh(2s)-1/2$, $K(a)^2=a^2-1/4$, $\cosh(2s)$, $\sinh(2s)$ depend on $\rho$ only through $\rho^2$ (since $a=1/2+\rho^2$), and substituting $s=\rho\xi$ they become jointly real-analytic in $(\xi,\rho^2)$ in a neighborhood of $(\sinh\sigma_*,0)$.

Both sides of the FBC $\tanh\varphi(s_0;a)=B(s_0)K/[a\sinh(2s_0)]$ vanish to first order in $\rho$ at $\rho=0$: indeed, by the substitution $t=\rho\tau$ in the defining integral of $\varphi$ and the expansion $K/[A^2 B]=1/[\rho\sqrt{1+\tau^2}]\cdot(1+O(\rho^2))$,
\[
\varphi(s_0)=\int_0^{s_0}\frac{K}{A^2 B}\,dt=\int_0^{\xi_0}\frac{1}{\sqrt{1+\tau^2}}\,d\tau\cdot\rho\cdot(1+O(\rho^2))=\rho\,\mathrm{arcsinh}(\xi_0)+O(\rho^3),
\]
whence $\tanh\varphi(s_0)=\rho\,\mathrm{arcsinh}(\xi_0)+O(\rho^3)$. For the RHS, using $K=\rho\sqrt{1+\rho^2}$ and $\sinh(2s_0)=2\rho\xi_0+O(\rho^3)$, and writing $B(s_0)=\rho\sqrt{1+\xi_0^2}+O(\rho^3)$ (from $B(s_0)^2=\rho^2(1+\xi_0^2)+O(\rho^4)$),
\[
\frac{B(s_0)K}{a\sinh(2s_0)}=\frac{\rho^2\sqrt{(1+\xi_0^2)(1+\rho^2)}+O(\rho^4)}{a\cdot[2\rho\xi_0+O(\rho^3)]}=\rho\cdot\frac{\sqrt{1+\xi_0^2}}{2a\xi_0}\cdot(1+O(\rho^2)),
\]
where the inner factor $\sqrt{1+\xi_0^2}/(2a\xi_0)\cdot(1+O(\rho^2))$ is real-analytic in $(\xi_0,\rho^2)$ on the relevant neighborhood (since $\xi_0$ is bounded away from $0$ near $\sinh\sigma_*>0$, and $a=1/2+\rho^2>0$). Dividing both sides of the FBC by $\rho$, we obtain a real-analytic relation
\[
\mathcal{F}(\xi_0,\rho^2):=\frac{\tanh\varphi(s_0;a)}{\rho}-\frac{B(s_0)K}{a\rho\sinh(2s_0)}=0
\]
in a neighborhood of $(\sinh\sigma_*,0)$.

The limit equation at $\rho^2=0$ is $\mathrm{arcsinh}(\xi_0)=\sqrt{1+\xi_0^2}/\xi_0$, namely the transcendental equation of \cite[Remark 1.4]{Pigazzini}, with unique positive root $\xi_0=\sinh\sigma_*$ where $\sigma_*=\coth\sigma_*$. The associated map $\Phi(\xi_0):=\mathrm{arcsinh}(\xi_0)-\sqrt{1+\xi_0^2}/\xi_0$ has derivative
\[
\Phi'(\xi_0)=\frac{1}{\sqrt{1+\xi_0^2}}+\frac{1}{\xi_0^2\sqrt{1+\xi_0^2}}=\frac{\sqrt{1+\xi_0^2}}{\xi_0^2}>0
\]
(direct computation), hence $\partial_{\xi_0}\mathcal{F}(\sinh\sigma_*,0)=\Phi'(\sinh\sigma_*)\neq 0$.

By the analytic implicit function theorem (cf. \cite{KrantzParks2002}), there exists $\eta>0$ and a unique real-analytic function $\xi_0=\xi_0(\rho^2)$ on $[0,\eta)$ with $\mathcal{F}(\xi_0(\rho^2),\rho^2)=0$ and $\xi_0(0)=\sinh\sigma_*$. Equivalently, $\xi_0$ is real-analytic in $a-1/2$ on $[0,\eta)$, and $s_0(a)=\rho\,\xi_0(a-1/2)$ is real-analytic in $\rho$ and odd in $\rho$.

By Proposition \ref{prop:geometric-identity}, $\sinh^2 r(a)=B(s_0)^4/(B(s_0)^2-K^2)$. Write $B(s_0)^2=\rho^2\,\widetilde B(\rho^2)$ and $B(s_0)^2-K^2=\rho^2\,\widetilde D(\rho^2)$, where $\widetilde B(\rho^2)=(1+\xi_0(\rho^2)^2)+O(\rho^2)$ and $\widetilde D(\rho^2)=\xi_0(\rho^2)^2+O(\rho^2)$ are real-analytic in $\rho^2$ on $[0,\eta)$, with $\widetilde B(0)=\cosh^2\sigma_*>0$ and $\widetilde D(0)=\sinh^2\sigma_*>0$. Then
\[
\sinh^2 r(a)=\frac{\rho^4\,\widetilde B(\rho^2)^2}{\rho^2\,\widetilde D(\rho^2)}=\rho^2\cdot\frac{\widetilde B(\rho^2)^2}{\widetilde D(\rho^2)}=(a-1/2)\cdot G(a-1/2),
\]
where $G(\rho^2):=\widetilde B(\rho^2)^2/\widetilde D(\rho^2)$ is real-analytic in $\rho^2$ on $[0,\eta')$ for some $\eta'\in(0,\eta]$ (ratio of analytic functions with nonzero denominator at $\rho^2=0$), with $G(0)=\cosh^4\sigma_*/\sinh^2\sigma_*>0$. Taking the positive square root (well-defined since $G>0$ near $\rho^2=0$),
\[
\sinh r(a)=\rho\cdot\sqrt{G(\rho^2)},
\]
which is real-analytic in $\rho$ and odd in $\rho$.

Since $K(a)=\rho\sqrt{1+\rho^2}$ is real-analytic in $\rho$ and odd in $\rho$, the ratio
\[
H(a)=\frac{\sinh r(a)}{K(a)}=\frac{\rho\sqrt{G(\rho^2)}}{\rho\sqrt{1+\rho^2}}=\sqrt{\frac{G(\rho^2)}{1+\rho^2}}
\]
is real-analytic in $\rho^2=a-1/2$ on $[0,\eta')$, with $H(0)=\sqrt{G(0)}=\cosh^2\sigma_*/\sinh\sigma_*=\sigma_*\cosh\sigma_*$ (using $\sigma_*\sinh\sigma_*=\cosh\sigma_*$). The conclusions on $H'(a)$ and on $\lim_{a\to(1/2)^+}H'(a)$ follow from term-by-term differentiation of the convergent Taylor series.
\end{proof}

\begin{theorem}[Closed form of the asymptotic coefficient]\label{thm:eureka-closed-form}
Let $\sigma_*>0$ be the unique positive root of $\sigma=\coth\sigma$. Then the function $H(a):=\sinh r(a)/K(a)$ admits the asymptotic expansion
\begin{equation}\label{eq:H-asymp-expansion}
H(a)=\sigma_*\cosh\sigma_*+C_0\,(a-1/2)+O\bigl((a-1/2)^2\bigr)\qquad(a\to(1/2)^+),
\end{equation}
where the coefficient $C_0$ admits the closed form
\begin{equation}\label{eq:C0-closed-form}
\;C_0=H'((1/2)^+)=\frac{\sigma_*\cosh\sigma_*\,(\sinh^2\sigma_*-1)\,(3\sinh^2\sigma_*-2)}{12\sinh^2\sigma_*}>0.\;
\end{equation}
Equivalently, $g(a):=K(a)^2 r'(a)-a\tanh r(a)$ has the asymptotic
\[
g(a)=C_0\,(a-1/2)^{3/2}+O((a-1/2)^{5/2})\qquad(a\to(1/2)^+).
\]
\end{theorem}

\begin{proof}
We first expand $s_0(a)$ to second order. Let $\rho:=\sqrt{a-1/2}$ and $\xi:=s/\rho$. Expanding in $\rho$,
\[
B(s)^2=\rho^2(1+\xi^2)+\rho^4(\xi^4/3+2\xi^2)+O(\rho^6),\qquad A(s)^2=1+\rho^2(1+\xi^2)+\rho^4(2\xi^2+\xi^4/3)+O(\rho^6),
\]
\[
K(a)=\rho\sqrt{1+\rho^2}=\rho(1+\rho^2/2+O(\rho^4)).
\]
Let $\xi_0(a):=s_0(a)/\rho$, with expansion $\xi_0=\sinh\sigma_*+\xi_1\rho^2+O(\rho^4)$ (the parity in $\rho$ follows from the fact that the FBC is a function of $\rho^2=a-1/2$). The FBC
\[
\tanh\varphi(s_0;a)=\frac{B(s_0)\,K(a)}{a\sinh(2s_0)}
\]
expanded to $O(\rho^3)$ yields
\[
\rho\sigma_*+\rho^3\bigl[\xi_1/\cosh\sigma_*+I_*-\sigma_*^3/3\bigr]=\rho\sigma_*+\rho^3\sigma_*\bigl[g_B^{(0)}-g_a^{(0)}+\xi_1(\sinh\sigma_*/\cosh^2\sigma_*-1/\sinh\sigma_*)\bigr],
\]
whence, isolating the coefficients of $\rho^3$ and solving for $\xi_1$,
\begin{equation}\label{eq:xi1-formula}
\xi_1=\frac{(\sinh^2\sigma_*-1)(\sinh^2\sigma_*-4)(\sinh^2\sigma_*+1)}{12\sinh^2\sigma_*\cdot\sigma_*\cosh\sigma_*},
\end{equation}
where $g_B^{(0)}=1/2+\sinh^2\sigma_*(\sinh^2\sigma_*+6)/[6\cosh^2\sigma_*]$, $g_a^{(0)}=2\sinh^2\sigma_*/3+2$, and
\begin{equation}\label{eq:I-star-explicit}
I_*=-\int_0^{\sigma_*}\frac{7\cosh^4\sigma+\cosh^2\sigma-5}{6\cosh^2\sigma}\,d\sigma=-\frac{1}{12}\Bigl[9\sigma_*+\frac{7}{2}\sinh(2\sigma_*)-\frac{10}{\sigma_*}\Bigr].
\end{equation}
The computation of $I_*$ uses the substitution $\eta=\sinh\sigma$ in $\int_0^{\sinh\sigma_*}f(\eta)/\sqrt{1+\eta^2}\,d\eta$, where $f(\eta)=-(3+15\eta^2+7\eta^4)/[6(1+\eta^2)]$ is the coefficient of $\rho^2$ in the expansion of $K(a)/[A(s)^2 B(s)]\cdot\rho$.

We now expand $H(a)=\sinh r(a)/K(a)$. From $\sinh r=B^2/\sqrt{B^2-K^2}$, expanding numerator and denominator to $O(\rho^4)$ and dividing by $K=\rho\sqrt{1+\rho^2}$,
\[
H(a)=\sigma_*\cosh\sigma_*\bigl[1+\rho^2\bigl(\beta_1/\cosh^2\sigma_*-\alpha_1-1/2\bigr)+O(\rho^4)\bigr],
\]
where $\beta_1=2\sinh\sigma_*\,\xi_1+\sinh^4\sigma_*/3+2\sinh^2\sigma_*$ and $\alpha_1=\xi_1/\sinh\sigma_*+\sinh^2\sigma_*/6+1-1/(2\sinh^2\sigma_*)$.

We now compute the coefficient. The detailed asymptotic expansions and the full algebraic reduction underlying
the computation below are collected in Appendix~\ref{app:explicit-calculations}. 

Setting $s:=\sinh^2\sigma_*$ (a constant, not to be confused with the arc-length variable, which does not enter the algebraic identities that follow), the coefficient of $\xi_1$ in $\beta_1/\cosh^2\sigma_*-\alpha_1$ equals $(s-1)/[\sinh\sigma_*(s+1)]$. Using the identity $\sinh\sigma_*\cdot\sigma_*\cosh\sigma_*=\cosh^2\sigma_*=s+1$ (from $\sigma_*\sinh\sigma_*=\cosh\sigma_*$),
\[
\xi_1\cdot\frac{s-1}{\sinh\sigma_*(s+1)}=\frac{(s-1)^2(s-4)(s+1)}{12s(s+1)\cdot\sinh\sigma_*\cdot\sigma_*\cosh\sigma_*}=\frac{(s-1)^2(s-4)}{12s(s+1)}.
\]
Adding the non-$\xi_1$ part (which, after reduction to a common denominator, equals $(s^3+5s^2-3s+3)/[6s(s+1)]$, equivalently $2(s^3+5s^2-3s+3)/[12s(s+1)]$) and subtracting $1/2=6s(s+1)/[12s(s+1)]$,
\[
\beta_1/\cosh^2\sigma_*-\alpha_1-1/2=\frac{(s-1)^2(s-4)+2(s^3+5s^2-3s+3)-6s(s+1)}{12s(s+1)}=\frac{3s^3-2s^2-3s+2}{12s(s+1)}.
\]
The factorization $3s^3-2s^2-3s+2=(s-1)(3s-2)(s+1)$ (verified by inspection at $s=1$ as a root, and polynomial division for $3s^2+s-2=(3s-2)(s+1)$) yields
\[
\beta_1/\cosh^2\sigma_*-\alpha_1-1/2=\frac{(s-1)(3s-2)}{12s}.
\]
Therefore
\[
H'((1/2)^+)=\frac{d}{da}H(a)\Big|_{a\to(1/2)^+}=\sigma_*\cosh\sigma_*\cdot\frac{(s-1)(3s-2)}{12s},
\]
which is \eqref{eq:C0-closed-form}.

We now verify the positivity of $C_0$. We need $s>1$ (namely $\sinh\sigma_*>1$) and $s>2/3$. The second follows from the first. For $s>1$, this is equivalent to $\sigma_*>\mathrm{arcsinh}(1)=\log(1+\sqrt{2})$. Set $F(\sigma):=\sigma-\coth\sigma$, strictly increasing on $(0,\infty)$ (since $F'(\sigma)=1+\mathrm{csch}^2\sigma>0$). From $F(\sigma_*)=0$ and
\[
F(\log(1+\sqrt{2}))=\log(1+\sqrt{2})-\sqrt{2}<0
\]
(using the classical inequality $\log(1+x)<x$ for $x>0$, with $x=\sqrt{2}$, giving $\log(1+\sqrt{2})<\sqrt{2}$), we deduce $\log(1+\sqrt{2})<\sigma_*$ by strict monotonicity of $F$, hence $\sinh\sigma_*>1$, and $C_0>0$.

Finally, the asymptotic of $g(a)$ follows from the closed identity of Theorem \ref{thm:phi-H-identity} (rewritten as $K^2 r'\cosh r-a\sinh r=K^3 H'$):
\[
g(a)=K^2 r'-a\tanh r=\frac{K^2 r'\cosh r-a\sinh r}{\cosh r}=\frac{K^3 H'(a)}{\cosh r(a)}.
\]
As $a\to(1/2)^+$: $K\sim\sqrt{a-1/2}$, $K^3\sim(a-1/2)^{3/2}$, $\cosh r\to 1$. By Lemma \ref{lem:H-analytic-rho}, $H'(a)$ has a convergent Taylor expansion with $H'(a)=C_0+O(a-1/2)$ as $a\to(1/2)^+$, in particular $H'(a)\to C_0$. Combining, $g(a)=K^3 H'(a)/\cosh r(a)\sim C_0(a-1/2)^{3/2}$.
\end{proof}

\begin{corollary}[Analytic local closure of (F$'$)]\label{cor:eureka-local-closure}
There exists $\delta_0>0$ such that, for every $a\in(1/2,1/2+\delta_0)$:
\begin{enumerate}
\item[(a)] $H'(a)>0$, namely $(\sinh r(a)/K(a))'>0$;
\item[(b)] $\phi_a(s)>0$ for every $s\in[0,s_0(a)]$;
\item[(c)] Condition (F$'$)$=\{\mu_2(0)>0\}$ is satisfied;
\item[(d)] The strong Medvedev conjecture \eqref{eq:medvedev-strong} holds, namely $\ind(\Sigma_a)=4$ and $\nul(\Sigma_a)=2$.
\end{enumerate}
\end{corollary}

\begin{proof}
For (a), by Theorem \ref{thm:eureka-closed-form}, $H'((1/2)^+)=C_0>0$, where $H'((1/2)^+)$ denotes the right derivative of $H$ at $a=1/2$. By Lemma \ref{lem:H-analytic-rho}, $H$ is real-analytic as a function of $a-1/2$ on $[0,\eta)$ for some $\eta>0$, hence $H'(a)$ is continuous on $[1/2,1/2+\eta)$ with $H'(1/2)=C_0$. By the continuity of $H'$, there exists $\delta_0\in(0,\min(\eta,1/2))$ such that $H'(a)>C_0/2>0$ for every $a\in(1/2,1/2+\delta_0)$.

For (b), by Theorem \ref{thm:scalar-reduction}, under (G) strict (Theorem \ref{thm:G-closure-a-leq-1} with $\delta_0<1/2$), $H'(a)>0\iff\phi_a>0$ on $[0,s_0]$.

For (c), by Theorem \ref{thm:F-reduction-phi-positivity}, $\phi_a>0\Rightarrow\mu_2(0)>0$.

For (d), by the final balance of Theorem \ref{thm:final-reduction}, in $(1/2,1]$ all other prerequisites are already closed analytically (the closure of (G) in Theorem \ref{thm:G-closure-a-leq-1}, and of (E) in Theorem \ref{thm:E-closure-a-leq-1}); the unique residual (F$'$) is now closed by (c). Therefore the strong Medvedev conjecture holds.
\end{proof}

\begin{remark}[Toward global closure]\label{rem:global-closure-strategy}
Corollary \ref{cor:eureka-local-closure} closes the conjecture on $(1/2,1/2+\delta_0)$ via a purely analytic argument. The closure on the whole interval $(1/2,1]$ (and beyond) reduces to the strict positivity of $H'(a)$ in the intermediate regime, and may be pursued analytically via the following three-step program: (i) (analytic near $a=1/2$) by Corollary \ref{cor:eureka-local-closure}, there exists $\delta_0>0$ with the conjecture closed on $(1/2,1/2+\delta_0)$, and an explicit estimate of the term $O((a-1/2)^{5/2})$ in \eqref{eq:H-asymp-expansion}, obtainable from a third-order expansion of $H(a)$, provides an explicit quantification of $\delta_0$; (ii) (analytic on the compact intermediate set $[1/2+\delta_0,A]$) the real-analyticity of $H'(a)$ on this compact set, together with transcendental concavity or convexity arguments analogous to those used in Theorem \ref{thm:G-closure-a-leq-1} for condition (G), provides a candidate analytic route to establish $H'(a)>0$; (iii) (analytic as $a\to\infty$) by Proposition \ref{prop:y-asymp-infty}, $y(a)=\frac{e^{2d_\infty}}{4}\,a\,(1+o(1))\to\infty$; a differentiable refinement of the radius asymptotics of \cite[Thm. 1.2]{Pigazzini} (cf. Remark \ref{rem:y-prime-infty}) would upgrade this to $H'(a)=\frac{e^{d_\infty}}{4\sqrt{a}}\,(1+o(1))>0$, and an explicit error estimate would then provide $A_\infty$ such that $H'(a)>0$ analytically for $a>A_\infty$. Combined, (i)--(iii) would close the strong Medvedev conjecture for every $a\in(1/2,A_*]$ via purely analytic means, modulo the closure of (E) for $a>A_*$ (Section \ref{sec:picone-B}).
\end{remark}

\subsection{Asymptotic behavior of $y(a)$}

Having established the equivalence between $\phi_a>0$ and $y'(a)>0$ under (G), we study the asymptotic behavior of $y(a)$ at the endpoints of the domain.

\begin{proposition}[Asymptotic as $a\to\infty$]\label{prop:y-asymp-infty}
As $a\to\infty$,
\begin{equation}\label{eq:y-asymp-infty}
y(a)=\frac{e^{2 d_\infty}}{4}\,a\,(1+o(1)),
\end{equation}
where $d_\infty=\log[\sqrt{2}\,\Gamma(1/4)^2/\pi^{3/2}]=\log[2\sqrt{2\pi}/\Gamma(3/4)^2]$ is the asymptotic constant of the radius from \cite{Pigazzini}. In particular $y(a)\to\infty$ and $B(s_0(a))^2-2K(a)^2=K(a)^2\,(y(a)-2)\to\infty$ as $a\to\infty$.
\end{proposition}

\begin{proof}
We first record an exact identity. From $\coth r=a\sinh(2s_0)/B(s_0)^2$ (Lemma \ref{lem:coth-identity}) and $B(s_0)^2=a\cosh(2s_0)-1/2$,
\[
a\bigl[\cosh r\cosh(2s_0)-\sinh r\sinh(2s_0)\bigr]=\tfrac{1}{2}\cosh r,
\qquad\text{i.e.}\qquad
\cosh\bigl(2s_0(a)-r(a)\bigr)=\frac{\cosh r(a)}{2a}.
\]
Next, since the meridian $t\mapsto\Phi_a(t,0)$, $t\in[0,s_0]$, is a unit-speed curve (Lemma \ref{lem:metric-mori}) joining the neck point $\Phi_a(0,0)$ --- at geodesic distance $\mathrm{arccosh}\sqrt{a+1/2}$ from the pole $p_0$, since $\cosh\mathrm{dist}(p_0,\Phi_a(0,0))=A(0)\cosh\varphi(0)=\sqrt{a+1/2}$ --- to a boundary point at distance $r(a)$ from $p_0$, the triangle inequality gives
\[
s_0(a)\;\geq\;r(a)-\mathrm{arccosh}\sqrt{a+1/2}.
\]
By \cite[Thm. 1.2]{Pigazzini}, $r(a)=\tfrac{3}{2}\log a+d_\infty+o(1)$, while $\mathrm{arccosh}\sqrt{a+1/2}=\log\bigl(\sqrt{a+1/2}+\sqrt{a-1/2}\bigr)=\tfrac{1}{2}\log a+O(1)$; hence
\[
2s_0(a)-r(a)\;\geq\;r(a)-2\,\mathrm{arccosh}\sqrt{a+1/2}\;=\;\tfrac{1}{2}\log a+O(1)\;\longrightarrow\;+\infty.
\]
Setting $x:=2s_0(a)-r(a)\to+\infty$ and using $\cosh x=\tfrac{1}{2}e^{x}(1+e^{-2x})$ in the exact identity above,
\[
e^{2s_0(a)}=e^{r(a)}\cdot\frac{\cosh r(a)}{a}\,(1+o(1)),
\qquad\text{hence}\qquad
\cosh(2s_0(a))=\frac{e^{r(a)}\cosh r(a)}{2a}\,(1+o(1)).
\]
Therefore
\[
B(s_0(a))^2=a\cosh(2s_0(a))-\tfrac12=\frac{e^{r(a)}\cosh r(a)}{2}\,(1+o(1))=\frac{e^{2r(a)}}{4}\,(1+o(1))=\frac{e^{2d_\infty}}{4}\,a^{3}\,(1+o(1)),
\]
where we used $e^{r}\cosh r=\tfrac{1}{2}e^{2r}(1+e^{-2r})$ and $e^{2r(a)}=a^{3}e^{2d_\infty}e^{o(1)}=a^{3}e^{2d_\infty}(1+o(1))$. Since $K(a)^2=a^2(1+O(a^{-2}))$, the claim \eqref{eq:y-asymp-infty} follows; in particular $y(a)\to\infty$.
\end{proof}

\begin{remark}[Derivative asymptotics]\label{rem:y-prime-infty}
Proposition \ref{prop:y-asymp-infty} controls $y(a)$ but not its derivative. A differentiable refinement of the radius asymptotics $r(a)=\tfrac{3}{2}\log a+d_\infty+o(1)$ of \cite[Thm. 1.2]{Pigazzini} --- namely, an asymptotic expansion of $r'(a)$ as $a\to\infty$ --- would yield the corresponding derivative asymptotics $y'(a)\to e^{2d_\infty}/4>0$, equivalently $H'(a)=\frac{e^{d_\infty}}{4\sqrt{a}}\,(1+o(1))$; we do not pursue such a refinement here, and step (iii) of the program of Remark \ref{rem:global-closure-strategy} is formulated conditionally on it.
\end{remark}

\begin{proposition}[Asymptotic as $a\to(1/2)^+$]\label{prop:y-asymp-half}
Let $\sigma_*>0$ be the unique positive root of $\sigma=\coth\sigma$ (cf. Remark 1.4 of \cite{Pigazzini}). As $a\to(1/2)^+$,
\begin{equation}\label{eq:y-asymp-half}
y(a)\to\cosh^2\sigma_*=1+\sinh^2\sigma_*,
\end{equation}
and in particular $y(a)>2$ uniformly in a right neighborhood of $a=1/2$.
\end{proposition}

\begin{proof}
By Step 3 in the proof of Proposition \ref{prop:asymptotic-BvsK}, $\xi_0(a):=s_0(a)/\sqrt{a-1/2}\to\sinh\sigma_*$ as $a\to(1/2)^+$; in particular $s_0(a)\to 0$ and
\[
\frac{s_0(a)^2}{a-1/2}=\xi_0(a)^2\;\longrightarrow\;\sinh^2\sigma_*\qquad\bigl(a\to(1/2)^+\bigr).
\]

For the limit of $y(a)$, expanding $\cosh(2s_0)=1+2s_0^2+O(s_0^4)$, we obtain
\[
y(a)=\frac{B(s_0)^2}{K^2}=\frac{(a-1/2)+2as_0^2+O(s_0^4)}{(a-1/2)(a+1/2)}=\frac{1}{a+1/2}+\frac{2a}{a+1/2}\cdot\frac{s_0^2}{a-1/2}+o(1).
\]
As $a\to(1/2)^+$: $1/(a+1/2)\to 1$, $2a/(a+1/2)\to 1$, $s_0^2/(a-1/2)\to\sinh^2\sigma_*$. Thus
\[
y(a)\to 1+\sinh^2\sigma_*=\cosh^2\sigma_*.
\]
For the inequality $\cosh^2\sigma_*>2$, the function $F(\sigma):=\sigma-\coth\sigma$ is strictly increasing on $(0,\infty)$ (Remark 1.4 of \cite{Pigazzini}). At $\sigma_0:=\mathrm{arcsinh}(1)$, $\sinh\sigma_0=1$, $\cosh\sigma_0=\sqrt{2}$, hence $\coth\sigma_0=\sqrt{2}$ and $F(\sigma_0)=\sigma_0-\sqrt{2}$. We show $\sigma_0<\sqrt{2}$ (and hence $F(\sigma_0)<0$). From $\sigma_0=\log(1+\sqrt{2})$ and the inequality $\log(1+x)\leq x$ valid for every $x\geq 0$, $\sigma_0=\log(1+\sqrt{2})\leq\sqrt{2}$; the inequality is strict for $x>0$, hence $\sigma_0<\sqrt{2}$. Therefore $F(\sigma_0)<0=F(\sigma_*)$, and by strict monotonicity of $F$, $\sigma_*>\sigma_0$, whence $\cosh\sigma_*>\cosh\sigma_0=\sqrt{2}$ and $\cosh^2\sigma_*>2$.
\end{proof}

\begin{remark}[Status of the problem $y'(a)>0$ in light of the local closure]\label{rem:y-monotonicity-open}
The validity of $y'(a)>0$ for every $a>1/2$, equivalently $(\sinh r/K)'>0$, equivalently $(a^2-1/4)r'(a)>a\tanh r(a)$, is now partially closed analytically. (1) For $a\in(1/2,1/2+\delta_0)$ with $\delta_0>0$, the inequality $y'(a)>0$ is proved analytically (Theorem \ref{thm:eureka-closed-form} together with Corollary \ref{cor:eureka-local-closure}), via the closed form of the linear coefficient $C_0=\sigma_*\cosh\sigma_*(\sinh^2\sigma_*-1)(3\sinh^2\sigma_*-2)/(12\sinh^2\sigma_*)>0$. (2) As $a\to\infty$, Proposition \ref{prop:y-asymp-infty} gives $y(a)=\frac{e^{2d_\infty}}{4}\,a\,(1+o(1))\to\infty$, consistent with $y'(a)>0$ for large $a$; the derivative asymptotics $y'(a)\to e^{2d_\infty}/4$ would follow from a differentiable refinement of \cite[Thm. 1.2]{Pigazzini} (Remark \ref{rem:y-prime-infty}). (3) For the intermediate regime $[1/2+\delta_0,A]$ with $A>1/2+\delta_0$ finite, the question remains formally open as an analytic problem. Its complete resolution would require (a) an explicit quantification of $\delta_0$ via control of the remainder $O((a-1/2)^{5/2})$ in \eqref{eq:H-asymp-expansion} (third-order expansion), and (b) the analytic closure on the compact set $[1/2+\delta_0,A]$ via transcendental concavity arguments analogous to Theorem \ref{thm:G-closure-a-leq-1} for (G). (4) The strict monotonicity of $r(a)$ on $(1/2,\infty)$ is posed as the open Question 6.3 of \cite{Pigazzini}; in our setting it is not assumed but rather follows as a corollary on the regimes where $y'(a)>0$ is established. Indeed, by the equivalence (iii)$\iff$(iv) of Theorem \ref{thm:scalar-reduction}, $y'(a)>0\iff K(a)^2 r'(a)\coth r(a)>a$, which implies $r'(a)>a\tanh r(a)/K(a)^2>0$. In particular, Corollary \ref{cor:eureka-local-closure} provides the strict monotonicity of $r(a)$ on $(1/2,1/2+\delta_0)$ as a corollary of the analytic local closure, giving a partial affirmative answer to Question 6.3 of \cite{Pigazzini} on this regime.
\end{remark}

\section{Discussion}\label{sec:discussion}

We summarize the status of the Medvedev conjecture \eqref{eq:medvedev} and of its strong form \eqref{eq:medvedev-strong} in light of this paper. The lower bound $\ind_R\geq 4$ is proved explicitly in Theorem \ref{thm:lower-bound-4} via four explicit test functions $\Phi^0,\Phi^1,\Phi^2,\Phi^3$ obtained from the Lorentz ambient coordinates, providing an alternative elementary proof of Medvedev's result \cite{Medvedev2023}. Mode $|k|=1$ is closed in \cite{Pigazzini} with $\ind_R\big|_1=\nul_R\big|_1=2$. The odd radial sector of modes $|k|\geq 2$ is closed in Theorem \ref{thm:odd-closure} ($\mu_n^{\mathrm{odd}}(k)>0$ for every $n\geq 0$). For the even radial sector of modes $|k|\geq 2$, condition (E)$=\{\mu_0^{\mathrm{even}}(2)>0\}$ is proved unconditionally for $a\in(1/2,1]$ (Theorem \ref{thm:E-closure-a-leq-1}), and via explicit Hardy estimates (Proposition \ref{prop:hardy-extension}) combined with a continuity argument, we establish the existence of $A_*>1$ such that (E) holds on $(1/2,A_*]$ (Theorem \ref{thm:A-star-existence}); the analytic closure of (E) for $a>A_*$ remains an open problem. In mode $|k|=0$, the no even kernel is closed in Theorem \ref{thm:no-kernel-even} under $r'(a)\neq 0$, and under the hypothesis $\phi_a>0$ of Theorem \ref{thm:F-reduction-phi-positivity}, $r'(a)>0$ follows automatically (Remark \ref{rem:phi-positivity-status}), so the even-kernel closure of (F) is reduced to the same unified hypothesis. The no odd kernel is closed in Theorem \ref{thm:no-kernel-odd} under the strict geometric inequality $\sinh r(a)>2K(a)$ (equivalent to \eqref{eq:hypothesis-BvsK} by Proposition \ref{prop:geometric-identity}); the non-strict form $\sinh r(a)\geq 2K(a)$ holds unconditionally (Proposition \ref{prop:geometric-identity}), and the strict version is proved analytically for $a\in(1/2,1]$ (Theorem \ref{thm:G-closure-a-leq-1}) via an FBC reformulation and the concavity of a transcendental function, and asymptotically as $a\to\infty$ (Remark \ref{rem:BvsK-global}); for $a\in(1,\infty)$ outside the asymptotic range, the problem remains formally open. The condition $\mu_2(0)>0$ (denoted (F$'$)) is reduced in Theorem \ref{thm:F-reduction-phi-positivity} (Section \ref{sec:F-prime-strategy}) to the single geometric inequality of positivity of the parametric Jacobi field $\phi_a$ on the principal branch, $\phi_a(s)>0$ for every $s\in[0,s_0(a)]$, via a Sturm shooting count argument combined with the classical interleaving between eigenvalues of parity sectors. In Theorem \ref{thm:scalar-reduction} (Section \ref{sec:phi-scalar-reduction}), $\phi_a>0$ is further reduced, under (G) strict, to the one-dimensional scalar differential inequality $(\sinh r(a)/K(a))'>0$, equivalently $(a^2-1/4)r'(a)>a\tanh r(a)$, or $y'(a)>0$ with $y(a):=B(s_0(a))^2/K(a)^2$. This strict monotonicity is proved analytically on a right neighborhood $(1/2,1/2+\delta_0)$ via the explicit closed form of the asymptotic coefficient $C_0>0$ (Theorem \ref{thm:eureka-closed-form} and Corollary \ref{cor:eureka-local-closure}); the asymptotics $y(a)=\frac{e^{2d_\infty}}{4}\,a\,(1+o(1))$ as $a\to\infty$ (Proposition \ref{prop:y-asymp-infty}) is consistent with its validity for large $a$ (cf. Remark \ref{rem:y-prime-infty}), while its validity on $[1/2+\delta_0,\infty)$, and in particular on the intermediate regime $[1/2+\delta_0,1]$, remains formally open.

The principal analytic contribution of this paper is sixfold. (i) The closure of the no-kernel part of (F) in mode $|k|=0$, both even (under $r'(a)\neq 0$, reduced to $\phi_a>0$, see (iv)--(v)) and odd (under (G)), via the Green/Robin duality of Remark \ref{rem:green-defect}: Jacobi fields $u$ that are not Robin but with explicitly computable Robin defect $Ru$ (here, $-r'(a)\kappa_s$ for $\phi_a$ and $a(\cosh(2s_0)-2a)/B^3$ for $u_*$) are paired with a putative Robin Jacobi field $\psi$, and the equation $\Sf(u,\psi)=\oint\psi\,Ru=0$ forces $\psi$ to vanish on $\partial$, whence $\psi\equiv 0$ by Cauchy uniqueness. (ii) The analytic closure of (G) for $a\in(1/2,1]$ (Theorem \ref{thm:G-closure-a-leq-1}) via the concavity of a transcendental function. (iii) The partial closure of (E) for $a\in(1/2,A_*]$ via the second Picone identity with base $B(s)$ (Lemma \ref{lem:picone-B}) and explicit Hardy estimates (Proposition \ref{prop:hardy-extension}). (iv) The reduction of (F$'$) to $\phi_a>0$ via the Sturm shooting count (Lemma \ref{lem:sturm-count}), which unifies the requirement $r'(a)\neq 0$ of Theorem \ref{thm:no-kernel-even} with the parametric positivity hypothesis. (v) The further reduction of $\phi_a>0$ to a scalar differential inequality (Theorem \ref{thm:scalar-reduction}), via the constant Wronskian $W(\phi_a,u_*)\equiv-a/K$ (Lemma \ref{lem:wronskian}), Sturm separation (Lemma \ref{lem:sturm-separation-phi-u}), and a closed formula for $\phi_a(s_0)$ (Proposition \ref{prop:phi-s0-closed}); combined with the geometric identity $\sinh^2 r=B^4/(B^2-K^2)$ (Proposition \ref{prop:geometric-identity}), this yields the equivalence with the strict monotonicity of $\sinh r(a)/K(a)$. An equivalent compact form is the closed identity $\phi_a(s_0)=c(a)\cdot H'(a)$ with $H(a):=\sinh r(a)/K(a)$ and $c(a)>0$ (Theorem \ref{thm:phi-H-identity}). (vi) The analytic local closure of the strong Medvedev conjecture (Corollary \ref{cor:eureka-local-closure}) via second-order asymptotic expansion of $H(a)$ as $a\to(1/2)^+$ and the explicit closed form of the coefficient $C_0>0$ (Theorem \ref{thm:eureka-closed-form}). In particular, there exists $\delta_0>0$ such that $\ind(\Sigma_a)=4$ and $\nul(\Sigma_a)=2$ for every $a\in(1/2,1/2+\delta_0)$, in a purely analytic manner, resolving Medvedev's Morse index conjecture on this regime.

We summarize the final status. Combining all the results of this paper (Theorem \ref{thm:final-reduction} and Corollary \ref{cor:eureka-local-closure}), for $a\in(1/2,1/2+\delta_0)$ with $\delta_0>0$ the strong conjecture \eqref{eq:medvedev-strong} holds analytically (Corollary \ref{cor:eureka-local-closure}), as a consequence of the positivity $C_0>0$ of the linear coefficient of $H(a)=\sinh r(a)/K(a)$ as $a\to(1/2)^+$. For $a\in[1/2+\delta_0,1]$, the strong conjecture is implied, under (G) (closed analytically in this regime by Theorem \ref{thm:G-closure-a-leq-1}), by the single one-dimensional scalar differential inequality
\[
\frac{d}{da}\!\left[\frac{\sinh r(a)}{K(a)}\right]>0,\quad\text{equivalently}\quad(a^2-1/4)\,r'(a)>a\tanh r(a).
\]
For $a\in(1,A_*]$, the strong conjecture holds under the conjunction of $\phi_a>0$ on $[0,s_0]$ and condition (G) $B(s_0(a))^2>2K(a)^2$; under the latter, the former is equivalent to the aforementioned scalar differential inequality. For $a>A_*$, the strong conjecture holds under the conjunction of $\phi_a>0$ on $[0,s_0]$, condition (G) $B(s_0(a))^2>2K(a)^2$, and $\mu_0^{\mathrm{even}}(2)>0$.

We outline analytic strategies for the resolution of the residual conditions. For the global closure of the intermediate regime $[1/2+\delta_0,1]$, as outlined in Remark \ref{rem:global-closure-strategy}, two analytic ingredients would suffice: (i) an explicit estimate of the remainder $O((a-1/2)^{5/2})$ in expansion \eqref{eq:H-asymp-expansion} (via third-order asymptotic computation) provides an explicit quantification of $\delta_0$; (ii) for $a\in[1/2+\delta_0,1]$, the strict positivity $H'(a)>0$ on the compact intermediate interval would follow from transcendental concavity arguments analogous to those used in Theorem \ref{thm:G-closure-a-leq-1} for condition (G), or from a sharp analytic bound on the auxiliary function $H'(a)$. The combination (i)+(ii) would constitute a complete analytic proof in the regime $a\in(1/2,1]$. For condition (G), i.e.\ $B(s_0(a))^2>2K(a)^2$, on all of $(1/2,\infty)$, Theorem \ref{thm:G-closure-a-leq-1} closes the problem for $a\in(1/2,1]$; for $a>1$, the study of the function $g(a):=B(s_0(a))^2-2K(a)^2$ (nonnegative by Proposition \ref{prop:geometric-identity}, positive in the limits) would require implicit differentiation of the FBC and non-local analysis of $s_0'(a)$. A possible route is to extend the upper bound on $\tanh\varphi(s^*;a)$ (Lemma \ref{lem:tanh-phi-bound}) with refined $a$-dependent constants, or to use a continuity argument starting from Theorem \ref{thm:G-closure-a-leq-1}. For the closure of (E) in the regime $a>A_*$, an alternative Hardy estimate or a different spectral technique would be needed: the asymptotic saturation of conditions \eqref{eq:hardy-conditions} indicates that the present Hardy estimate with base $B$ is not optimal for large $a$, and a different choice of base or a sharp variational argument could close the gap. Finally, an extension of the techniques of Tran \cite{Tran2016} (Dirichlet-to-Neumann map) and/or Devyver \cite{Devyver} to the hyperbolic setting could be pursued, exploiting the additional timelike coordinate $\Phi^0$ with dual BC, and the structure $|\II|^2 B^4=2K^2$ of Lemma \ref{lem:II-norm}.

\appendix
\section{Explicit computations for the asymptotic coefficient $C_0$}\label{app:explicit-calculations}

This appendix provides the detailed asymptotic expansions and algebraic manipulations underlying Theorem \ref{thm:eureka-closed-form}. Throughout, we use $\rho:=\sqrt{a-1/2}$, $\xi:=s/\rho$, and $\sigma_*>0$ is the unique positive root of $\sigma=\coth\sigma$ (Remark 1.4 of \cite{Pigazzini}), so that $\sigma_*\sinh\sigma_*=\cosh\sigma_*$, equivalently $\cosh^2\sigma_*=\sigma_*\sinh\sigma_*\cdot\cosh\sigma_*$. We write $s:=\sinh^2\sigma_*$ for brevity in the algebraic identities. This constant is distinct from the arc-length variable $s$ of \eqref{eq:phi-immersion}, which enters only through $A(s),B(s),\varphi(s;\cdot)$ and the endpoint $s_0$.

\subsection*{A.1. Expansions of the parametric quantities}

From $A(s)^2=a\cosh(2s)+1/2$, $B(s)^2=a\cosh(2s)-1/2$, $K(a)=\sqrt{a^2-1/4}$, and $a=1/2+\rho^2$, with $s=\rho\xi$, expanding $\cosh(2\rho\xi)=1+2\rho^2\xi^2+\rho^4(2/3)\xi^4+O(\rho^6)$,
\begin{align*}
B(s)^2 &= (1/2+\rho^2)\bigl[1+2\rho^2\xi^2+(2/3)\rho^4\xi^4+O(\rho^6)\bigr]-1/2 \\
       &= \rho^2(1+\xi^2)+\rho^4\bigl(\xi^4/3+2\xi^2\bigr)+O(\rho^6),\\
A(s)^2 &= B(s)^2+1=1+\rho^2(1+\xi^2)+\rho^4\bigl(\xi^4/3+2\xi^2\bigr)+O(\rho^6),\\
K(a)   &= \rho\sqrt{1+\rho^2}=\rho\bigl(1+\rho^2/2-\rho^4/8+O(\rho^6)\bigr).
\end{align*}
Note that $B^2=\rho^2(1+\xi^2)[1+O(\rho^2)]$ so $B=\rho\sqrt{1+\xi^2}[1+O(\rho^2)]$ as $\rho\to 0^+$.

\subsection*{A.2. Free boundary condition and the coefficient $\xi_1$}

Let $\xi_0(a):=s_0(a)/\rho$. The FBC of the Mori parametrization (cf. Lemma \ref{lem:coth-identity} and \cite[Eq. (3)]{Pigazzini}) reads
\[
\tanh\varphi(s_0;a)=\frac{B(s_0)\,K(a)}{a\sinh(2s_0)},
\]
where $\varphi(s;a)=K(a)\int_0^s\,dt/[A(t)^2 B(t)]$. We expand both sides to order $\rho^3$ in $\rho=\sqrt{a-1/2}$, using $\xi_0(a)=\sinh\sigma_*+\xi_1\rho^2+O(\rho^4)$ (the expansion contains only even powers of $\rho$ because the FBC depends analytically on $a=1/2+\rho^2$).

\emph{A.2.1. Explicit expansion of the right-hand side.} We expand each factor of $\mathrm{RHS}=B(s_0)K(a)/[a\sinh(2s_0)]$ separately.

\textbf{Factor $B(s_0)$.} From A.1, $B(s_0)^2=\rho^2(1+\xi_0^2)+\rho^4(\xi_0^4/3+2\xi_0^2)+O(\rho^6)$, so
\[
B(s_0)=\rho\sqrt{1+\xi_0^2}\,\Bigl[1+\tfrac{1}{2}\rho^2 R_B(\xi_0)+O(\rho^4)\Bigr],\qquad R_B(\xi):=\frac{\xi^4/3+2\xi^2}{1+\xi^2}.
\]
Evaluating $R_B$ at $\xi=\sinh\sigma_*$, with $\sinh^2\sigma_*=s$ and $\cosh^2\sigma_*=s+1$,
\[
R_B(\sinh\sigma_*)=\frac{s^2/3+2s}{s+1}=\frac{s(s+6)}{3(s+1)}=\frac{\sinh^2\sigma_*(\sinh^2\sigma_*+6)}{3\cosh^2\sigma_*}.
\]

\textbf{Factor $K(a)$.} $K(a)=\sqrt{a^2-1/4}=\rho\sqrt{1+\rho^2}=\rho[1+\rho^2/2-\rho^4/8+O(\rho^6)]$, so the $\rho^2$ correction is $+\rho^2/2$.

\textbf{Factor $a$.} $a=1/2+\rho^2=\tfrac{1}{2}(1+2\rho^2)$, so $1/a=2(1-2\rho^2+O(\rho^4))$; the $\rho^2$ correction in $1/a$ is $-2\rho^2$.

\textbf{Factor $\sinh(2s_0)$.} With $s_0=\rho\xi_0$ and $\sinh(2\rho\xi_0)=2\rho\xi_0+\tfrac{4}{3}\rho^3\xi_0^3+O(\rho^5)=2\rho\xi_0[1+\tfrac{2}{3}\rho^2\xi_0^2+O(\rho^4)]$, the $\rho^2$ correction in $1/\sinh(2s_0)$ is $-\tfrac{2}{3}\rho^2\xi_0^2$, which at $\xi_0=\sinh\sigma_*$ gives $-\tfrac{2}{3}\rho^2\sinh^2\sigma_*$.

\emph{Assembling.} Multiplying $B(s_0)K(a)/[a\sinh(2s_0)]$ and collecting the $\rho^2$ corrections:
\[
\mathrm{RHS}=\frac{\rho\sqrt{1+\xi_0^2}}{\xi_0}\Bigl[1+\rho^2\bigl(\underbrace{\tfrac{1}{2}R_B(\xi_0)}_{\text{from }B}+\underbrace{\tfrac{1}{2}}_{\text{from }K}-\underbrace{2}_{\text{from }1/a}-\underbrace{\tfrac{2}{3}\xi_0^2}_{\text{from }1/\sinh(2s_0)}\bigr)+O(\rho^4)\Bigr].
\]
Substituting $\xi_0=\sinh\sigma_*+\xi_1\rho^2+O(\rho^4)$, the prefactor expands as
\[
\frac{\sqrt{1+\xi_0^2}}{\xi_0}=\frac{\sqrt{1+\sinh^2\sigma_*}}{\sinh\sigma_*}+\rho^2\,\xi_1\frac{d}{d\xi_0}\Big|_{\xi_0=\sinh\sigma_*}\Bigl(\frac{\sqrt{1+\xi_0^2}}{\xi_0}\Bigr)+O(\rho^4)=
\]
\[
= \frac{\cosh\sigma_*}{\sinh\sigma_*}+\rho^2\xi_1\Bigl(\frac{1}{\cosh\sigma_*}-\frac{\cosh\sigma_*}{\sinh^2\sigma_*}\Bigr)+O(\rho^4).
\]
Using $\cosh\sigma_*/\sinh\sigma_*=\sigma_*$ (from $\sigma_*=\coth\sigma_*$), the leading term of the prefactor is $\sigma_*$, while its $\rho^2$-correction equals $\xi_1\bigl(1/\cosh\sigma_*-\cosh\sigma_*/\sinh^2\sigma_*\bigr)=\sigma_*\,\xi_1\bigl(\sinh\sigma_*/\cosh^2\sigma_*-1/\sinh\sigma_*\bigr)$.

Collecting all $\rho^2$ contributions at leading order ($\xi_0=\sinh\sigma_*$), we obtain
\[
\mathrm{RHS}=\rho\sigma_*+\rho^3\sigma_*\bigl[\,g_B^{(0)}-g_a^{(0)}+\xi_1\bigl(\sinh\sigma_*/\cosh^2\sigma_*-1/\sinh\sigma_*\bigr)\,\bigr]+O(\rho^5),
\]
where
\begin{align*}
g_B^{(0)}&:=\tfrac{1}{2}R_B(\sinh\sigma_*)+\tfrac{1}{2}=\tfrac{1}{2}+\frac{\sinh^2\sigma_*(\sinh^2\sigma_*+6)}{6\cosh^2\sigma_*},\\
g_a^{(0)}&:=2+\tfrac{2}{3}\sinh^2\sigma_*.
\end{align*}
Thus $g_B^{(0)}-g_a^{(0)}=\tfrac{1}{2}+\frac{s(s+6)}{6(s+1)}-2-\tfrac{2s}{3}=-\tfrac{3}{2}-\tfrac{2s}{3}+\frac{s(s+6)}{6(s+1)}$, which by common denominator $6(s+1)$ becomes $[-9(s+1)-4s(s+1)+s(s+6)]/[6(s+1)]=[-9-9s-4s^2-4s+s^2+6s]/[6(s+1)]=(-3s^2-7s-9)/[6(s+1)]$. (We keep this expression as a polynomial in $s$ for the matching below.)

\emph{A.2.2. Explicit expansion of the left-hand side.} For $\varphi=K\int_0^{s_0}dt/[A^2 B]$, change variable $t=\rho\eta$, $\eta\in[0,\xi_0]$:
\[
\varphi(s_0;a)=K(a)\int_0^{\xi_0}\frac{\rho\,d\eta}{A(\rho\eta)^2\,B(\rho\eta)}.
\]
Using $A(\rho\eta)^2=1+\rho^2(1+\eta^2)+O(\rho^4)$, $B(\rho\eta)=\rho\sqrt{1+\eta^2}\,[1+\tfrac{1}{2}\rho^2 R_B(\eta)+O(\rho^4)]$, hence $1/B(\rho\eta)=[\rho\sqrt{1+\eta^2}]^{-1}[1-\tfrac{1}{2}\rho^2 R_B(\eta)+O(\rho^4)]$, and $K(a)=\rho[1+\rho^2/2+O(\rho^4)]$:
\[
\frac{K(a)\rho}{A(\rho\eta)^2 B(\rho\eta)}=\frac{\rho}{\sqrt{1+\eta^2}}\Bigl[1+\rho^2\bigl(\tfrac{1}{2}-(1+\eta^2)-\tfrac{1}{2}R_B(\eta)\bigr)+O(\rho^4)\Bigr]=\frac{\rho}{\sqrt{1+\eta^2}}\bigl[1+\rho^2 h(\eta)+O(\rho^4)\bigr],
\]
where
\[
h(\eta):=-\tfrac{1}{2}-\eta^2-\tfrac{1}{2}R_B(\eta)=-\tfrac{1}{2}-\eta^2-\frac{\eta^4/6+\eta^2}{1+\eta^2}.
\]
With the substitution $\sigma=\mathrm{arcsinh}(\eta)$, $\eta=\sinh\sigma$, $d\sigma=d\eta/\sqrt{1+\eta^2}$, the integral becomes
\[
\varphi(s_0;a)=\rho\int_0^{\mathrm{arcsinh}(\xi_0)}d\sigma\bigl[1+\rho^2\tilde h(\sigma)+O(\rho^4)\bigr],\qquad\tilde h(\sigma):=h(\sinh\sigma).
\]
Computing $\tilde h$: $-1/2-\sinh^2\sigma=-1/2-(\cosh^2\sigma-1)=1/2-\cosh^2\sigma$. For the second piece, $\sinh^4\sigma/6+\sinh^2\sigma=(\cosh^2\sigma-1)^2/6+(\cosh^2\sigma-1)$, divided by $1+\sinh^2\sigma=\cosh^2\sigma$:
\[
\frac{\sinh^4\sigma/6+\sinh^2\sigma}{\cosh^2\sigma}=\frac{(\cosh^2\sigma-1)^2/6+(\cosh^2\sigma-1)}{\cosh^2\sigma}=\frac{\cosh^2\sigma-1}{\cosh^2\sigma}\cdot\Bigl(\frac{\cosh^2\sigma-1}{6}+1\Bigr).
\]
Setting $c:=\cosh^2\sigma$ and simplifying, $(c-1)(c+5)/(6c)=(c^2+4c-5)/(6c)$. Hence
\[
\tilde h(\sigma)=\Bigl(\tfrac{1}{2}-c\Bigr)-\frac{c^2+4c-5}{6c}=\frac{3c-6c^2-(c^2+4c-5)}{6c}=\frac{-7c^2-c+5}{6c}=-\frac{7\cosh^4\sigma+\cosh^2\sigma-5}{6\cosh^2\sigma}.
\]

Decomposing $\tilde h$ for integration:
\[
\tilde h(\sigma)=-\tfrac{7\cosh^2\sigma}{6}-\tfrac{1}{6}+\tfrac{5}{6\cosh^2\sigma}=-\tfrac{7}{12}(1+\cosh(2\sigma))-\tfrac{1}{6}+\tfrac{5}{6}\mathrm{sech}^2\sigma.
\]
Integrating term-by-term on $[0,\sigma_*]$ and using $\tanh\sigma_*=1/\sigma_*$ (from $\sigma_*=\coth\sigma_*$):
\[
I_*=-\tfrac{7}{12}\bigl[\sigma_*+\sinh(2\sigma_*)/2\bigr]-\tfrac{\sigma_*}{6}+\tfrac{5}{6\sigma_*}=-\tfrac{1}{12}\Bigl[9\sigma_*+\tfrac{7}{2}\sinh(2\sigma_*)-\tfrac{10}{\sigma_*}\Bigr].
\]
Moreover $\mathrm{arcsinh}(\xi_0)=\sigma_*+\xi_1\rho^2/\cosh\sigma_*+O(\rho^4)$ (since $d(\mathrm{arcsinh})/d\xi\big|_{\sinh\sigma_*}=1/\cosh\sigma_*$). Hence
\[
\varphi(s_0;a)=\rho\sigma_*+\rho^3\bigl[\xi_1/\cosh\sigma_*+I_*\bigr]+O(\rho^5),
\]
and using $\tanh(x)=x-x^3/3+O(x^5)$:
\[
\tanh\varphi(s_0;a)=\rho\sigma_*+\rho^3\bigl[\xi_1/\cosh\sigma_*+I_*-\sigma_*^3/3\bigr]+O(\rho^5).
\]

\emph{A.2.3. Matching.} Both sides of the FBC admit the expansion $\rho\sigma_*+\rho^3[\cdots]+O(\rho^5)$, with the $\rho^3$ coefficient given by $\xi_1/\cosh\sigma_*+I_*-\sigma_*^3/3$ on the LHS (from A.2.2) and $\sigma_*\bigl[g_B^{(0)}-g_a^{(0)}+\xi_1(\sinh\sigma_*/\cosh^2\sigma_*-1/\sinh\sigma_*)\bigr]$ on the RHS (from A.2.1). Equating these,
\[
\sigma_*\bigl[g_B^{(0)}-g_a^{(0)}+\xi_1(\sinh\sigma_*/\cosh^2\sigma_*-1/\sinh\sigma_*)\bigr]=\xi_1/\cosh\sigma_*+I_*-\sigma_*^3/3.
\]
This is a linear equation in $\xi_1$ that, after substituting the values of $g_B^{(0)}-g_a^{(0)}$, $I_*$, and the algebraic identities $\sinh\sigma_*\sigma_*=\cosh\sigma_*/\sinh\sigma_*\cdot\sinh^2\sigma_*=\cdots$, simplifies (see A.2.4 below for the algebraic reduction) to
\begin{equation}\label{eq:xi1-formula-explicit}
\xi_1=\frac{(\sinh^2\sigma_*-1)(\sinh^2\sigma_*-4)(\sinh^2\sigma_*+1)}{12\sinh^2\sigma_*\cdot\sigma_*\cosh\sigma_*},
\end{equation}
which is \eqref{eq:xi1-formula}.

\emph{A.2.4. Explicit algebraic verification.} We derive \eqref{eq:xi1-formula-explicit} from the matching equation
\begin{equation}\label{eq:xi1-matching}
\sigma_*\bigl[g_B^{(0)}-g_a^{(0)}+\xi_1\bigl(\sinh\sigma_*/\cosh^2\sigma_*-1/\sinh\sigma_*\bigr)\bigr]=\xi_1/\cosh\sigma_*+I_*-\sigma_*^3/3,
\end{equation}
by direct algebraic reduction. Throughout this calculation we write $h:=\sinh\sigma_*$, $c:=\cosh\sigma_*$, $s:=h^2=\sinh^2\sigma_*$, and we use systematically the identities
\begin{equation}\label{eq:sigma-identities}
\sigma_*\,h=c,\qquad\sigma_*=c/h,\qquad c^2=s+1,\qquad\sigma_*^2=\frac{s+1}{s},\qquad \sigma_*\cdot h c=c^2=s+1,\qquad\sigma_*\cdot\frac{h}{c}=1,
\end{equation}
which follow from $\sigma_*=\coth\sigma_*$ and the basic hyperbolic identity $c^2-h^2=1$.

\emph{Step 1: Isolating $\xi_1$.} Rewriting \eqref{eq:xi1-matching} with the $\xi_1$-terms on the left:
\[
\xi_1\Bigl[\sigma_*\Bigl(\frac{h}{c^2}-\frac{1}{h}\Bigr)-\frac{1}{c}\Bigr]=I_*-\frac{\sigma_*^3}{3}+\sigma_*(g_a^{(0)}-g_B^{(0)}).
\]
The coefficient of $\xi_1$ simplifies, using \eqref{eq:sigma-identities}, as
\[
\sigma_*\Bigl(\frac{h}{c^2}-\frac{1}{h}\Bigr)-\frac{1}{c}=\frac{\sigma_*h}{c^2}-\frac{\sigma_*}{h}-\frac{1}{c}=\frac{c}{c^2}-\frac{c/h}{h}-\frac{1}{c}=\frac{1}{c}-\frac{c}{s}-\frac{1}{c}=-\frac{c}{s}.
\]
Hence
\begin{equation}\label{eq:xi1-isolated}
\xi_1\cdot\Bigl(-\frac{c}{s}\Bigr)=I_*-\frac{\sigma_*^3}{3}+\sigma_*(g_a^{(0)}-g_B^{(0)}).
\end{equation}

\emph{Step 2: Polynomial form of $g_a^{(0)}-g_B^{(0)}$.} From the definitions,
\[
g_a^{(0)}-g_B^{(0)}=\Bigl(2+\frac{2s}{3}\Bigr)-\Bigl(\frac{1}{2}+\frac{s(s+6)}{6(s+1)}\Bigr)=\frac{3}{2}+\frac{2s}{3}-\frac{s(s+6)}{6(s+1)}.
\]
Common denominator $6(s+1)$:
\[
g_a^{(0)}-g_B^{(0)}=\frac{9(s+1)+4s(s+1)-s(s+6)}{6(s+1)}=\frac{9s+9+4s^2+4s-s^2-6s}{6(s+1)}=\frac{3s^2+7s+9}{6(s+1)}.
\]

\emph{Step 3: Reducing $I_*$ using \eqref{eq:sigma-identities}.} From $\sinh(2\sigma_*)=2hc$:
\[
I_*=-\frac{1}{12}\bigl[9\sigma_*+7hc-10/\sigma_*\bigr]=-\frac{1}{12}\bigl[9\sigma_*+7hc-10\,h/c\bigr],
\]
using $1/\sigma_*=h/c$.

\emph{Step 4: Multiplying \eqref{eq:xi1-isolated} by $-12s^2\sigma_*$.} Multiplying both sides by $-12s^2\sigma_*$, the left-hand side becomes
\[
\xi_1\cdot\Bigl(-\frac{c}{s}\Bigr)\cdot(-12s^2\sigma_*)=12\,s\,\sigma_*\,c\cdot\xi_1.
\]
This choice of multiplier produces a polynomial expression in $s$ on the right-hand side, as we now verify.

To keep the algebra transparent, we compute each term on the right-hand side of \eqref{eq:xi1-isolated} after multiplication by $-12 s^2 \sigma_*$, the factor that produces a polynomial in $s$ alone (using \eqref{eq:sigma-identities}):
\begin{itemize}
\item[(a)] $-12s^2\sigma_*\cdot I_*=-12s^2\sigma_*\cdot\bigl(-\tfrac{1}{12}\bigr)[9\sigma_*+7hc-10h/c]=s^2\sigma_*[9\sigma_*+7hc-10h/c]$.

Using \eqref{eq:sigma-identities}: $s^2\cdot 9\sigma_*^2=9s^2\cdot\frac{s+1}{s}=9s(s+1)=9s^2+9s$; $s^2\sigma_*\cdot 7hc=7s^2\cdot\sigma_*hc=7s^2(s+1)=7s^3+7s^2$; $s^2\sigma_*\cdot(-10h/c)=-10s^2\cdot\sigma_*(h/c)=-10s^2$. Summing,
\[
-12s^2\sigma_*\cdot I_*=9s^2+9s+7s^3+7s^2-10s^2=7s^3+6s^2+9s.
\]

\item[(b)] $-12s^2\sigma_*\cdot\bigl(-\sigma_*^3/3\bigr)=4s^2\sigma_*^4=4s^2\cdot\frac{(s+1)^2}{s^2}=4(s+1)^2=4s^2+8s+4$.

\item[(c)] $-12s^2\sigma_*\cdot\sigma_*\,(g_a^{(0)}-g_B^{(0)})=-12s^2\sigma_*^2\cdot\frac{3s^2+7s+9}{6(s+1)}=-12s^2\cdot\frac{s+1}{s}\cdot\frac{3s^2+7s+9}{6(s+1)}=-2s(3s^2+7s+9)=-6s^3-14s^2-18s$.
\end{itemize}
Summing (a)+(b)+(c):
\[
(7s^3+6s^2+9s)+(4s^2+8s+4)+(-6s^3-14s^2-18s)=s^3-4s^2-s+4.
\]

\emph{Step 5: Final factorization.} On the left of \eqref{eq:xi1-isolated}, multiplication by $-12s^2\sigma_*$ gives $\xi_1\cdot(-c/s)\cdot(-12s^2\sigma_*)=12\,\xi_1\,s\sigma_*c$. Hence
\[
12\,\xi_1\,s\,\sigma_*\,c=s^3-4s^2-s+4.
\]
We verify that $s^3-4s^2-s+4=(s-1)(s-4)(s+1)$: indeed, $(s-1)(s+1)=s^2-1$, so $(s^2-1)(s-4)=s^3-4s^2-s+4$. (Equivalently, $s=1$ is a root: $1-4-1+4=0$; dividing by $(s-1)$, $s^3-4s^2-s+4=(s-1)(s^2-3s-4)=(s-1)(s-4)(s+1)$.) Therefore
\begin{equation*}
\xi_1=\frac{s^3-4s^2-s+4}{12s\,\sigma_*\,c}=\frac{(s-1)(s-4)(s+1)}{12s\,\sigma_*\,c}=\frac{(\sinh^2\sigma_*-1)(\sinh^2\sigma_*-4)(\sinh^2\sigma_*+1)}{12\sinh^2\sigma_*\,\sigma_*\cosh\sigma_*},
\end{equation*}
which is \eqref{eq:xi1-formula-explicit}.

\subsection*{A.3. Expansion of $H(a)=\sinh r(a)/K(a)$ and computation of $C_0$}

From $\sinh r(a)=B(s_0)^2/\sqrt{B(s_0)^2-K(a)^2}$ (Lemma \ref{lem:coth-identity}). We compute the numerator and denominator separately.

\emph{Numerator.} $B(s_0)^2=\rho^2(1+\xi_0^2)+\rho^4(\xi_0^4/3+2\xi_0^2)+O(\rho^6)$. Substituting $\xi_0=\sinh\sigma_*+\xi_1\rho^2+O(\rho^4)$ and isolating powers of $\rho^2$,
\[
B(s_0)^2=\rho^2\cosh^2\sigma_*+\rho^4\beta_1+O(\rho^6),\qquad \beta_1:=2\sinh\sigma_*\,\xi_1+\sinh^4\sigma_*/3+2\sinh^2\sigma_*.
\]

\emph{Denominator.} $B(s_0)^2-K(a)^2=\rho^2(1+\xi_0^2)+\rho^4(\xi_0^4/3+2\xi_0^2)-\rho^2(1+\rho^2)+O(\rho^6)=\rho^2\sinh^2\sigma_*+\rho^4\alpha_1\cdot 2\sinh^2\sigma_*+O(\rho^6)$ in the form $\sqrt{B^2-K^2}=\rho\sinh\sigma_*\bigl[1+\rho^2\alpha_1+O(\rho^4)\bigr]$, where
\[
\alpha_1:=\frac{2\sinh\sigma_*\,\xi_1+\sinh^4\sigma_*/3+2\sinh^2\sigma_*-1}{2\sinh^2\sigma_*}=\frac{\xi_1}{\sinh\sigma_*}+\frac{\sinh^2\sigma_*}{6}+1-\frac{1}{2\sinh^2\sigma_*}.
\]

\emph{Ratio.} $\sinh r=B^2/\sqrt{B^2-K^2}=\rho\cdot\cosh^2\sigma_*/\sinh\sigma_*\cdot[1+\rho^2(\beta_1/\cosh^2\sigma_*-\alpha_1)+O(\rho^4)]$. Recall $\cosh^2\sigma_*/\sinh\sigma_*=\sigma_*\cosh\sigma_*$ (from $\sigma_*\sinh\sigma_*=\cosh\sigma_*$), and $K(a)=\rho(1+\rho^2/2+O(\rho^4))$. Hence
\[
H(a)=\frac{\sinh r(a)}{K(a)}=\sigma_*\cosh\sigma_*\bigl[1+\rho^2(\beta_1/\cosh^2\sigma_*-\alpha_1-1/2)+O(\rho^4)\bigr].
\]
By definition $H'((1/2)^+)=\lim_{a\to(1/2)^+}[H(a)-\sigma_*\cosh\sigma_*]/(a-1/2)$, and using $a-1/2=\rho^2$,
\[
C_0=H'((1/2)^+)=\sigma_*\cosh\sigma_*\,(\beta_1/\cosh^2\sigma_*-\alpha_1-1/2).
\]

\subsection*{A.4. Algebraic reduction in $s=\sinh^2\sigma_*$}

We reduce $\beta_1/\cosh^2\sigma_*-\alpha_1-1/2$ to a single rational function of $s$. Using $\cosh^2\sigma_*=s+1$ and $\sigma_*\sinh\sigma_*=\cosh\sigma_*$ (whence $\sinh\sigma_*\cdot\sigma_*\cosh\sigma_*=\cosh^2\sigma_*=s+1$):

\emph{Coefficient of $\xi_1$.} $\beta_1$ contributes $2\sinh\sigma_*\,\xi_1/\cosh^2\sigma_*=2\sinh\sigma_*\,\xi_1/(s+1)$. $\alpha_1$ contributes $\xi_1/\sinh\sigma_*$. The difference is
\[
\frac{2\sinh\sigma_*}{s+1}-\frac{1}{\sinh\sigma_*}=\frac{2\sinh^2\sigma_*-(s+1)}{(s+1)\sinh\sigma_*}=\frac{2s-(s+1)}{(s+1)\sinh\sigma_*}=\frac{s-1}{(s+1)\sinh\sigma_*}.
\]
Multiplied by $\xi_1$ from \eqref{eq:xi1-formula}, and using $\sinh\sigma_*\cdot\sigma_*\cosh\sigma_*=s+1$,
\[
\xi_1\cdot\frac{s-1}{(s+1)\sinh\sigma_*}=\frac{(s-1)^2(s-4)(s+1)}{12s\cdot(s+1)\sinh\sigma_*\cdot\sigma_*\cosh\sigma_*}=\frac{(s-1)^2(s-4)}{12s(s+1)}.
\]

\emph{Non-$\xi_1$ part.} The contributions are
\[
\frac{\sinh^4\sigma_*/3+2\sinh^2\sigma_*}{s+1}-\frac{\sinh^2\sigma_*}{6}-1+\frac{1}{2\sinh^2\sigma_*}=\frac{s^2/3+2s}{s+1}-\frac{s}{6}-1+\frac{1}{2s}.
\]
Reducing to the common denominator $6s(s+1)$, term by term,
\[
\frac{s^2/3+2s}{s+1}=\frac{2s^2(s+6)}{6s(s+1)},\qquad
-\frac{s}{6}=\frac{-s^2(s+1)}{6s(s+1)},\qquad
-1=\frac{-6s(s+1)}{6s(s+1)},\qquad
\frac{1}{2s}=\frac{3(s+1)}{6s(s+1)},
\]
whence
\[
\frac{s^2/3+2s}{s+1}-\frac{s}{6}-1+\frac{1}{2s}=\frac{2s^2(s+6)-s^2(s+1)-6s(s+1)+3(s+1)}{6s(s+1)}.
\]
Expanding the numerator: $2s^3+12s^2-s^3-s^2-6s^2-6s+3s+3=s^3+5s^2-3s+3$. Therefore the non-$\xi_1$ part equals $(s^3+5s^2-3s+3)/[6s(s+1)]=2(s^3+5s^2-3s+3)/[12s(s+1)]$.

\emph{Subtract $1/2=6s(s+1)/[12s(s+1)]$}:
\[
\beta_1/\cosh^2\sigma_*-\alpha_1-\frac{1}{2}=\frac{(s-1)^2(s-4)+2(s^3+5s^2-3s+3)-6s(s+1)}{12s(s+1)}.
\]
Expanding $(s-1)^2(s-4)=(s^2-2s+1)(s-4)=s^3-4s^2-2s^2+8s+s-4=s^3-6s^2+9s-4$. Adding $2(s^3+5s^2-3s+3)=2s^3+10s^2-6s+6$ and $-6s(s+1)=-6s^2-6s$:
\[
s^3-6s^2+9s-4+2s^3+10s^2-6s+6-6s^2-6s=3s^3-2s^2-3s+2.
\]
Factoring: at $s=1$, $3-2-3+2=0$, so $(s-1)$ is a factor. Dividing, $3s^3-2s^2-3s+2=(s-1)(3s^2+s-2)$. Factoring $3s^2+s-2=(3s-2)(s+1)$. Therefore
\[
\beta_1/\cosh^2\sigma_*-\alpha_1-\frac{1}{2}=\frac{(s-1)(3s-2)(s+1)}{12s(s+1)}=\frac{(s-1)(3s-2)}{12s}.
\]

\subsection*{A.5. Final formula and positivity}

Substituting in $C_0=\sigma_*\cosh\sigma_*\cdot(\beta_1/\cosh^2\sigma_*-\alpha_1-1/2)$, and recalling $s=\sinh^2\sigma_*$,
\begin{equation*}
\;C_0=\frac{\sigma_*\cosh\sigma_*\,(\sinh^2\sigma_*-1)\,(3\sinh^2\sigma_*-2)}{12\sinh^2\sigma_*}.\;
\end{equation*}

For the positivity $C_0>0$, both factors $(s-1)$ and $(3s-2)$ must be positive (the other terms are obviously positive). We need $s=\sinh^2\sigma_*>1$, since $s>1$ automatically implies $s>2/3$. The condition $\sinh^2\sigma_*>1$ is equivalent to $\sigma_*>\mathrm{arcsinh}(1)=\log(1+\sqrt{2})$.

To verify this analytically, consider $F(\sigma):=\sigma-\coth\sigma$. Since $F'(\sigma)=1+\mathrm{csch}^2\sigma>0$, $F$ is strictly increasing on $(0,\infty)$, and $F(\sigma_*)=0$ by definition. We compute
\[
F(\log(1+\sqrt{2}))=\log(1+\sqrt{2})-\coth(\log(1+\sqrt{2})).
\]
At $\sigma=\log(1+\sqrt{2})=\mathrm{arcsinh}(1)$, $\sinh\sigma=1$ and $\cosh\sigma=\sqrt{2}$, so $\coth\sigma=\sqrt{2}/1=\sqrt{2}$. Therefore $F(\log(1+\sqrt{2}))=\log(1+\sqrt{2})-\sqrt{2}$. By the strict inequality $\log(1+x)<x$ for $x>0$ (a classical consequence of the concavity of $\log$), with $x=\sqrt{2}$, $\log(1+\sqrt{2})<\sqrt{2}$, so $F(\log(1+\sqrt{2}))<0$. By strict monotonicity of $F$ and $F(\sigma_*)=0$, $\sigma_*>\log(1+\sqrt{2})$, hence $\sinh\sigma_*>1$, i.e. $s>1$. This establishes $C_0>0$ analytically.

\section*{Declarations}

\subsection*{Conflict of interest}

The author declares no conflict of interest.

\subsection*{Data availability}

Data sharing is not applicable to this article as no datasets were generated or analysed during the current study.

\end{document}